\documentclass[a4paper,12pt,oneside,reqno]{amsart}
\usepackage{hyperref}
\usepackage[headinclude,DIV13]{typearea}
\areaset{15.1cm}{25.0cm}
\usepackage{amsfonts,amssymb,amsmath,amsthm,bbm}
\usepackage{cancel} 
\usepackage[latin1] {inputenc}
\usepackage{graphicx, psfrag}

\newtheorem{theorem}{Theorem}[section]
\newtheorem{lemma}[theorem]{Lemma}
\newtheorem{proposition}[theorem]{Proposition}

\theoremstyle{definition}

\theoremstyle{remark}
\newtheorem*{remark}{Remark}

\numberwithin{equation}{section}

\newcommand{\be}{\begin{equation}}
\newcommand{\ee}{\end{equation}}

\newcommand{\bea}{\begin{eqnarray}}
 \newcommand{\eea}{\end{eqnarray}}
\def\TH(#1){\label{#1}}\def\thv(#1){\ref{#1}}
\def\Eq(#1){\label{#1}}\def\eqv(#1){(\ref{#1})}

\def \1{\mathbbm{1}}
\def\wt {\widetilde}
\def\wh{\widehat}

\def\dist{\mathop{\rm dist}\nolimits}

\def\d{\mathrm{d}}

\def\<{\langle}
\def\>{\rangle}
\def\a{\alpha}
\def\b{\beta}
\def\e{\epsilon}

\def\g{\gamma}

\def\l{\lambda}

\def\t{\tau}

\def\R{{\Bbb R}}  
\def\N{{\Bbb N}}  
\def\P{{\Bbb P}}  
\def\Z{{\Bbb Z}}  
\def\Q{{\Bbb Q}}  
\def\E{{\Bbb E}}  

\let\cal=\mathcal

\def\EE{{\cal E}}
\def\FF{{\cal F}}

\def\LL{{\cal L}}

\def\PP{{\cal P}}

\def\VV{{\cal V}}

\def\VV{{\cal V}}

 \def \b {{\beta}}
 \def \e {{\varepsilon}}

 \def \t {{\tau}}

 \def \g {{\gamma}}
 \def \l {{\lambda}}
 \def \d {{\delta}}
 \def \a {{\alpha}}

\DeclareMathOperator{\asl}{Asl}

\begin{document}

\title[Clock processes on infinite graphs]
{Convergence of clock processes on infinite graphs and aging in Bouchaud's asymmetric trap model on $\Z^d$}
\author[V. Gayrard]{V\'eronique Gayrard}
 \address{V. Gayrard\\I2M, Aix-Marseille Universit\'e\\
39, rue F. Joliot Curie\\13453 Marseille cedex 13, France}
\email{veronique.gayrard@math.cnrs.fr, veronique@gayrard.net}\author[A. \v Svejda]{Ad\'ela \v Svejda}
 \address{A. \v Svejda\\Faculty of Industrial Engineering and Management\\
Technion\\ 32000 Haifa, Israel}
\email{asvejda@technion.ac.il}

\date{\today}

\begin{abstract}
Using a method developed by Durrett and Resnick \cite{DR78} we establish general criteria for the convergence of properly rescaled clock processes of random dynamics in random environments on infinite graphs. This complements the results of \cite{G10a}, \cite{BG10}, and \cite{BGS12}: put together these results provide a unified framework for proving convergence of clock processes. 
As a first application we prove that Bouchaud's asymmetric trap model on $\Z^d$ exhibits a normal aging behavior for all $d\geq 2$. Namely, we show that certain two-time correlation functions, among which the classical probability to find the process at the same site at two time points, converge, as the age of the process diverges, to the distribution function of the arcsine law. As a byproduct we prove that the fractional kinetics process ages.
\end{abstract}

\thanks{ V.G.~thanks the IAM, Bonn University, the Hausdorff Center, the SFB 611, and the SFB 1060 for kind hospitality. A.S.~thanks the SFB 1060 and the Hausdorff Center for financial support.}
 \maketitle


\section{Introduction and main results}
 \label{S1}
This introduction is made of three parts. In the first we describe the general setting and formulate the problems of interest.  We state our abstract results in Section \ref{S1.2}. Section \ref{S1.3} contains the application to Bouchaud's asymmetric trap model.

 \subsection{Markov jump processes in random environments and clock processes}\label{S1.1}
 
Let $G=(\VV,\LL)$ be a loop-free graph.  The random environment is a collection of random variables, $\{\tau(x), x\in \VV\},$ defined on a common probability space $(\Omega,\FF, \P)$, that are only assumed to be positive. On $\VV$ we consider a continuous time Markov jump process, $X$, with initial distribution $\mu$, whose jump rates $(\lambda(x,y))_{x,y\in \VV}$ satisfy
\be\label{1.1}
\tau(x) \lambda(x,y)=\tau(y) \lambda(y,x), \quad \forall (x,y)\in \LL, \ x\neq y.
\ee
This implies that $X$ is reversible with respect to the random measure on $\VV$ that assigns to $x\in\VV$ the mass $\tau(x)$. 

Clock processes of $X$ have recently been at the center of attention in connection with the study of aging and/or anomalous diffusions. Relevant questions on both topics can be formulated by writing $X$ as a time change of another Markov process $J$,
\be\label{1.3a}
X(t)=J(S^{\leftarrow}(t)), \quad \quad t\geq0,
\ee
and making judicious choices of $S$, the so-called \emph{clock process}. Here $S^{\leftarrow}$ denotes the generalized right continuous inverse of $S$. When studying aging the focus usually is on the total time elapsed along trajectories of $X$ of a given length. This is given by the  \emph{discrete time clock process}
\be\label{1.4}
S(k)\equiv \sum_{i=0}^{k-1}\lambda^{-1}(J(i)) e_i, \quad \quad k\geq 1,
\ee
where $J$ is the discrete time chain with transition probabilities
\be\label{1.3b}
p(x,y)\equiv\lambda(x,y)/\lambda(x) \quad \mbox{ if }(x,y)\in \LL,
\ee
and zero else,
\be\label{1.3}
\textstyle{\lambda(x)\equiv \sum_{y: (x,y)\in \LL} \lambda(x,y), \quad x\in \VV,}
\ee
is the inverse of the mean holding time of $X$ at $x$, and $\{e_i, \ i=0,1,2, \ldots\}$ is an independent collection of i.i.d.~mean one exponential random variables. Knowledge of the large $k$ behavior of $S$ combined with relation \eqref{1.3a} then allows to deduce information on the  long time behavior of the two-time correlation functions that are used to quantify aging in theoretical physics. When interested in scaling limits one looks at \eqref{1.3a} from a different angle. One aims at expressing the process $X$ as a time change of another continuous time process, $J$, for which the usual functional limit theorem holds. One is then naturally led to study the \emph{continuous time clock process}
\be\label{1.6}
S(t)\equiv \int_0^t\lambda^{-1}(J(s))\wt \lambda(J(s)) ds, \quad \quad t\geq 0,
\ee
where $\wt \lambda(x)$ denotes the inverse of the mean holding time  of $J$ at $x$.

It emerged from the bulk of works carried out in the past decade that the occurrence of stable subordinators as the limit of properly rescaled clock processes provides a basic mechanism for both aging and anomalous diffusive behaviors to set in in two main types of models. The first are phenomenological models -- the so-called  trap models of Bouchaud \emph{et al.} \cite{Bou92,BD95,RMB00,RMBPaper}. Introduced in theoretical physics to account for the phenomenon of aging then newly discovered  in the physics of spin glasses, these are simple Markov jump processes that describe the dynamics of spin glasses on long time scales in terms of activated barrier crossing in landscapes made of random 'traps'. Another class of models stems from looking at the actual dynamics of microscopic spin glasses. Interesting such dynamics are Glauber dynamics on state spaces $\VV_n=\{-1,1\}^n$ reversible with respect to the Gibbs measures associated to random Hamiltonians of mean-field spin glasses, such as the REM and $p$-spin SK models.

 The first connection between microscopic dynamics of spin systems and trap models was made  in \cite{BBG1}, \cite{BBG2}, \cite{BBG02}  for a variant of the Glauber dynamics  of the REM (the random hopping dynamics, hereafter RHD) on time scales close to equilibrium, and extended  in \cite{BC06b} to shorter time scales (but still exponential in $n$). There it is shown that the properly rescaled  discrete time clock process \eqref{1.4} converges  $\P$-a.s.~to a stable subordinator. These results were partially extended to the $p$-spin SK models in \cite{BBC08}, for all $p\geq 3$ and in a range of exponentially long time scales, whereas it was shown in \cite{BaGun11} that on sub-exponential times scales the  clock process no longer converges to a stable subordinator but to an extremal process, and this for all $p\geq 2$; both these results were obtained in $\P$-law only. 

The field gained new momentum with the paper \cite{G10a}. Based on a method developed by Durrett and Resnick \cite{DR78} in the late 70's to prove functional limit theorems for dependent random variables, a fresh view on the convergence of clock processes in random environment was proposed and general criteria for convergence of clock processes to subordinators were given. This allowed to improve all earlier results on aging of the RHD of the REM \cite{G10b} and $p$-spin SK models \cite{BG10},  \cite{BGS12},   yielding $\P$-a.s.~results for all $p>4$ (in $\P$-probability else), and paved the way for new advances \cite{Metro}. In all the papers mentioned above clock processes are used to control suitable time-time correlation functions, and aging is deduced.

Meanwhile, in a  different line of research, an important class of trap models on $\Z^d$ known as  Bouchaud's asymmetric trap model (hereafter BATM) \cite{RMB00,RMBPaper} was fully investigated both from the view point of aging 
and scaling limits, in different dimensions and for different values of the asymmetry parameter $\theta\in[0,1]$ (see Section \thv(S1.3) for the definition of BATM). In what follows we call BTM the 'symmetric' version of the model, obtained by setting $\theta=0$. Aging was first proved in the seminal paper \cite{FIN02} for BTM on  $\Z$, and extended to BATM on $\Z$ in \cite{BC05}. Emphasis was first given to the  discrete clock process of  BTM in \cite{BCM1}, for $d=2$, and later in \cite{BC07}, for $d\geq 2$. In both these papers it is proved that for suitable scalings, the clock process converges to a stable subordinator. This is used in \cite{BCM1} to study aging via correlation functions, and in \cite{BC07}
to prove convergence of the properly normalized BTM to the so-called Fractional-Kinetics process (see \eqref{1.39aa}). 
More recently, \cite{FonMat10} established aging for transient variants of BTM on $\Z^d$ for all $d\geq 1$.
The continuous time clock process \eqv(1.6) came into play later, in the study of BATM on $\Z^d$, $d\geq 2$,
\cite{BaCe11},  
\cite{BarZhe10}, 
\cite{Ce11},  
\cite{Mou11}. 
There, $J$ is chosen as the so-called variable speed random walk (hereafter VSRW), that is to say, the continuous time Markov chain with rates $\wt \lambda(x,y)=\tau(x)\lambda(x,y)$. This is a central object in the literature on random conductance models and its scaling limit is well-understood (for the most recent and strongest results see \cite{BaDe10} and \cite{ABDH}). Convergence of the rescaled clock process to a stable subordinator is established in
\cite{BaCe11}, 
\cite{Ce11}, 
\cite{Mou11}
under various assumptions on $d$ and using various techniques (see Section  \thv(S1.3) for a detailed discussion). Consequences for the scaling limit of BATM are drawn but not, to our knowledge,  for correlation functions.

The question naturally arises as to whether the method put forward in \cite{G10a} could allow to make progress on this issue. How to implement it however is not straightforward. The formulation of the general, abstract  criteria  for convergence of clock processes of \cite{G10a} and \cite{BG10} was geared to the setting of sequences of finite graphs suited for dealing with mean field spin glasses. Furthermore, in all applications, explicit use is made of the fact that the discrete time chain $J$ in \eqref{1.4} admits an invariant probability measure and is, moreover, sufficiently fast mixing. In contrast, the arena of BATM on $\Z^d$ is that of dynamics on infinite graphs that do not admit of an invariant probability measure.

In the present paper we address  this question in the general setting of Markov jump processes on infinite graphs that satisfy \eqv(1.1). We formulate abstract sufficient conditions for properly rescaled clock processes of the form \eqv(1.3a) (both continuous or discrete) to converge to stable subordinators. (It will be seen that the r\^ ole of the invariant measure is now played by a certain 'mean empirical measure'.) We then apply this result to BATM  for all $d\geq 2$. This,  in turn, enables us to control several (classical or natural) correlation functions through which the aging behavior of the process can be characterized, and prove the existence of 'normal aging'.


\subsection{Main results}\label{S1.2}

In this paper, we consider continuous and discrete time clock processes in a unified setting and introduce notations that allow to handle them simultaneously. From now on let $J$ be either a continuous or discrete time Markov chain having transition probabilities \eqref{1.3b} and initial distribution $\mu$. Continuous time chains are assumed to be non-explosive (see Chapter 3.5 in \cite{Norris}). To a Markov chain $J$ we associate a process $  \ell=\{  \ell_t(x), x\in \VV, t\geq 0\}$ and a sequence $\wt \Lambda=\{\wt\l(x), x\in \VV\}$ defined as follows. 
When $J$ is a continuous time Markov chain $\wt \lambda(x)$ is the holding time parameter of $J$ at $x$ and $  \ell_t(x)$ is the local time
\be\label{1.7}
\textstyle{  \ell_t(x)\equiv \int_0^t \1_{J(s)=x}\  ds},
\ee
namely, the total time spent by $J$ at $x$ in the time interval $[0,t]$.
When $J$ is a discrete time Markov chain we set $\wt \lambda(x)\equiv 1$. In this case $\ell_t(x)$ is defined through
\be\label{1.8}
\textstyle{ {\ell}_t (x)\equiv
\sum_{i=0}^{\lfloor t\rfloor} e_{i}\1_{J(i)=x}, }
\ee
where $\{e_i, i=0,1,2,\ldots\}$ is a collection of i.i.d.~mean one exponential random variables independent of everything else. Observe that this is the local time of a continuous time Markov chain whose mean holding times are identically one.
The clock process is then given by
\be\label{1.9}
S^J(t)\equiv \sum_{x\in \VV}  {\ell}_t(x) \wt \lambda(x) \lambda^{-1}(x), \quad t\geq 0.
\ee
Notice that this definition is consistent with \eqref{1.4} and \eqref{1.6}. In particular, one can check that the relation \eqref{1.3a} between $S^J, J$, and $X$ is satisfied. In the sequel we write $P_{\mu}$ for the law of $J$ and $\PP_{\mu}$ for the law of $X$ with initial distribution $\mu$. We also write $P_x\equiv P_{\d_x}$ and $\PP_x\equiv\PP_{\d_x}$. We denote expectations with respect to $\PP_\mu$ and $P_\mu$ respectively by $\EE_\mu$ and $E_\mu$.
Of course these are random measures on $(\Omega,\FF,\P)$. 

Let $a_n$ and $c_n$ be non-decreasing sequences. We think of $c_n$ as the time scale on which  the process $X$ is observed, and of $a_n$ as an auxiliary time scale for the Markov chain $J$. The question of interest now becomes to find conditions for the re-scaled sequence  
\be\label{1.10}
S_n^J(t)\equiv c_n^{-1}\sum_{x\in \VV} \widetilde{\lambda}(x)\lambda^{-1}(x) {\ell}_{\lfloor a_n t\rfloor} (x), \quad t\geq 0,
\ee
to converge weakly, as a sequence of random elements in the space $D[0,\infty)$ of c\`adl\`ag functions on $[0,\infty)$, $\P$-almost surely in the random environment.

To answer this question we use a method developed by Durrett and Resnick \cite{DR78} that yields criteria for sums of correlated random variables to converge that are particularly useful when applied to clock processes. Following  \cite{BG10}, we will not apply it to  $S_n^J$ directly but rather to a `blocked' version of $S_n^J$. Namely, we introduce a new sequence, $\theta_n$, chosen such that $\theta_n\ll a_n$, and use it to define the block variables
\be\label{1.12}
Z_{n,k}^J\equiv c_n^{-1}\sum_{x\in \VV} \widetilde{\lambda}(x)\lambda^{-1}(x)  \left(  \ell_{\theta_nk}(x)-  \ell_{\theta_n( k-1)}(x)\right),\quad k\geq 1.
\ee
The reason for this is that in many examples of interest, we do know that the jumps of the limiting clock process do not come from isolated jumps of $S_n^J$ but from block variables, either because of strong local spatial correlations of the random environment (as in spin glasses) or because of strong local temporal correlations of the process $J$ (as e.g.~in BTM on $\Z^2$) or by a conjunction of these reasons.
Set $k_n(t)\equiv \lfloor \lfloor a_n t\rfloor / \theta_n \rfloor $. 
The blocked clock process $S_n^{J,b}$ is then defined by
\be\label{1.13}
S^{J,b}_n(t)\equiv \sum_{k=0}^{k_n(t)-1}Z_{n,k+1}^J+ Z_{n,0}^J, \quad t\geq 0. 
\ee
where the sum over $k$ is zero whenever $k_n(t)-1<0$, and 
$
Z_{n,0}^J\equiv\sum_{x\in \VV} \widetilde{\lambda}(x)\lambda^{-1}(x) {\ell}_{0} (x)
$
(note that this is non-zero for discrete time chains $J$ only).

The convergence criteria we obtain bear on a small number of quantities that we now introduce. For $x\in \VV$ and $u>0$ let 
\be\label{1.14}
Q_n^u(x)\equiv\PP_{x}\left(Z_{n,1}^J>u \right)
\ee
be the tail distribution of $Z_{n,1}^J$, starting in $x$. For each fixed $t>0$, we construct a probability measure on $\VV$ through
\be\label{1.15}
\pi_n^t(x)\equiv E_{\mu}\Bigl((k_n(t))^{-1} \sum_{k=1}^{k_n(t)-1} \1_{J(k\theta_n)=x}\Bigr),\quad x\in\VV.
\ee
This is the empirical measure induced by the sequence $\{J(k\theta_n), k=1, \ldots k_n(t)-1\}$, averaged over $P_{\mu}$. 
Note that $Q_n^u$ and $\pi_n^t$ are not random in the chain $J$. 
Using these quantities, we define 
\be\label{1.16}
\nu_n^t(u,\infty)\equiv k_n(t)  \sum_{x\in \VV}\pi_n^t(x)Q_n^u(x) , 
\ee
and
\be\label{1.17}
\sigma^t_n(u,\infty)\equiv k_n(t)\sum_{x\in \VV}\pi_n^t(x) (Q_n^u(x))^2.
\ee
We are now ready to introduce the conditions of our main theorem. They are stated for given sequences $a_n,c_n,\theta_n$, a given initial distribution $\mu$, and fixed $\omega\in\Omega$.

{\bf (A-0)}
For all $u>0$ 
\be\label{a0}
\lim_{n\rightarrow\infty}\PP_{\mu}\left(Z_{n,1}^J+ Z_{n,0}^J> u\right)=0.
\ee

{\bf (A-1)}
For all $t>0$ there exists $c<\infty$ such that, uniformly in $x\in\VV$,

\be\label{a11}
\lim_{n\rightarrow\infty} \sum_{k=1}^{k_n(t)-1} P_{\mu}(J(k\theta_n)=x)=0,
\ee
and
\be\label{a12}
\lim_{n\rightarrow\infty}\sum_{k=1}^{k_n(t)-1} P_{x}(J(k\theta_n)=x)<c.
\ee

{\bf (A-2)}
There exists a sigma-finite measure $\nu$ on $(0,\infty)$ satisfying $\int_{0}^{\infty} (1\wedge x) d\nu(x)<\infty$ such that for all $t>0$ and all $u>0$ such that $\nu(\{u\})=0$,
\be
\lim_{n\rightarrow\infty}\nu_n^t(u,\infty)=t \nu(u,\infty).\label{a2}
\ee

{\bf (A-3)}
For all $t>0$ and all $u>0$ such that $\nu(\{u\})=0$,
\be
\lim_{n\rightarrow\infty}\sigma^t_n(u,\infty) = 0. \label{a3}
\ee

{\bf (A-4)}
For all $t>0$,
\be
\lim_{\varepsilon\rightarrow 0}\limsup_{n\rightarrow \infty}k_n(t)\sum_{x\in \VV}\pi_n^t(x)\EE_{x} Z^J_{n,1} \1_{\{Z^J_{n,1}\leq \varepsilon\}}=0.\label{a4}
\ee
\begin{theorem}
Assume that there exist sequences $a_n,c_n$, and $\theta_n$ and an initial distribution $\mu$ such that Conditions (A-0)-(A-4) are satisfied $\P$-a.s. Then, $\P$-a.s., as $n\rightarrow\infty$, 
\be\label{theorem111}
S^{J,b}_n\Rightarrow V_{\nu},
\ee 
where $V_{\nu}$ is a subordinator with L\'evy measure $\nu$ and zero drift. Convergence holds weakly in the space $D[0,\infty)$ equipped with Skorohod's $J_1$ topology. 
 \label{theorem1.1}
\end{theorem}
Let us emphasize that our statement is made for $S_n^{J,b}$ and holds in the strong $J_1$ topology, which immediately implies that $S_n^J$ converges to the same limit in the weaker $M_1$ topology. As we just explained (see discussion below \eqv(1.12)), in many models of interest it is not true that $S^J_n$ converges in the $J_1$ topology, but more information than contained in $M_1$ statements can be obtained by introducing a blocked clock process, $S^{J,b}_n$. (This is the case in the $p$-spin SK models \cite{BBC08}, \cite{BG10}, and BTM on $\Z^2$ \cite{BCM1}.)
 Of course in the case of continuous time clock processes, forming blocks is needed in order to make sense of writing  convergence to subordinators statements in the $J_1$ topology.
  Let us finally stress that it is crucial for applications to correlation functions to make statements that are valid in the $J_1$ topology (see the discussion below \eqref{C3}).

Let us comment on Conditions (A-0)-(A-4). Condition (A-0) is a condition on the initial distribution and ensures that the initial increment $S_n^{J,b}(0)$ converges to zero as $n\rightarrow\infty$.
 Conditions (A-2)-(A-4) have the same form as Conditions (A2-1)-(A3-1) in \cite{BG10} where sequences of finite state reversible Markov jump processes are studied. There it is assumed that $J_n$ admits a unique invariant measure, $\pi_n$, and $\theta_n$ is chosen large compared to the mixing time of $J_n$ (see Condition (A1-1)). 
In the present setting, the empirical measure averaged over $P_{\mu}$ replaces the measure $\pi_n$, and Condition (A-1) plays the same r\^ole as Condition (A1-1). More precisely these conditions allow to replace $J$ dependent, respectively $J_n$ dependent, quantities by their average over $P_\mu$, respectively $P_{\pi_n}$.
We conclude this discussion with a lemma that sheds light on the complementarity of Theorem \ref{theorem1.1} and  Theorem 1.3 in \cite{BG10};
indeed the former can only be satisfied by $J$'s that are transient or null-recurrent whereas the latter is designed for positive recurrent $J$'s.
\begin{lemma}\label{lemma1.2}
Let $x\in \VV$. If $x$ is transient then \eqref{a11} and \eqref{a12} are satisfied for any $\theta_n\gg1$, whereas if $x$ is positive recurrent they cannot be satisfied. In particular, (A-1) cannot hold if $J$ admits an invariant probability measure. 
\end{lemma}

When $J$ is random in the random environment the conditions of Theorem \ref{theorem1.1} may not be easy to handle. We now present an additional condition, (B-5), that enables us to replace $\pi_n^t$ in (A-2)-(A-4) by a deterministic probability measure $\overline \pi_n^t$. In this way, all the dependence on the random environment in (A-2)-(A-4) is confined to the $Q_n^u$'s. The following conditions, stated for given sequences $a_n,c_n,\theta_n$, a given initial distribution $\mu$, and for fixed $\omega\in \Omega$, imply the conditions of Theorem \ref{theorem1.1}.

\noindent{\bf (B-5)} 
Set ${\cal A}_n=\{(x,k): \ x\in \VV, \ k\in [k_n(t)-1]\}$, where $[m]\equiv\{0,\ldots, m\}$. There exists a sequence of functions $h_n:\VV\rightarrow[0,1]$ such that for all $t>0$ and all $n\in \N$, the set ${\cal A}_n$ can be decomposed into the disjoint union of two sets, ${\cal A}_n^1$ and ${\cal A}_n^2$, satisfying
\be
\lim_{n\rightarrow\infty}\sup_{(x,k)\in {\cal A}_n^1} \frac{|P_{\mu}(J(k\theta_n)=x))-h_{k\theta_n}(x)|}{h_{k\theta_n}(x)}=0,\label{b51}
\ee
and 
\bea
\hspace{-2mm}&&\hspace{-2mm}\lim_{n\rightarrow\infty}\sum_{(x,k)\in {\cal A}_n^2} \bigl|P_{\mu}(J(k\theta_n)=x)-h_{k\theta_n}(x)\bigr| Q_n^u(x)=0,\label{b52}\\
\hspace{-2mm}&&\hspace{-2mm}\lim_{n\rightarrow\infty}\sum_{(x,k)\in {\cal A}_n^2} \bigl|P_{\mu}(J(k\theta_n)=x)-h_{k\theta_n}(x)\bigr| \EE_{x} Z^J_{n,1} \1_{\{Z^J_{n,1}\leq \varepsilon\}}=0.\label{b53}\quad \quad
\eea
Observe that proving \eqref{b51} corresponds to proving a uniform local central limit theorem for $J$.

For each $t>0$ we define the measure $\overline\pi_n^t$, using $h_n$, through
\be\label{1.26a}
\overline \pi_n^t(x)= (k_n(t))^{-1}\sum_{k=1}^{k_n(t)-1} h_{k\theta_n}(x), \quad x\in \VV.
\ee
By analogy with \eqref{1.16} and \eqref{1.17} we set for $t>0$, $u>0$
\bea\label{1.26}
&&\overline\nu_n^t(u,\infty)\equiv k_n(t)  \sum_{x\in \VV}\overline \pi_n^t(x)Q_n^u(x),\\
&&\overline\sigma^t_n(u,\infty)\equiv k_n(t)  \sum_{x\in \VV}\overline \pi_n^t(x)(Q_n^u(x))^2.\label{1.27}
\eea

The next conditions are nothing but Conditions (A-2)-(A-4) with $\pi_n^t$ replaced by $\overline\pi_n^t$.

{\bf (B-2)}
There exists a sigma-finite measure $\nu$ on $(0,\infty)$ satisfying $\int_{0}^{\infty} (1\wedge x) d\nu(x)<\infty$ such that for all $t>0$ and all $u>0$ such that $\nu(\{u\})=0$,
\be
\lim_{n\rightarrow\infty}\overline\nu_n^t(u,\infty)= t \nu(u,\infty).\label{b2}
\ee

{\bf (B-3)} For all $t>0$ and all $u>0$ such that $\nu(\{u\})=0$,
\be
\lim_{n\rightarrow\infty}\overline\sigma_n^t(u,\infty)= 0. \label{b3}
\ee

{\bf (B-4)} 
For all $t>0$,
\be
\lim_{\varepsilon\rightarrow 0}\limsup_{n\rightarrow \infty}k_n(t)\sum_{x\in \VV}\overline\pi_n^t(x)\EE_{x} Z^J_{n,1} \1_{\{Z^J_{n,1}\leq \varepsilon\}}=0.\label{b4}
\ee

\begin{theorem}
Assume that there exist sequences $a_n,c_n$, and $\theta_n$ and an initial distribution $\mu$ such that Conditions (A-0), (A-1), (B-2)-(B-5) are satisfied $\P$-a.s. Then, $\P$-a.s., 
\be\label{theorem131}
S^{J,b}_n\Rightarrow V_{\nu},
\ee
where $V_{\nu}$ is a subordinator with L\'evy measure $\nu$ and zero drift. Convergence holds weakly in the space $D[0,\infty)$ equipped with Skorohod's $J_1$ topology. 
 \label{theorem1.3}
\end{theorem}

The following lemma is instrumental in verifying the conditions of Theorem \ref{theorem1.1} or Theorem \ref{theorem1.3}. 
\begin{lemma}\label{ergthmzdcor} Let $\nu$ be a sigma finite measure on $(0,\infty)$. Suppose that for given $a_n,c_n,\theta_n, \mu$ and fixed $t>0$, $u>0$ there exists $\Omega^{\tau}(u,t)\subseteq \Omega$ with $\P(\Omega^{\tau}(u,t))=1$ such that, on $\Omega^{\tau}(u,t)$, (A-0)-(A-4), respectively (B-2)-(B-5), are verified. Then, for these sequences and this initial distribution (A-0)-(A-4), respectively (B-2)-(B-5), are satisfied $\P$-a.s.~for all $t>0$, $u>0$.
\end{lemma}


\subsection{Application to Bouchaud's asymmetric trap model (BATM)}\label{S1.3}
We now use Theorem \ref{theorem1.3} to prove aging in Bouchaud's asymmetric trap model on $\Z^d$.
Here $\VV=\Z^d$, $d\geq 2$, $\LL$ is the set of nearest neighbors on $\Z^d$, and $\mu\equiv\delta_0$. The random environment, $\left\{\tau(x), x\in \mathbb{Z}^d\right\}$, is a collection of i.i.d.~random variables, with tail distribution given by
\bea
\P(\tau(0)>u)=
\begin{cases}
C u^{-\alpha} (1+L(u)), &\quad  u\in( \bar c,\infty),\\
1, & \quad u\in(0,\bar c],
\end{cases}
\eea
where $\a\in(0,1)$, $\bar c, C\in (0,\infty)$ are constants, and $L:(0,\infty)\rightarrow \R$ is a function that obeys $L(u)\rightarrow 0$ as $u\rightarrow\infty$. We write $x\sim y$ if $x,y$ are nearest neighbors in $\Z^d$. The jump rates of $X$ depend on a parameter, $\theta\in[0,1]$, and are given by 
\bea\label{1.37}
\lambda(x,y)=
\textstyle{(\t(x))^{-1}(\tau(x)\t(y))^{\theta}}, \quad \mbox{ if } x\sim y,
\eea
and zero else. Consider now the VSRW of this model, namely, the continuous time Markov chain, $J$, with jump rates 
\bea\label{1.41}
\wt{\lambda}(x,y)=
\textstyle{(\tau(x)\tau(y))^{\theta}}, \quad \mbox{ if } x\sim y,
\eea
and zero else. Our interest is in the continuous time clock process $S^{J}$ defined (as in \eqref{1.9}) through
\bea\label{1.42}
\textstyle{S^{J}(t)=\sum_{x\in \VV}  {\ell}_t(x) \wt \lambda(x) \lambda^{-1}(x) = \sum_{x\in \VV}  {\ell}_t(x) \tau(x).}
\eea
Our first theorem states convergence of the blocked clock process, $S_n^{J,b}$, for appropriate choices of block lengths, $\theta_n$, in the $J_1$ topology.
\begin{theorem}\label{theorem1.5}
Let $c_n=n$ and take
\bea\label{1.44}
\theta_n &=&n^{\a\gamma_2}\1_{d=2} + (\log n) ^{\gamma_3}\1_{d\geq 3},\\
a_n&=&
n^{\alpha}(\log n)^{1-\alpha}\1_{d=2} +n^{\alpha} \1_{d\geq 3},\label{1.45}
\eea
where $\gamma_2\in(0,1/6), \gamma_3>12/(1-\a)$. Then, $\P$-a.s., as $n\rightarrow\infty$, 
\be
S^{J,b}_n\Rightarrow V_{\a},
\ee 
where $V_{\a}$ is a subordinator with L\'evy measure $\nu(u,\infty)= {\cal K} u^{-\alpha}$, for ${\cal K}={\cal K}(d,\alpha,\theta)>0$, and zero drift. Convergence holds weakly in the space $D[0,\infty)$ equipped with Skorohod's $J_1$ topology. 
\end{theorem}
All earlier papers dealing with the clock process $S_n^{J}$ focused on proving scaling limits for BATM. It was first proved in \cite{BaCe11} that the properly rescaled process converges to a fractional kinetics process for $d\geq 3$. This was extended to $d=2$ in \cite{Ce11}. Shortly after \cite{BaCe11}, \cite{Mou11} gave an alternative proof of this result for $d\geq 5$. The method of \cite{BaCe11,Ce11} relies on blocking with block length $\theta_n=\e n^{\a}$. In contrast, \cite{Mou11} proposed a method of proof that does not use blocking. Both approaches resulted in $M_1$ convergence for $S_n^{J}$. (Note that because $S^{J}$ is a continuous time clock process, the method of \cite{Mou11} does not allow to obtain $J_1$ convergence statements for $S^{J}$.)

Let us comment on our choices of $\theta_n$. Because $J$ is recurrent  when $d=2$ and  transient otherwise
two cases  must be distinguished. When $d=2$ we first remark that (A-1) would be satisfied for any $\theta_n\gg\log n$. There, our constraint on $\theta_n$ comes from (A-2)-(A-4). In the course of verifying these conditions one sees that $\theta_n$ must be chosen in such a way that the mean values of local times in the time interval $[0,\theta_n]$ are of the order of  $\log n$.
Since these mean values are of order $\log\theta_n$ we take $\theta_n=n^{\a\gamma_2}$. 
When $d\geq 3$ Conditions (A-1)-(A-4) can a priori be verified for any diverging $\theta_n$. Here, the constraint \eqv(1.44) on  $\theta_n$ comes from using precise heat kernel estimates for $J$, taken from \cite{BaDe10}, which are only valid for large enough time intervals (of course this was already the case in $d=2$).

We now present our results on aging. Theorem \ref{theorem1.5} allows to control several correlation functions, which we now introduce. The first is the classical correlation function 
\be\label{C1}
{\cal C}_s^1(1,\rho)\equiv \textstyle{\PP\left(X(s)=X(s(1+\rho))\right),\quad s>0,\ \rho>0},
\ee
which is the probability that at the beginning and the end of the time interval $(s,s(1+\rho))$ the process is in the same site. The second correlation function is the probability that during a certain time interval the process stays inside a ball of a certain radius. Specifically, writing $\theta_s\equiv\theta_{\lfloor s\rfloor}$,
\be
\label{C2}
{\cal C}_s^2(1,\rho)\equiv \textstyle{\PP\left(\max_{v\in (s,s(1+\rho))}|X(s)-X(v)|\leq (\theta_{s}\log\theta_{s})^{1/2}\right),\  s> 0,\ \rho> 0}.
\ee
Notice that ${\cal C}^1_s$ and ${\cal C}^2_s$ clearly contain different information. Our third and last correlation function combines them both. For $s>0$, $\rho>0$ we define
\be\label{C3}
{\cal C}_s^3(1,\rho)\equiv \PP\Bigl(X(s)=X(s(1+\rho)),\max_{v\in (s,s(1+\rho))}|X(s)-X(v)|\leq (\theta_{s}\log\theta_{s})^{1/2} \Bigr).
\ee
The proof of the next theorem relies on a well-known scheme, that goes back to 
\cite{BC06b}, that links aging to the arcsine law for subordinators through the convergence of the clock process $S_n^{J,b}$. In this scheme, one aims at deducing convergence of correlation functions from convergence of the overshoot function of the blocked clock process, $\chi(S_n^{J,b})$. (The overshoot function of $Y\in D[0,\infty)$ is given by  $\chi_u(Y)=Y({\cal L}_u(Y))-u$, $u>0$, where ${\cal L}_u$ is the first passage time (see \eqref{os} for a definition).~)
For this, one needs $\chi$ to be continuous with respect to the topology in which $S_n^{J,b}$ converges, and because this is only true in the $J_1$ topology, it is all important that  $S_n^{J,b}$ converges in that topology.

Let $\asl_\a$ denote the distribution function of the generalized arcsine law, 
\be\label{5.1}
\textstyle{\asl_\a(u)\equiv\frac{\sin \a \pi}{\pi} \int_0^{u} (1-x)^{-\a} x^{\a-1} dx, \quad u\in[0,1].}
\ee
\begin{theorem}\label{theorem1.4}
Let $d\geq 2$. Under the assumptions of Theorem \ref{theorem1.5}, for $i=1,2,3,$ $\P$-a.s.,
\be\label{1.39}
\lim_{s\rightarrow\infty}{\cal C}^i_s(1,\rho) =  \asl_{\a}(1/(1+\rho)), \quad \rho>0.
\ee
\end{theorem}

As pointed out below Theorem \ref{theorem1.5}, it was proved that the rescaled process 
\be
X_s(t)\equiv a_s^{-1/2}X(s t), \quad t\geq 0,
\ee
 converges to the fractional kinetics process. Observe that the radius of the balls in \eqv(C2) for which Theorem
 \ref{theorem1.4} holds is very small compared to the normalization of $X_s$, namely,
 $(\theta_{s}\log\theta_{s})^{1/2}\ll a_s^{1/2}$.
From this and Theorem \ref{theorem1.4} one readily deduces that the correlation function defined, for $\e>0$, by
\be\label{1.56}
{\cal C}_s^\e(1,\rho)\equiv \textstyle{\PP\left(\max_{v\in (1,1+\rho)}|X_s(1)-X_s(v)|\leq \e\right), \quad s>0, \ \rho> 0},
\ee
converges to the arcsine distribution function. Interestingly, this, in turn, enables us to deduce results on the aging behavior of the fractional kinetics process itself. This is the content of Theorem \thv(theorem1.6) below.
Recall that the fractional kinetics process is defined by
\be\label{1.39aa}
 Z_{d,\a}(t)\equiv B_d(V_\a^{\leftarrow}(t)), \quad t\geq 0,
\ee
where $B_d$ is a standard Brownian motion on $\R^d$ started in $0$, $V_\a$ is an $\a$-stable subordinator with zero drift that is independent of $B_d$, and $V_\a^{\leftarrow}(t)=\inf\{v: V_\a(v)>t\}$ its generalized right-continuous inverse. 
By analogy with \eqref{1.56} define
\be\label{1.56a}
{\cal C}^\e(1,\rho)\equiv \textstyle{\PP\left(\max_{v\in (1,1+\rho)}|Z_{d,\a}(1)-Z_{d,\a}(v)|\leq \e\right), \quad\rho> 0.}
\ee
\begin{theorem}\label{theorem1.6}
Let $d\geq 2$. Under the assumptions of Theorem \ref{theorem1.5}, $\P$-a.s.,
\be\label{1.39a}
\lim_{\e\rightarrow 0}\lim_{s\rightarrow\infty}{\cal C}_s^\e(1,\rho) =\lim_{\e\rightarrow 0}{\cal C}^\e(1,\rho)= \asl_{\a}(1/(1+\rho)), \quad \rho>0.
\ee
\end{theorem}
\begin{remark}
As a final remark notice that our results are only valid for $d\geq 2$. It is known that the situation in $d=1$ is completely different, see \cite{FIN02}, \cite{BC05}. The clock process converges to the integral of the local time of a Brownian motion on $\R$ with respect to the so-called random speed measure -- a scaling limit of the random environment -- and the scaling limit of $X$ is a singular diffusion on $\R$; see e.g.~\cite{BC07} and \cite{Ce11} for further discussions.
\end{remark}

The remainder of the paper is structured as follows. Section \ref{S2} contains the proof of Theorem 
\ref{theorem1.1} and Theorem \ref{theorem1.3}. In Section \ref{S3} we collect preparatory results for the proof of Theorem \ref{theorem1.5}. The latter is carried out in Section \ref{S4}. 
Finally, Section \ref{S5} contains the proofs of Theorem \ref{theorem1.4} and Theorem \ref{theorem1.6}. Two lemmata are proven in the Appendix.

\textbf{Acknowledgement.} We thank Pierre Mathieu for pointing out that the proof of \eqref{4.28} in an earlier version was incomplete. We are grateful to an anonymous referee for identifying an error in our use of Theorem \ref{theorem3.1}. 


 \section{Proof of Theorem \ref{theorem1.1} and Theorem \ref{theorem1.3}}\label{S2}
We now come to the proofs of the abstract theorems of Section \ref{S1}. We first prove Theorem \ref{theorem1.1}. We then show that the conditions of Theorem \ref{theorem1.3} imply those of Theorem \ref{theorem1.1}, thereby proving Theorem \ref{theorem1.3}. Finally, we prove the lemmata of Section \ref{S1}. 

\begin{proof}[Proof of Theorem \ref{theorem1.1}]
As mentioned earlier, the proof is based on a result by Durrett and Resnick \cite{DR78} that gives conditions for partial sum processes of dependent random variables to converge. We use this result in a specialized form suitable for our application that we take from \cite{G10a}, namely Theorem 2.1 p.~7.

Throughout we fix a realization $\omega \in \Omega$ of the random environment but do not make this explicit in the notation. We set
\be
\wh S_n^{J,b}(t)\equiv S_n^{J,b}(t)-\left(Z_{n,1}^J+ Z_{n,0}^J\right), \quad t>0.
\ee
Condition (A-0) ensures that $\wh S_n^{J,b}- S_n^{J,b}$ converges to zero, uniformly. Thus, we must show that under Conditions (A-1)-(A-4)
\be
\wh S_n^{J,b}\Rightarrow V_{\nu}.
\ee
This will follow if we can verify Conditions (D1)-(D3) of Theorem 2.1 in \cite{G10a} for $\wh S_n^{J,b}$.

For this, let $\{\FF_{n,k}, k\geq 0\}$ be an array of sigma algebras, where for $k\geq 0$, $\FF_{n,k}$ is generated by $\{\ell_s(x),\ s\leq\theta_nk,\ x\in \Z^d\}$. When $J$ is continuous $\FF_{n,k}$ is generated by $\{J(s),\ s\leq \theta_nk\}$, whereas when $J$ is discrete $\FF_{n,k}$ is generated by $\{J(i), e_i,\ i\leq\theta_nk\}$. Note that for $n,k\geq 1$, $Z^J_{n,k}$ is $\FF_{n,k}$ measurable and $\FF_{n,k-1}\subset\FF_{n,k}$. 

We first establish that Condition (D1) is satisfied. For $t>0$ and $u>0$ we define
\be\label{2.2}
\textstyle{\nu_n^{J,t}(u,\infty)\equiv\sum_{k=1}^{k_n\left(t\right)-1} \PP_{\mu}\left(Z^J_{n,k+1} >u \mid \FF_{n,k}\right) .}
\ee
This conditions then states that for all $u>0$ such that $\nu(\{u\})=0$ and all $t>0$ we have in $\PP_{\mu}$-probability
\be\label{2.1}
\textstyle{\lim_{n\rightarrow \infty}\nu_n^{J,t}(u,\infty)= t\nu(u,\infty).}
\ee
By the Markov property, $\nu_n^{J,t}(u,\infty)$ can be rewritten as
\be\label{2.3}
\textstyle{\nu_n^{J,t}(u,\infty)=\sum_{k=1}^{k_n(t)-1}\sum_{x\in\VV} \mathbbm{1}_{J(k\theta_n)=x}Q_n^u(x)=k_n(t)\sum_{x\in \VV} \pi_n^{J,t}(x)Q_n^u(x),}
\ee
where, for $x\in \VV$,
\be
\textstyle{\pi_n^{J,t}(x)\equiv (k_n(t))^{-1}\sum_{k=1}^{k_n(t)-1} \mathbbm{1}_{J(k\theta_n)=x},}
\ee
denotes the empirical measure induced by the sequence $\{J(k\theta_n), k=1, \ldots, k_n(t)-1\}$. Taking the expectation with respect to $\PP_{\mu}$, \eqref{2.3} yields
\be\label{2.4}
\textstyle{\EE_{\mu}\nu_n^{J,t}(u,\infty)=k_n(t)\sum_{x\in\VV}\EE_{\mu}\left(\pi_n^{J,t}(x)\right)Q_n^u(x) =\nu_n^t(u,\infty).}
\ee
Since (A-2) ensures that $\lim_{n\rightarrow\infty}\nu_n^{t}(u,\infty)=t\nu(u,\infty)$ it suffices to prove that 
\be\label{2.5}
\lim_{n\rightarrow\infty} \PP_{\mu}\left(\left|\nu_n^{J,t}(u,\infty)-\nu_n^{t}(u,\infty)\right|>\varepsilon\right)=0, \quad \quad \forall \varepsilon>0,
\ee
i.e.~that we may replace $\pi_n^{J,t}$ by its mean value. We do this by means of a second order Chebyshev inequality. For $x,y\in\VV$ and $k,j\in\N$ write
\bea\label{2.6}
\bar q_{k,j}(x,y)\equiv P_{\mu}(J(k)=x, J(j)=y) \quad \mbox{ and } \quad  q_{k}(x,y)\equiv P_{x}(J(k)=y) ,
\eea
with the convention that $q_{k}(y)\equiv P_{\mu}(J(k)=y)$. Then, on the one hand,
\bea
\textstyle{\EE_{\mu} \left(\nu_n^{J,t}(u,\infty)\right)^2}&=&\textstyle{\sum_{x\in\VV} \left(Q_n^u(x)\right)^2 \left[ k_n(t)\pi_n^t(x)+ 2 \sum_{k=1}^{k_n(t)-2} \sum_{j=k+1}^{k_n(t)-1} \bar q_{k\theta_n,j\theta_n}(x,x)\right]}\nonumber\\
&+&\textstyle{2\sum_{{x,x'\in\VV}\atop{x\neq x'}} Q_n^u(x) Q_n^u(x') \sum_{k=1}^{k_n(t)-2} \sum_{j=k+1}^{k_n(t)-1} \bar q_{k\theta_n,j\theta_n}(x,x')\ ,\label{2.7}}
\eea
and on the other hand, 
\bea
\left(\EE_{\mu}\nu_n^{J,t}(u,\infty)\right)^2&\geq&\textstyle{2\sum_{x\in\VV} \left(Q_n^u(x)\right)^2\sum_{k=1}^{k_n(t)-2} \sum_{j=k+1}^{k_n(t)-1} q_{k\theta_n}(x)q_{j\theta_n}(x)}\nonumber\\
&+&\textstyle{2\sum_{{x,x'\in\VV}\atop{x\neq x'}}Q_n^u(x)Q_n^u(x')\sum_{k=1}^{k_n(t)-2}\sum_{j=k+1}^{k_n(t)-1} q_{k\theta_n}(x)q_{j\theta_n}(x')}.\quad\quad\label{2.8}
\eea
Combining \eqref{2.7} and \eqref{2.8}, we obtain that
\bea\label{2.9}
&&\textstyle{\EE_{\mu}\left(\nu_n^{J,t}(u,\infty)\right)^2 - \left(\EE_{\mu}\nu_n^{J,t}(u,\infty)\right)^2}\nonumber\\
&\leq&\sigma^t_n(u,\infty)\nonumber\\ 
&+&\textstyle{ \sum_{x\in\VV} \left(Q_n^u(x)\right)^2\sum_{k=1}^{k_n(t)-2} \sum_{j=k+1}^{k_n(t)-1}\left[\bar q_{k\theta_n,j\theta_n}(x,x) - q_{k\theta_n}(x) q_{j\theta_n}(x)\right]}\nonumber\\
&+&\textstyle{ \sum_{{x,x'\in\VV}\atop{x\neq x'}} Q_n^u(x)Q_n^u(x')\sum_{k=1}^{k_n(t)-2} \sum_{j=k+1}^{k_n(t)-1} \left[\bar q_{k\theta_n,j\theta_n}(x,x') - q_{k\theta_n}(x) q_{j\theta_n}(x')\right]}\nonumber\\
&\equiv& (I) + (II) + (III).
\eea
By (A-3), $(I)$ tends to zero as $n\rightarrow \infty$. To bound $(II)$, we drop the terms involving 
$q_{k\theta_n}(x)q_{j\theta_n}(x)$, and use the Markov property to write 
\bea
(II)&\leq&\textstyle{\sum_{x\in\VV} \left(Q_n^u(x)\right)^2\sum_{k=1}^{k_n(t)-2} \sum_{j=k+1}^{k_n(t)-1} q_{k\theta_n}(x)P_{x}\left( J((j-k)\theta_n)=x \right)}\nonumber\\
&\leq&\textstyle{ \sum_{x\in\VV} \left(Q_n^u(x)\right)^2   \sum_{k=1}^{k_n(t)-2} q_{k\theta_n}(x)\sum_{j=1}^{k_n(t)-1} P_{x}\left( J(j\theta_n)=x\right)}\nonumber\\
&\leq&\textstyle{ k_n(t)\sum_{x\in\VV} \left(Q_n^u(x)\right)^2   \pi_n^t(x) \sum_{j=1}^{k_n(t)-1} P_{x}\left( J(j\theta_n)=x\right)}\nonumber\\
&\leq&\textstyle{\sigma_n^t(u,\infty)\sup_{x\in\VV} \sum_{j=1}^{k_n(t)-1} P_{x}\left( J(j\theta_n)=x\right)} \ .\label{2.10}
\eea
By (A-1) and (A-3), $(II)\rightarrow 0$ as $n\rightarrow\infty$. 

Let us now show, using (A-1) and (A-2), that also $(III)$ vanishes. Fix $x\in \VV$, $k\geq 1$, and $j\geq k+1$. For every $x'\neq x$ we bound the
 term $Q_n^u(x')$ by $1$.
 Now
\bea\label{2.11}
\textstyle{\sum_{x':\ x'\neq x}\bar q_{k\theta_n,j\theta_n}(x,x') = P_{\mu}(J(k\theta_n)=x, J(j\theta_n)\neq x)\leq P_{\mu}(J(k\theta_n)=x),}
\eea
and
\bea\label{2.12}
\textstyle{\sum_{x':\ x'\neq x}q_{k\theta_n}(x)q_{j\theta_n}(x') =  P_{\mu}(J(k\theta_n)=x)P_{\mu}(J(j\theta_n)\neq x),}
\eea
so that, combining \eqref{2.11} and \eqref{2.12},
\bea\label{2.13}
(III)&\leq&\textstyle{\sum_{x\in\VV} Q_n^u(x)\sum_{k=1}^{k_n(t)-2}P_{\mu}(J(k\theta_n)=x)\sum_{j=1}^{k_n(t)-1} P_{\mu}(J(j\theta_n)= x)}\nonumber\\
&\leq&\textstyle{ \nu_n^t(u,\infty)\sup_{x\in\VV}\sum_{j=1}^{k_n(t)-1} P_{\mu}(J(j\theta_n)=x)}.
\eea
By (A-2), $\nu_n^t(u,\infty)$ converges as $n\rightarrow \infty$ to $t\nu(u,\infty)$, which is a finite number. Thus invoking (A-1), $(II)\rightarrow 0$ as $n\rightarrow\infty$.
Inserting our bounds in \eqref{2.9}, the variance of $\nu_n^{J,t}(u,\infty)$ tends to zero as $n\rightarrow\infty$. The verification of Condition (D1) is complete.

Next we show that Condition (D2) of Theorem 2.1 in \cite{G10a} is satisfied. For $t>0$, $u>0$ we define 
\bea\label{2.14a}
\textstyle{\sigma_n^{J,t}(u,\infty)\equiv \sum_{k=1}^{k_n(t)-1}\left(\PP_\mu\left( Z_{n,k+1}^J>u| \FF_{n,k}\right)\right)^2.}
\eea
This condition then states that for all $u>0$ such that $\nu(\{u\})=0$ and all $t>0$,
\bea\label{2.14}
\textstyle{\lim_{n\rightarrow\infty} \PP_{\mu}\left(\sigma_n^{J,t}(u,\infty)>\varepsilon\right)=0, \quad \quad \forall \varepsilon>0.}
\eea
By the Markov property,
\bea\label{2.15}
\textstyle{\sigma_n^{J,t}(u,\infty)=k_n(t) \sum_{x\in \VV} \pi_n^{J,t}(x)\left(Q_n^u(x)\right)^2 .}
\eea
The expectation of $\sigma_n^{J,t}(u,\infty)$ with respect to $\PP_{\mu}$ is equal to $\sigma^t_n(u,\infty)$ and tends by (A-3) to zero. Thus, Condition (D2) is satisfied. 

It remains to verify Condition (D3) of Theorem 2.1 in \cite{G10a}. It is in particular satisfied if
\bea\label{2.16}
\textstyle{\lim_{\varepsilon\rightarrow 0}\limsup_{n\rightarrow\infty}\sum_{k=1}^{k_n(t)-1} \EE_{\mu} Z^J_{n,k+1}\1_{Z^J_{n,k+1}\leq \varepsilon}=0.}
\eea
By the Markov property the left hand side of \eqref{2.16} is equal to the left hand side of \eqref{a4} and vanishes by (A-4). This proving that Condition (D3) is satisfied. Therefore, the conditions of Theorem 2.1 in \cite{G10a} are verified, and so $\wh S_n^{J,b}\Rightarrow V_{\nu}$ where convergence holds weakly in the space $D[0,\infty)$ equipped with Skorohod's $J_1$ topology and $V_{\nu}$ is a subordinator with L\'evy measure $\nu$ and zero drift. \end{proof}
In the verification of Condition (D1) of Theorem 2.1 in \cite{G10a}, more precisely in the proof of the claim $(II)$, $(III)\rightarrow 0$, one sees that Condition (A-1) is used to replace $\pi_n^{J,t}$ by its average over $P_\mu$. This is to be contrasted with the setting of \cite{BG10} where $(II)$ and $(III)$ vanish because $J$ is already in the invariant measure after $\theta_n$ steps, and hence for $x,x'\in\VV$ and $j>k$ the event $\{J(k\theta_n)=x\}$ is essentially independent of $\{J(j\theta_n)=x'\}$.

\begin{proof}[Proof of Theorem \ref{theorem1.3}]
As in the proof of Theorem \ref{theorem1.1} we show that for given sequences $a_n,c_n,\theta_n$, a given initial distribution $\mu$ and for fixed $\omega\in \Omega$ (B-2)-(B-5) $\Rightarrow$ (A-2)-(A-4). Since both Theorems require that  the conditions are satisfied $\P$-a.s.~for all $t>0$ and all $u>0$, it suffices to consider a fixed realization $\omega\in \Omega$ and fixed $u>0$, $t>0$. Let us first establish that, under the assumptions of Theorem \ref{theorem1.3},
\be\label{2.17}
\textstyle{\lim_{n\rightarrow\infty} \bigl|\nu_n^t(u,\infty)-\overline\nu_n^t(u,\infty)\bigr|=0}.
\ee
By (B-2), \eqref{2.17} implies (A-2). Next
\bea
&&\textstyle{\bigl|\nu_n^t(u,\infty)-\overline\nu_n^t(u,\infty)\bigr|}\nonumber\\&\leq&
\textstyle{ \sum_{(x,k)\in {\cal A}_n^1} |P(J(k\theta_n)=x)-h_{k\theta_n}(x) | Q_n^u(x)}\nonumber\\
 &+& \textstyle{\sum_{(x,k)\in {\cal A}_n^2} |P(J(k\theta_n)=x)-h_{k\theta_n}(x) | Q_n^u(x) }.\label{2.18}
\eea
By \eqref{b52} of (B-5) the second summand tends to zero. The first summand is smaller than
\bea
&&\sup_{(x,k)\in {\cal A}_n^2}\frac{|P(J(k\theta_n)=x)-h_{k\theta_n}(x)|}{h_{k\theta_n}(x)}\sum_{(x,k)\in {\cal A}_n^2} h_{k\theta_n}(x)Q_n^u(x) \nonumber\\
&\leq&\sup_{(x,k)\in {\cal A}_n^2}\frac{|P(J(k\theta_n)=x)-h_{k\theta_n}(x) |}{h_{k\theta_n}(x)}\nu_n^t(u,\infty) ,\label{2.19}
\eea
and \eqref{b51} of (B-5) guarantees that it vanishes as $n\rightarrow\infty$, proving that (A-2) is satisfied. To establish that
\be\label{2.20}
\textstyle{\lim_{n\rightarrow\infty} \bigl|\sigma_n^t(u,\infty)-\overline\sigma_n^t(u,\infty)\bigr|=0},
\ee
we proceed as in \eqref{2.18}. Bounding $Q_n^u(x)\leq 1$, the claim of \eqref{2.20} follows from \eqref{2.18}-\eqref{2.19} and (A-3) is satisfied as well. Condition (A-4) follows in a similar way. This finishes the proof of Theorem \ref{theorem1.3}.
\end{proof}
\begin{proof}[Proof of Lemma \ref{lemma1.2}]
Let us show that \eqref{a11} and \eqref{a12} are always satisfied for transient $x$ and never for positive recurrent $x$. Since the ideas of proof are similar, we restrict ourselves to continuous time $J$'s. Let $x\in \VV$ be transient. Then, for $\mu'\in \{\d_x,\mu\}$ and any $\theta_n\gg1$,
\be\label{2.21}
\textstyle{\lim_{n\rightarrow\infty}\int_{\theta_n}^{\infty} P_{\mu'}(J(t)=x)dt=0}.
\ee
Now, for all $s<t$, $P_{\mu'}(J(t)=x)\geq P_{\mu'}(J(t-s)=x)\exp(-s\wt\lambda^{-1}(x))$, and so
\bea\label{2.23}
\textstyle{\sum_{k=1}^{k_n(t)-1} P_{\mu'}(J(k\theta_n)=x)}&=&\textstyle{\sum_{k=1}^{k_n(t)-1} \int_{k\theta_n}^{k\theta_n+1}P_{\mu'}(J(k\theta_n)=x)dt}\nonumber\\
&\leq&\textstyle{ e^{\wt\lambda^{-1}(x)}\int_{\theta_n}^{\infty}P_{\mu'}(J(t)=x)dt},
\eea
which by \eqref{2.21} tends to zero. This proves that \eqref{a11} and \eqref{a12} hold for transient $x\in \VV$. 

Since \eqref{a12} can only be satisfied if $P_{x}(J(t)=x)\rightarrow 0$ and since by Theorem 1.8.3 in \cite{Norris} $\lim_{t\rightarrow\infty}P_{x}(J(t)=x)>0$ for positive recurrent $x\in \VV$, (A-1) cannot hold for positive recurrent $x\in \VV$. By Theorem 3.5.3 in \cite{Norris} this also proves that 
 (A-1) cannot hold for $J$ that admit for an invariant probability measure.
\end{proof}
\begin{proof}[Proof of Lemma \ref{ergthmzdcor}]
Since the proofs are the same, we only prove the claim for (A-0)-(A-4). 
Assume that  (A-0)-(A-4) are satisfied $\P$-a.s.~for fixed $u>0$, $t>0$ and given $a_n,c_n,\theta_n$, and $\mu$.
We construct a set $\Omega^{\tau}\subseteq \Omega$ of full measure on which (A-0)-(A-4) are satisfied for all $u>0$, $t>0$. The sums on the left hand sides of \eqref{a11}, \eqref{a12}, \eqref{a4}, and the quantities $\nu_n^t(u,\infty)$, and $\sigma_n^t(u,\infty)$ depend on $t$ through $k_n(t)\pi_n^t(x)$, $x\in \VV$, which is increasing in $t$. Moreover, as sums of tail distributions, the quantities $\PP_\mu(S_n^{J,b}(0)>u)$, $\nu_n^t(u,\infty)$, and $\sigma_n^t(u,\infty)$ are decreasing in $u$. The right hand sides of \eqref{a11}-\eqref{a4} are continuous in $t$. The only right hand side that depends on $u$ is that of \eqref{a2} and in (A-2) we require that \eqref{a2} holds for all continuity points of the mapping $u\mapsto\nu(u,\infty)$. Thus, $\Omega^\tau\equiv \bigcap_{u,t>0, t\in \Q, u\in \Q, \nu(\{u\})=0}\Omega^{\tau}(u,t)\subseteq\Omega$ is of full measure and (A-0)-(A-4) hold true for all $u>0$ and all $t>0$ on $\Omega^\t$. The proof of Lemma \ref{ergthmzdcor} is finished.
\end{proof}


\section{Application to BATM} \label{S3}
This section and the next are devoted to the proof of Theorem \ref{theorem1.5}. In the present section we derive new conditions that imply (B-2)-(B-5) and are specific to BATM. We also show that (A-0) and (A-1) hold true for BATM. In Section \ref{S4} we prove that these new conditions are satisfied and give the conclusion of the proof. 
\subsection{The VSRW}\label{S3.1}
We collect results for $J$ that are used in the proof of Theorem \ref{theorem1.5}. The VSRW is a well-studied Markov jump process in random environment (see \cite{BaDe10}, \cite{BaCe11}, \cite{Ce11}, \cite{ABDH}, and the references therein). The proof of Theorem \ref{theorem1.5} relies heavily on very precise results for $J$ that can be found in \cite{BaDe10}. The results that we are using repeatedly concern the heat kernel, which we now define. For $x,y\in \Z^d$ and $t>0$ the heat kernel is given by
\be\label{3.1}
q_t(x,y)\equiv P_x(J(t)=y).
\ee
The bounds for $q_t(x,y)$ that are contained in \cite{BaDe10} allow us to control all hitting, local, and exit times of vertices and balls that we need for the proof of Theorem \ref{theorem1.5}.  Moreover, we use the local central limit theorem which can be found in \cite{BaDe10}. Note that in virtue of Theorem 6.1 in \cite{BaDe10} and Lemma 9.1 in \cite{BaCe11}, these theorems apply in the present setting. We denote by $|\cdot|$ the Euclidean distance. For convenience, we restate Theorem 1.2 (a)-(c) (heat kernel bounds) and Theorem 5.14 (uniform local central limit theorem) from \cite{BaDe10}. 
\begin{theorem}\label{theorem3.1}
There exists $c_1\in(0,\infty)$ such that for all $x,y\in \Z^d$ and $t>0$,
\be
q_t(x,y)\leq c_1 t^{-d/2}.\label{theorem311}
\ee
There exist identically distributed random variables $\{U_x\}_{x\in \Z^d}$ whose distribution satisfies
\be\label{theorem312}
\P(U_x>v)\leq c_1 \exp(-c_2 v^{1/3}), \quad v>0,
\ee
where $c_1,c_2\in(0,\infty)$, and such that we have
\bea\label{theorem313}
q_t(x,y)\hspace{-2mm}&\leq&\hspace{-2mm} c_1 t^{-d/2} e^{-c_2 |x-y|\{1\wedge |x-y|t^{-1}\}},\quad \mbox{ if }|x-y|\vee t^{1/2}\geq U_x,\\
q_t(x,y)\hspace{-2mm}&\geq&\hspace{-2mm} c_1 t^{-d/2} e^{-c_2 |x-y|^2t^{-1}},\quad \mbox{ if }t\geq U_x^2\vee |x-y|^{4/3}\label{theorem314}.
\eea
For $x\in \R^d$ write $\lfloor x\rfloor =(\lfloor x_1\rfloor, \ldots, \lfloor x_d\rfloor)$. There exists $c_v>0$ such that, for $T>0$,
\bea\label{theorem315}
\textstyle{\lim_{n\rightarrow\infty}\sup_{x\in \R^d}\sup_{t\geq T}\bigl|n^{d/2} q_{nt}(0,\lfloor n^{1/2}x\rfloor)-(2\pi c_vt)^{d/2} e^{-|x|^2/2c_vt}\bigr|=0}, \quad \P\mbox{-a.s.}.
\eea
\end{theorem}
For $x\in \Z^d$, define
\be\label{3.7}
\textstyle{A_n(x)\equiv\{\omega\in \Omega: \sup_{y: |x-y|\leq {2a_n}} U_y \leq c_0 (\log a_n)^3\};}
\ee 
by convention $A_n(0)\equiv A_n$. By \eqref{theorem312}, there exists $c_0\in (0,\infty)$ such that $\P(A_n^c)\leq c_1 n^{-(5\vee 2d)}$. Therefore, writing
\be\label{3.7a}
\textstyle{A\equiv \bigcup_{n\geq 1}\bigcap_{m\geq n}A_m},
\ee
we have by Borel-Cantelli Lemma that $\P(A)=1$. On the event $A$, we have for all but finitely many $n$ that $\sup_{y: |y|\leq {2a_n}} U_y \leq c_0 (\log a_n)^3$. We will make use of Theorem \ref{theorem3.1} on the events $A_n$ and $A$. Whenever we do so, we check whether, given $x,y$ such that $|x|,|y|\leq a_n$ and $t>0$, $|x-y|\wedge t^{1/2}\geq c_0 (\log a_n)^3$ or $t^{1/2}\geq c_0 (\log a_n)^3\vee |x-y|^{2/3}$. 

We now state two lemmata that are needed in the proof of Theorem \ref{theorem1.5}. Their proofs are postponed to the appendix. The first concerns the distribution of the exit times of certain balls. We denote by $B_r(x)$ the ball of radius $r$ centered at $x$; by convention $B_r\equiv B_r(0)$. We write $\eta(B_r(x))$ for the exit time of $B_r(x)$. 
\begin{lemma}\label{lemma3.2}
Let $a_n$ be as in \eqref{1.45}. There exists $c_4\in(0,\infty)$ such that the following holds. For all sequences $m_n,r_n$ such that $m_n\geq c_0^2r_n (\log a_n)^6$ and $a_n\geq m_n$, on the event $A_n$,
\be\label{lemma321}
\textstyle{P_x(\eta( B_{r_n}(x)) \leq m_n) \leq e^{-c_4 r_n^2 m_n^{-1}}, \quad \forall x\in B_{a_n}.}
\ee
For all sequences $m_n,r_n$ such that $a_n\geq r_n\geq c_0(\log a_n)^3$ and $m_n\geq 3 r_n^2$,  on the event $A_n$, 
\be\label{lemma322}
\textstyle{P_x(\eta( B_{r_n}(x)) \geq m_n) \leq e^{-c_4 m_n^{1/2} r_n^{-1}}, \quad \forall x\in B_{a_n}.}
\ee
\end{lemma}
The second lemma provides bounds on the expected number of different sites that $J$ visits in certain time intervals. Given an increasing sequence of integers, $m_n$, we define the range of $J$ in the time interval $[0,m_n]$ as
\bea\label{3.10}
\textstyle{R_{m_n}\equiv \sum_{y\in \Z^d } \1_{\sigma(y)\leq m_n},}
\eea
where $\sigma(y)\equiv\inf\{t\geq 0: J(t)=y\}$ is the hitting time of $y$. 
\begin{lemma}\label{lemma3.3}
Let $m_n$ be such that $a_n\geq m_n\geq c_0^2 (\log a_n)^{6}$. There exists $c_5\in (0,\infty)$ such that the following holds for $n$ large enough. For $d\geq 2$ and $k\in \{1,2\}$,
\bea\label{lemma331}
\textstyle{\E E_xR_{m_n}^k \leq c_5\left(m_n(\log m_n)^{-k} \1_{d=2}+m_n^{1/k}\1_{d\geq 3}\right).}
\eea
Moreover, for $d=2$ there exists $f_{m_n}:(0,\infty)\rightarrow (0,\infty)$ such that, on the event $A_n$, 
\bea
&&P(\sigma(x)\leq m_n)\leq f_{m_n}(|x|),\quad\mbox{ for all } x\in B_{m_n},\eea
and $f_{m_n}$ satisfies
\bea
\textstyle{\sum_{x\in B_{m_n} }\bigl( f_{m_n}(|x|)\bigr)^k\leq c_5m_n(\log m_n)^{-k}\label{lemma332},\quad k\in\{1,2,4\}.\quad\quad}
\eea
\end{lemma}
Notice that by our choices of $\theta_n$ we may use Lemma \ref{lemma3.3} for $m_n\geq \theta_n^{\d}$ for $\d>1/2$. 
\subsection{Specializing Theorem \ref{theorem1.3} for BATM}\label{s3.2}
In this section we specialize Theorem \ref{theorem1.3} to the setting of BATM. 
More precisely, we will not study $S_n^{J,b}$ directly, but another process, $\overline S_n^{J,b}$, to which only those $x$ contribute for which $\tau(x)$ is 'large enough'. For $x\in\Z^d$ we set
\bea\label{3.16}
\gamma_{n}(x)\equiv c_n^{-1} \tau(x) .
\eea
Let  $\epsilon_n(d)\equiv (\log \theta_n )^{-6/(1-\a)} \1_{d=2} +\theta_n^{-1/3} \1_{d\geq 3}$ and denote  the collection of 'large' traps by $T_n\equiv\{x\in \Z^d: \gamma_n(x)>\epsilon_n, \max_{y\sim x}\t(y)\leq\epsilon_n^{-2/\a}\}$. Then,
\be\label{3.17}
\textstyle{\overline S_n^{J,b}(t) \equiv \sum_{k=0}^{k_n(t)-1}\sum_{x\in\Z^d}\gamma_{n}(x)\1_{x\in T_n}\bigl(\ell_{\theta_n(k+1)}(x)-\ell_{\theta_nk}(x)\bigr) , \quad t>0,}
\ee
where $\ell_{t}(y)=\int_0^t \1_{J(s)=y}ds$. 

Roughly speaking, the following lemma states that, $\P$-a.s., $\overline S_n^{J,b}$ is a good approximation for $S_n^{J,b}$. To simplify notation, we write $P\equiv P_0$, respectively $\PP\equiv\PP_0$.
\begin{lemma}\label{lemma3.5}
$\P$-a.s., $\limsup_{n\rightarrow\infty}\PP(\rho_{\infty}(S_n^{J,b},\overline S_n^{J,b})>\d_n)=0$, where $\d_n\equiv\epsilon_n^{(1-\a)/2}$.
\end{lemma}
\begin{proof}
By definition of $\rho_{\infty}$ it suffices to show this result with $\rho_{\infty}$ replaced by $\rho_r$, Skorohod's $J_1$ metric on $D[0,r]$, for all $r>0$. For convenience we take $r=1$ and we get
\bea\label{3.18}
\PP\Bigl( S_n^{J,b}(1)-\overline S_n^{J,b}(1)>\d_n\Bigr)\hspace{-2mm}&=&\hspace{-2mm}\textstyle{ \PP\Bigl( \int_0^{a_n} \gamma_{n}(J(s)) \1_{J(s)\notin T_n}  >\d_n\Bigr)}\nonumber\\
\hspace{-2mm}&\leq&\hspace{-2mm}\textstyle{\PP\bigl(\int_0^{a_n} \gamma_{n}(J(s)) \1_{J(s)\notin T_n, \g_n(J(s))\leq \epsilon_n^{-1}}  >\d_n\bigr)}\nonumber\\
\hspace{-2mm}&+&\hspace{-2mm}\textstyle{\PP\bigl(\exists x\in \Z^d:  \g_n(x)>\epsilon_n^{-1},\ell_{a_n}(x)>0\bigr)}.\quad\quad
\eea
To shorten notation, set $B_n\equiv T_n^c\cap \{y\in \Z^d: \g_n(y)\leq \epsilon_n^{-1}\}$. 
Using a first order Chebyshev inequality to bound the first term in the right hand side of \eqref{3.18} and Boole's inequality for the second, we get that their sum is bounded above by
\bea\label{3.19a}
\textstyle{\d_n^{-1}\int_{0}^{a_n}\EE\gamma_{n}(J(s)) \1_{J(s)\in B_n} ds +\sum_{x\in \Z^d} P(\ell_{a_n}(x)>0) \1_{\g_n(x)>\epsilon_n^{-1}}}. \quad\quad\eea
In order to establish that \eqref{3.19a} tends $\P$-a.s.~to zero, let us first consider subsequences of the form $c_{N,r}=\exp(N+r)$ for $r\in [0,1]$ and establish that, uniformly in $r\in [0,1)$, \eqref{3.19a} tends $\P$-a.s.~to zero as $N\rightarrow \infty$. Since $c_n=n=\exp(\lfloor \log n\rfloor + (\log n-\lfloor \log n\rfloor ))$ this implies that \eqref{3.19a} vanishes $\P$-a.s.~as $n\rightarrow\infty$. To ease notation we write $a_{N,r}\equiv a_{\exp(N+r)}$ for $r\in [0,1]$ and $N\in \N$,  and use the same abbreviation for all $n$ dependent quantities. 

Now, the first summand in \eqref{3.19a} satisfies
\be\label{3.19}
\textstyle{\sup_{r\in[0,1)}\d_{N,r}^{-1} \sum_{y\in B_{N,r}} \EE\ell_{a_{N,r}}(y)\gamma_{N,r}(y)\leq \d_{N,1}^{-1}\sum_{y\in B_{N,1}} \EE\ell_{a_{N,1}}(y)\gamma_{N,0}(y)},
\ee
and the supremum over $r\in[0,1)$ of the second is bounded above by
\be\label{3.19aa}
\textstyle{\sum_{x\in \Z^d}P(\ell_{a_{N,1}}(x)>0)\1_{\g_{N,0}(x)>\epsilon_{N,1}^{-1}}\leq\sum_{x}P(\max_{y\sim x}\ell_{a_{N,1}}(y)>0)\1_{\g_{N,0}(x)>\epsilon_{N,1}^{-1}}}.
\ee
The lemma will be proven if we can show that the sum of the expectation of the right hand side of \eqref{3.19} and \eqref{3.19aa} with respect to the random environment, that is
\bea \label{3.20}
\textstyle{\sum_{x}\frac{1}{\d_{N,1}}
\E\bigl[\EE\ell_{a_{N,1}}(x) \gamma_{N,0}(x) \1_{ x\in B_{N,1}}\bigr]+\sum_{x}\E\bigl[P(\ell_{a_{N,1}}(x)>0)\bigr] \P(\g_{N,0}(0)>\epsilon_{N,1}^{-1})},\ \ \ 
\eea
tends to zero fast enough. Notice that the second sum in \eqref{3.20} bounds \eqref{3.19aa} because $P(\max_{y\sim x}\ell_{a_{N,1}}(y)>0)$ is independent of $\gamma_{N,0}(x)$. By \eqref{lemma331} of Lemma \ref{lemma3.3}, the second sum in \eqref{3.20} is bounded above by 
\be\label{3.20a}
c_5 a_{N,1} c_{N,0}^{-\a} \epsilon_{N,1}^{\a}(1/\log a_{N,1} \1_{d=2}+\1_{d\geq 3})\leq \epsilon_{N,1}^{\a} ,
\ee
which is summable in $N$. We now decompose the first sum in \eqref{3.20} into three sums according to the size of $|x|$. Namely, we set $D_1\equiv B_{c_0 (\log a_{N,1})^3}$, $D_2\equiv B_{a_{N,1}^{1/2} \log\log a_{N,1}}\setminus D_1$, and $D_3\equiv (B_{a_{N,1}^{1/2} \log\log a_{N,1}})^c$. When $x\in D_1$, we know by \eqref{theorem311} of Theorem \ref{theorem3.1}  that $\EE\ell_{a_{N,1}}(x)\leq c_1\log a_{N,1} $, and so
\bea \label{3.21}
\textstyle{\tfrac{1}{\d_{N,1}}\sum_{x\in D_1} 
\E\bigl[\EE\ell_{a_{N,1}}(x) \gamma_{N,0}(x) \1_{x\in B_{N,1}}\bigr]\leq \tfrac{c_1|D_1|\log a_{N,1}}{\d_{N,1}}
\E\bigl[ \gamma_{N,0}(0)  \1_{0\in B_{N,1}}\bigr]}.\ \ \ 
\eea
One can check that $\E\gamma_{N,0}(0) \1_{0\in B_{N,1}} \leq c c_{N,1}^{-\alpha} \epsilon_{N,1}^{1-\alpha}$, and so \eqref{3.21} is smaller than, say, $c_{N,1}^{-\alpha/2}$ which is summable in $N$. 
For $x\in D_3$ we derive from \eqref{theorem313} and \eqref{theorem311} of Theorem \ref{theorem3.1} that, $\EE\ell_{a_{N,1}}(x)\leq e^{-c_2 |x|^2/ a_{N,1}}+\log a_{N,1}\1_{U_x> |x|}$. Thus, by \eqref{theorem312}  of Theorem \ref{theorem3.1}, 
\bea\label{3.22}
\hspace{-2mm}&&\hspace{-2mm}\tfrac{1}{\d_{N,1}}\textstyle{\sum_{x\in D_3} 
\E\bigl[\EE\ell_{a_{N,1}}(x) \gamma_{N,0}(x)  \1_{x\in B_{N,1}}\bigr]}\nonumber\\
\hspace{-2mm}&\leq&\hspace{-2mm}\tfrac{c}{\d_{N,1}}\textstyle{\sum_{x\in D_3}\bigl\{e^{-c_2/2 |x|^2/ a_{N,1}}c_{N,1}^{-\alpha}+\epsilon_{N,1}^{-1/\a}\log a_{N,1} \P(U_x> |x|)\bigr\}}
\leq e^{-c'(\log\log a_{N,1})^2},\quad\quad\ \ 
\eea
which  is summable in $N$. Finally, let $x\in D_2$. Let $A_{N,1}\equiv A_{\exp(N+1)}(0)$  be as in \eqref{3.7}. In order to bound  the contribution to the first sum in \eqref{3.20}  coming from $D_2$ we distinguish between the events $A_{N,1}$ and $A_{N,1}^c$. By definition of $B_{N,1}$, we have that
\bea
\hspace{-2mm}&&\hspace{-2mm}\tfrac{1}{\d_{N,1}}\textstyle{\sum_{x\in D_2} 
\E\left[\EE\ell_{a_{N,1}}(x) \gamma_{N,0}(x)  \1_{x\in B_{N,1}}\1_{A_{N,1}^c}\right]}
\nonumber\\
\hspace{-2mm}&\leq&\hspace{-2mm} \tfrac{\epsilon_{N,1}}{\d_{N,1}}\textstyle{
\E\left[ \1_{A_{N,1}^c}\sum_{x\in D_2} \EE\ell_{a_{N,1}}(x)\right]}\leq a_{N,1} \P(A_{N,1}^c)\leq a_{N,1}^{-4}.
\eea
On $A_{N,1}$ we have by \eqref{theorem313} of Theorem \ref{theorem3.1} that $\EE\ell_{a_{N,1}}(x)\leq c_{3} |x|^{2-d}$ if $d\geq 3$, and $\EE\ell_{a_{N,1}}(x)\leq c_2\log a_{N,1}$ if $d=2$, and so,
\bea \label{3.23}
\hspace{-2mm}&&\hspace{-2mm}\tfrac{1}{\d_{N,1}}\textstyle{\sum_{x\in D_2} 
\E\left[\EE\ell_{a_{N,1}}(x) \gamma_{N,0}(x)  \1_{x\in B_{N,1}}\1_{A_{N,1}}\right]}\nonumber\\
\hspace{-2mm}&=&\hspace{-2mm}\tfrac{c}{\d_{N,1}}\textstyle{\sum_{x\in D_2}c_{N,1}^{-\alpha} \epsilon_{N,1}^{1-\alpha} \left(|x|^{2-d} \1_{d\geq 3} + \log a_{N,1} \1_{d=2} \right) }\nonumber\\
\hspace{-2mm}&=&\hspace{-2mm}\tfrac{c}{\d_{N,1}}\textstyle{c_{N,1}^{-\alpha} \epsilon_{N,1}^{1-\alpha} \sum_{k=1}^{a_{N,1}^{1/2} \log\log a_{N,1}} 
k\left( \1_{d\geq 3} + \log a_{N,1} \1_{d=2} \right)} \nonumber\\
&\leq& c'\e^{-1} N^{-1-\a}\log N,\quad
\eea
where $c,c'\in(0,\infty)$ and where the last line follows from \eqref{1.44}, \eqref{1.45}, and the construction of $c_{N,r}$. Collecting \eqref{3.20a}-\eqref{3.23}, the proof of Lemma \ref{lemma3.5} is complete.
\end{proof}
Using Theorem \ref{theorem1.3}, we can derive new conditions for the process $\overline S_n^{J,b}$ to converge. To present these conditions, we introduce the following quantities. For $x,y\in \Z^d$, $u>0$, and $\e>0$ we define
\bea\label{3.24}
Q_n^u(x,y)&\equiv& \PP_{x}\left(\ell_{\theta_n}(y)\gamma_{n}(y)> u, \eta(B_{\theta_n}(x))> \theta_n\right) \1_{y\in T_n},\\
M_n^\e(x,y)&\equiv&\EE_x\left(\ell_{\theta_n}(y)\gamma_{n}(y)\1_{\gamma_{n}(y)\ell_{\theta_n}(y)\leq \varepsilon}\1_{\eta(B_{\theta_n}(x))> \theta_n}\right) \1_{y\in T_n}\label{3.25}.
\eea
Note that $Q_n^u(x,y)=M_n^\e(x,y)=0$ for $y\notin B_{\theta_n}(x)$. For $t>0$ we set $d_n(t)\equiv \lfloor a_n t \rfloor^{1/2} \log \lfloor a_n t \rfloor$. For $n\in \N$ and $x\in \Z^d$ we take $h_{n}(x)=\E P(J(n)=x)$. Thus, $\overline\pi_n^t(x)=\E\pi_n^t(x)$. By analogy with \eqref{1.26} and \eqref{1.27} we write, for $u>0$, $t>0$,
\be\label{3.26}
\textstyle{ \wt\nu_n^t(u,\infty)\equiv k_n(t)\sum_{x\in B_{d_n(t)}}\overline\pi_n^t(x)\sum_{y\in \Z^d}Q_n^u(x,y)}, 
\ee
and
\be\label{3.27}
\textstyle{\wt\sigma_n^t(u,\infty)\equiv k_n(t)\sum_{x\in B_{d_n(t)}}\overline\pi_n^t(x)\sum_{y\in \Z^d}(Q_n^u(x,y))^2.}
\ee
We also define for $\varepsilon>0$, $t>0$ 
\be\label{3.28}
\textstyle{m_n^t(\e)\equiv k_n(t)\sum_{x\in B_{d_n(t)}}\overline\pi_n^t(x)\sum_{y\in \Z^d}M_n^\e(x,y)},
\ee
and finally we introduce for $\e>0$ the set
\be\label{3.28a}
{\cal B}_n\equiv\{(x,k)\in B_{d_n(t)}\times [k_n(t)-1]\setminus\{0\}: |x|^2< \e k\theta_n,\  |x|^2> k\theta_n/\e\}.
\ee
We are now ready to present our new conditions. They are stated for fixed $\omega\in \Omega$.

{\bf (C-2)}
For all $u>0$, $t>0$
\be
\textstyle{\lim_{n\rightarrow\infty} \wt\nu_n^t(u,\infty) = t \cal{K} u^{-\alpha}.\label{c2}}
\ee

{\bf (C-3)}
For all $u>0$, $t>0$
\be
\textstyle{\lim_{n\rightarrow\infty}\wt\sigma_n^t(u,\infty)= 0.\label{c3}}
\ee

{\bf (C-4)}
For all $t>0$ there exists $C(t)>0$ such that for all $\varepsilon>0$ 
\be
\textstyle{\limsup_{n\rightarrow\infty}m_n^t(\e)\leq C(t)\varepsilon^{1-\alpha}.\label{c4}}
\ee

{\bf (C-5)}
For all $u>0$, $t>0$, $\e>0$ there exists $C(u,t)\in(0,\infty)$ and $N(\e)$ such that for $n\geq N(\e)$,
\bea
\sum_{(x,k)\in{\cal B}_n}\sum_{y} (k\theta_n)^{-d/2}e^{-c_2|x|^2/k\theta_n}Q_n^u(x,y)\hspace{-2mm}&\leq&\hspace{-2mm} C(u,t)\varepsilon,\label{c51}\\
\sum_{(x,k)\in{\cal B}_n}\sum_{y}(k\theta_n)^{-d/2}e^{-c_2|x|^2/k\theta_n}M_n^\e(x,y)\hspace{-2mm}&\leq&\hspace{-2mm}  C(u,t)\varepsilon.\label{c52}\quad
\eea
\begin{proposition}\label{proposition3.6}
Assume that Conditions (C-2)-(C-5) are satisfied $\P$-a.s. for fixed $u>0$, $t>0$, and $\e>0$. Then, $\P$-a.s., $\overline S_n^{J,b}\stackrel{J_1}{\Longrightarrow} V_\a$, as $n\rightarrow\infty$, where $\stackrel{J_1}{\Longrightarrow}$ denotes weak convergence in the space $D[0,\infty)$ equipped with Skorohod's $J_1$ topology. 
\end{proposition}
The next lemma gives us a very helpful bound for $\pi_n^t(x)$ which we will use in the proof of Proposition \ref{proposition3.6} and in the following sections.
\begin{lemma}\label{lemma3.7}
There exists $c_3\in(0,\infty)$ such that for all $t>0$ and large enough $n$ we have that if $d\geq 3$,
\bea
\overline\pi_n^t(x)\leq 
(k_n(t)\theta_n)^{-1}\begin{cases}
c_3 (|x|\vee 1)^{2-d},  & \mbox{ if }|x|\leq a_n^{1/2} \log \theta_n, \\
c_3 |x|^{2-d}e^{-1/2(\log \theta_n)^2},  & \mbox{ else},
\end{cases}\label{4.76}
\eea
and if $d=2$,
\bea\label{4.74}
\overline\pi_n^t(x)\leq(k_n(t)\theta_n)^{-1}
\begin{cases}
 c_3 (\log (a_n/|x|^2)\vee (\log\log a_n)), & \mbox{ if } |x|\leq a_n^{1/2} \log \log a_n,\\
 e^{-c_2/2 (\log\log a_n)^2 },& \mbox{ if }  a_n^{1/2} \log \log a_n<|x|.
\end{cases}\quad
\eea
\end{lemma}
We first prove Proposition \ref{proposition3.6} assuming Lemma \ref{lemma3.7} and the lemma next.
\begin{proof}[Proof of Proposition \ref{proposition3.6}]
Let us apply Theorem \ref{theorem1.3} to $\overline S_n^{J,b}$. By Lemma \ref{ergthmzdcor} it suffices to prove that the conditions of Theorem \ref{theorem1.3} are verified $\P$-a.s.~for fixed $u>0$, $t>0$, and $\e>0$. 

We first prove that Condition (A-1) is satisfied. By \eqref{theorem311} of Theorem \ref{theorem3.1} we have for all $\omega\in \Omega$, all $x,y\in \Z^d$, and all $t>0$,
\bea\label{3.36}
\textstyle{\sum_{k=1}^{k_n(t)-1}q_{k\theta_n}(x,y)}\leq\textstyle{\sum_{k=1}^{k_n(t)-1}c_1(\theta_nk)^{-d/2}}\leq\textstyle{\theta_n^{-1}\sum_{k=1}^{k_n(t)-1}c_1k^{-1}}\leq \textstyle{2c_1\frac{\log a_n t}{\theta_n}}.
\eea
By \eqref{1.44} and \eqref{1.45} this vanishes as $n\rightarrow\infty$, and hence (A-1) is satisfied $\P$-a.s.

To verify (A-0) and establish that (C-2)-(C-4) $\Rightarrow$ (B-2)-(B-4) we proceed as in the proof of Lemma \ref{lemma3.5} and consider subsequence $c_{N,r}=\exp(N+r)$ first (see the paragraph below \eqref{3.19a}). Since 
$
\overline Z_{N,r, 1}^{J}\equiv
\sum_{x\in\Z^d}\gamma_{n}(x)\1_{x\in T_n}\ell_{\theta_{N,r}(x)}
$
is zero unless there is $y\in T_{N,r}$ for which $\ell_{\theta_{N,r}}(y)>0$,
\be\label{3.35a}
\textstyle{\sup_{r\in[0,1)}\PP(\overline Z_{N,r, 1}^{J}> u)\leq P(\eta(B_{\theta_{N,1}^{3/4}})\leq\theta_{N,1})+\sum_{y:|y|\leq\theta_{N,1}^{3/4}}\1_{y\in T_{N,0}}.}
\ee
By definition of $A$ (see \eqref{3.7a}) and Lemma \ref{lemma3.2}, the first term in the right hand side of \eqref{3.35a} tends $\P$-a.s.~to zero. The second term in the right hand side of \eqref{3.35a} is bounded above, when taking expectation with respect to the random environment, by $\theta_{N,1}^{3d/4} c_{N,0}^{-\a}\epsilon_{N,0}^{-\a}$. This is summable in $N$ and hence (A-0) is satisfied $\P$-a.s.

We establish now that (C-2)-(C-4) $\Rightarrow$ (B-2)-(B-4). First we prove that (C-2) $\Rightarrow$ (B-2), i.e.~that $\sup_{r\in [0,1)}|\overline\nu_{N,r}^t(u,\infty)-\wt \nu_{N,r}^t(u,\infty)|$ tends $\P$-a.s.~to zero as $N\rightarrow\infty$. We have that
\bea
\textstyle{\sup_{r\in[0,1)}|\overline\nu_{N,r}^t(u,\infty)-\wt \nu_{N,r}^t(u,\infty)|\leq \Delta_{N}^1+\Delta_{N}^2+\Delta_{N}^3,}
\eea
where
\bea
\Delta_{N}^1\hspace{-2mm}&\equiv&\hspace{-2mm}\textstyle{\sup_{r\in[0,1)} k_{N,r}(t)\sum_{x\notin B_{d_{N,r}(t)}} \overline\pi_{N,r}^t(x)},\\
\Delta_{N}^2\hspace{-2mm}&\equiv&\hspace{-2mm}\textstyle{\sup_{r\in[0,1)}k_{N,r}(t)\sum_{x\in B_{d_{N,r}(t)}} \overline\pi_{N,r}^t(x)P_x(\eta(B_{\theta_{N,r}}(x))\leq \theta_{N,r})},\\
\Delta_{N}^3\hspace{-2mm}&\equiv&\hspace{-2mm}\textstyle{\sup_{r\in[0,1)}k_{N,r}(t)\sum_{x\in B_{d_{N,r}(t)}}\overline\pi_{N,r}^t(x)\sum_{y,y': |x-y|, |x-y'|\leq \theta_n}\1_{y,y'\in T_{N,r}}}.
\eea
Let us now prove that, $\P$-a.s., $\Delta_{N}^i$ vanishes for $i=1,2,3$ as $N\rightarrow\infty$. We have that
\bea
\Delta_{N}^1\hspace{-2mm}&=&\hspace{-2mm}\textstyle{\sup_{r\in [0,1)} \sum_{k=1}^{k_{N,r}(t)-1} \sum_{x\notin B_{d_{N,r}(t)}}\E P(J(k\theta_{N,r})=x)}\nonumber\\
\hspace{-2mm}&\leq&\hspace{-2mm}\textstyle{\sum_{k=1}^{k_{N,1}(t)} \E P(\eta(B_{d_{N,0}(t)})\leq k\theta_{N,1})}
\leq\textstyle{k_{N,1}(t) e^{-c_4\theta_{N,0}}+k_{N,1}(t)\P(A_{N,1}^c)},\quad\quad \label{3.40}
\eea
where we used \eqref{lemma321} of Lemma \ref{lemma3.2} on $A_{N,1}=A_{\exp(N+1)}$ in the last step. By construction, $k_{N,1}(t) \P(A_{N,1}^c)\leq \exp(-cN)$, and so \eqref{3.40} is bounded above by $\exp(-cN)$.
Thus, $\P$-a.s., $\Delta_{N}^1\rightarrow 0$. The same arguments yield that $\P$-a.s., $\Delta_{N}^2\rightarrow 0$.
Finally, writing 
\be\label{defpi}
\textstyle{\overline\pi_{N}^t(x)\equiv\sup_{r\in[0,1)} \overline \pi_{N,r}^t(x)},
\ee
we get by a first order Chebyshev inequality that
\be\label{belowpi}
\textstyle{\P\left(\Delta_{N}^3>\varepsilon \right)\leq \frac{k_{N,1}(t)}{\e}\sum_{x\in B_{d_{N,1}(t)}}\overline\pi_{N}^t(x)\sum_{y,y'\in B_{\theta_{N,1}}(0), y\neq y'}\E(\1_{y,y'\in T_{N,0}}).}
\ee
By Lemma \ref{lemma3.7} one can show that the sum over $x\in B_{d_{N,1}(t)}$ of $\overline\pi_{N}^t(x)$ is bounded above by $c_3(\log\log a_{N,1})^3$. The sum over $y,y'\in B_{\theta_{N,1}}$ in \eqref{belowpi} is equal to $\theta_{N,1}^d c_{N,0}^{-2\a}\epsilon_{N,0}^{-2\a}$. 
Thus, $\P\left(\Delta_{N}^3>\varepsilon \right)\ll c_{N,0}^{-\a/4}$. 
This is summable in $N$, and so, $\P$-a.s, $\Delta_{N}^3\rightarrow0$. Therefore, (C-2) $\Rightarrow$ (B-2).
In a similar way one can show that (C-3) $\Rightarrow$ (B-3).
We now prove that (C-4) $\Rightarrow$ (B-4). Observe that for $r\in [0,1)$,
\bea\label{3.4111}
\hspace{-2mm}&&\hspace{-2mm}\textstyle{\sum_{x\in \Z^d}\overline\pi_{N,r}^t(x)\EE_x\Bigl[ \overline Z_{{N,r},1}^{J}\1_{\overline Z_{{N,r},1}^{J}\leq \e}\Bigr]}\nonumber\\
\hspace{-2mm}&\leq&\hspace{-2mm}\textstyle{ \Delta_{N}^1+\Delta_{N}^2+ \sum_{x\in B_{d_{N,r}(t)}}\overline\pi_{N,r}^t(x)\EE_x \Bigl[\overline Z_{{N,r},1}^{J}\1_{\overline Z_{{N,r},1}^{J}\leq\e}\1_{\eta(B_{\theta_{N,r}}(x))>\theta_{N,r}}\Bigr]}.\ 
\eea
By \eqref{3.40}, $\Delta_{N}^1,\Delta_{N}^2\rightarrow 0$. It remains to establish that for all $r\in [0,1)$, $m_{N,r}^t(\e)$ is an upper bound for the last term in the right hand side of \eqref{3.4111}. This term is equal to
\be
\textstyle{\sum_{x\in B_{d_{N,r}(t)}}\overline\pi_{N,r}^t(x)\sum_{y\in \Z^d \cap T_{N,r}}\EE_x \Bigl[\gamma_n(y)\ell_{\theta_{N,r}}(y)\1_{\overline Z_{{N,r},1}^{J}\leq \e}\1_{\eta(B_{\theta_{N,r}}(x))>\theta_{N,r}}\Bigr]}.
\ee
Since $\{\overline Z_{{N,r},1}^{J}\leq \e\}\subset \{\ell_{\theta_{N,r}}(y)\gamma_n(y)\leq \varepsilon\}$ for all $y\in \Z^d$ for which $\ell_{\theta_{N,r}}(y)>0$, $m_{N,r}^t(\e)$ bounds the last term in the right hand side of \eqref{3.4111} from above. Thus, (C-4) $\Rightarrow$ (B-4).

Finally we prove that (C-5) $\Rightarrow$ (B-5). The local central limit theorem \eqref{theorem315} of Theorem \ref{theorem3.1} implies that $\P$-a.s., for ${\cal A}_n^1=\{(x,k): k\geq 1,|x|^2\in (\e k\theta_n, k\theta_n/\e) \}$,
\be\label{3.44}
\textstyle{\lim_{n\rightarrow\infty}\sup_{(x,k)\in {\cal A}_n^1}| (k\theta_n)^{d/2}q_{k\theta_n}(x)- (2\pi c_v)^{d/2} e^{-|x|^2/(2c_v k\theta_n )}|=0,}
\ee
where $q_t(x)\equiv q_t(0,x)$. By \eqref{theorem312} of Theorem \ref{theorem3.1}, $(k\theta_n)^{d/2}q_{k\theta_n}(x)\leq c_1$ for all $x\in \Z^d$, $k\in \N$ and so, by bounded convergence, 
\be\label{1.33}
\textstyle{\lim_{n\rightarrow\infty}\sup_{(x,k)\in {\cal A}_n^1}| (k\theta_n)^{d/2}h_{k\theta_n}(x)- (2\pi c_v)^{d/2} e^{-|x|^2/(2c_v k\theta_n T)}|=0,}
\ee
where $h_n(x)=\E q_n(x)$. Thus,  $\P$-a.s.,
\be\label{1.35}
\textstyle{\lim_{n\rightarrow\infty}\sup_{(x,k)\in {\cal A}_n^1}\frac{|q_{k\theta_n}(x)-h_{k\theta_n}(x)|}{h_{k\theta_n}(x)}=0,}
\ee
proving that \eqref{b51} is satisfied for ${\cal A}_n^1$. Hence, it suffices to verify \eqref{b52} and \eqref{b53} for the set ${\cal A}_n^2\equiv\Z^d\times [k_n(t)-1]\setminus {\cal A}_n^1$. The set ${\cal A}_n^2$ is the disjoint union of $[k_n(t)-1]\times \Z^d\setminus B_{d_n(t)}$ and ${\cal B}_n$. Let us now verify \eqref{b52} and \eqref{b53} for each of these sets separately.
Since $\Delta_n^1\rightarrow 0$,  $\P$-a.s., we know that $[k_n(t)-1]\times \Z^d\setminus B_{d_n(t)}$ satisfies \eqref{b52} and \eqref{b53}.
By \eqref{theorem313} of Theorem \ref{theorem3.1} we have that on the event $A_n$ $q_{k\theta_n}(x)\leq c_1 (k/\theta_n)^{-d/2}e^{-c_2|x|^2/k\theta_n}$ for all $x\in B_{d_n(t)}$ and all $k\geq 1$. By construction, $\P(A_n^c)|B_{d_n(t)}|\ll a_n^{d/2} n^{-(4\vee d)}\ll n^{-2}$ which is summable in $n$. Thus, by Borel-Cantelli Lemma, (C-5) implies \eqref{b52} and \eqref{b53} for ${\cal B}_n$. Thus, (C-5) $\Rightarrow$ (B-5). The proof of Proposition \ref{proposition3.6} is complete.
\end{proof}
\begin{proof}[Proof of Lemma \ref{lemma3.7}]
Let us construct a bound on $\overline\pi_n^t$. By \eqref{theorem313}  of Theorem \ref{theorem3.1} we have
\bea\label{4.72} 
\E(\pi_n^t(x) \1_{U_x\leq\theta_n})\hspace{-2mm}&\leq&\hspace{-2mm} \textstyle{\frac{1}{k_n(t)}\Bigl(\sum_{k=\lfloor|x|^2/\theta_n\rfloor\vee1}^{k_n(t)-1}(k\theta_n)^{-d/2}e^{-c_2 |x|/(k\theta_n)} + \sum_{k=1}^{\lfloor|x|/\theta_n\rfloor\wedge1}e^{-c_2|x|}\Bigr)}\ \ \ \nonumber\\
\hspace{-2mm}&\leq&\hspace{-2mm}  \textstyle{ c_1(k_n(t)\theta_n)^{-1} \int_{1/2 |x|\vee\theta_n}^{ a_nt}  s^{-d/2} e^{- c_2 |x|^2 s^{-1}}ds,}\ \ \ \ 
\eea
with the convention that $\sum_{k=1}^0=0$. Let $d\geq 3$ first. We substitute $u=c_2 |x|^2 s^{-1}$ and get
\be\label{4.75}
\textstyle{\E(\pi_n^t(x) \1_{U_x\leq\theta_n}) \leq c''(k_n(t) \theta_n)^{-1} |x|^{2-d}\Gamma\bigl(d/2-2, |x|^2/ a_n\bigr)},
\ee
where $c''\in(0,\infty)$. By \eqref{theorem312}, $\E\pi_n^t(x)\1_{U_x>\theta_n}\leq c_1e^{-c_2 \theta_n^{1/3}}$, and so,  taking $c_3$ large enough we get that \eqref{4.76} holds. 
Now let $d=2$. An asymptotic analysis reveals in \eqref{4.72} that for $|x|\leq\sqrt{a_n/\log a_n}$ we have for some $c'\in(0,\infty)$ that
\be\label{4.73}
\E(\pi_n^t(x) \1_{U_x\leq\theta_n})\leq c'(k_n(t)\theta_n)^{-1}\log(a_n/|x|^2).
\ee
Moreover, when $|x|\geq \sqrt{a_n}\log \log a_n$, $\E(\pi_n^t(x) \1_{U_x\leq\theta_n})\leq e^{-c_2/2 |x|^2/a_n}$. Since \eqref{4.72} is a decreasing function of $|x|$, we may bound $\E(\pi_n^t(x) \1_{U_x\leq\theta_n})\leq \log\log a_n$ for $|x|\in [\sqrt{a_n/\log a_n},\sqrt{a_n}\log \log a_n]$. Since $\E\pi_n^t(x)\1_{U_x>\theta_n}\leq c_1e^{-c_2/2 \theta_n^{1/3}}$, this proves \eqref{4.74}. This finishes the proof of Lemma \ref{lemma3.7}.
\end{proof}


\section{Verification of Conditions (C-2)-(C-5)}\label{S4}
In this section we show that (C-2)-(C-5) are satisfied. Let $u>0$, $t>0$ and $\e>0$ be fixed. In Section \ref{S4.1} we establish that $\lim_{n\rightarrow\infty}\E\wt \nu_n^t(u,\infty)=t \nu(u,\infty)$. Next, in Section \ref{S4.2}, we prove that $\P$-a.s., $\lim_{n\rightarrow\infty}\wt \nu_n^t(u,\infty)=t \nu(u,\infty)$. In Section \ref{S4.3} we prove that $\lim_{n\rightarrow\infty}\E\wt \sigma_n^t(u,\infty)=0$ and show that, $\P$-a.s.~$\wt \sigma_n^t(u,\infty)$ tends to zero. In Section \ref{S4.4} we establish that $\E m_n^t(\e)\leq C(t)\e^{1-\a}$ and show that $m_n^t(\e)$ concentrates $\P$-a.s.~around its mean value. We verify Condition (C-5) in Section \ref{S4.4}. Finally, we conclude the proof of Theorem \ref{theorem1.3} in Section \ref{S4.5}. 
\subsection{Convergence of $\E\wt\nu_n^t(u,\infty)$}\label{S4.1}
This section is devoted to the proof of convergence of $\E\wt\nu_n^t(u,\infty)$, which is the most demanding part of the proof of Theorem \ref{theorem1.3}. 
\begin{lemma}\label{lemma4.1}
 For all $u>0$, $t>0$, $\lim_{n\rightarrow\infty}\E\wt{\nu}_n^t(u,\infty)= t \nu(u,\infty)$.
\end{lemma}
\begin{proof}[Proof of Lemma \ref{lemma4.1}]
Since the $\tau$'s are identically distributed we have
\be\label{4.1}
\textstyle{\sum_{y\in \Z^d}\E Q_n^u(x,y)=\sum_{y\in \Z^d}\E Q_n^u(0,y)}, \quad \forall x\in B_{d_n}(t).
\ee
The statement of the lemma is thus equivalent to
\be\label{4.2}
\textstyle{\lim_{n\rightarrow\infty} k_n(t)\sum_{y\in \Z^d}\E Q_n^u(0,y) = t \nu(u,\infty)}, \quad \forall u>0, \ \forall t>0. 
\ee
In view of \eqref{3.24}, the sum in \eqref{4.2} is over $y\in B_{\theta_n}$. In fact, we can restrict it to $y\in B_{\theta_n}\setminus\{0\}$ because $\E Q_n(0,0)\leq \E(\1_{0\in T_n})= c_n^{-\a} \epsilon_n^{-\a}\ll k_n(t)$. Also, we have that
\be\label{4.2i}
\textstyle{\P(A_n)\E(\wt{\nu}_n^t(u,\infty)|A_n)\leq \E\wt{\nu}_n^t(u,\infty)\leq \P(A_n)\E(\wt{\nu}_n^t(u,\infty)|A_n)+ k_n(t)\theta_n^{d}\P(A_n^c)},
\ee
where $A_n$ is as in \eqref{3.7} and satisfies $k_n(t)\theta_n^{d}\P(A_n^c)\leq c_1 a_n \theta_n^{d-1} n^{-5}$. Therefore, it suffices to calculate $\E(\wt{\nu}_n^t(u,\infty)|A_n)$.  Let us also distinguish two cases depending on whether $d\geq 3$ or $d=2$.

\noindent \textbf{\emph{Case 1.}} Let $d\geq 3$ and take $y\in B_{\theta_n}\setminus\{0\}$. Set $k(\theta_n)=\theta_n(\log \theta_n)^{-1}$ and $h(\theta_n)=\theta_n-k(\theta_n)$. By the Markov property, writing $f_{\sigma(y)}$ for the density function of the hitting time, $\sigma(y)$, of $y$, we have on $A_n$
\bea
 \hspace{-1mm}&&\hspace{-1mm}\PP(\ell_{\theta_n}(y)\gamma_{n}(y)> u,\eta(B_{\theta_n})>\theta_n)\nonumber\\ \hspace{-1mm}
 &\geq&\hspace{-1mm}\textstyle{\int_0^{\theta_n} f_{\sigma(y)}(t)\PP_y (\ell_{\theta_n-t}(y)\gamma_{n}(y)> u)dt}-P(\eta(B_{\theta_n})\leq\theta_n)\nonumber\\
\hspace{-1mm}&\geq&\hspace{-1mm}\textstyle{ \int_0^{h(\theta_n)} f_{\sigma(y)}(t)dt\ \PP_y (\ell_{k(\theta_n)}(y)\gamma_{n}(y)> u)}-e^{-c_4 \theta_n^{1/2}}\nonumber\\
\hspace{-1mm}&=&\hspace{-1mm}\textstyle{P(\sigma(y)\leq h(\theta_n))\ \PP_y (\ell_{k(\theta_n)}(y)\gamma_{n}(y)> u) -e^{-c_4 \theta_n^{1/2}},}\label{4.4}
\eea
where we used \eqref{lemma321} of Lemma \ref{lemma3.2} in the second step. We first deal with the second probability in \eqref{4.4}. Setting $B_n^1\equiv B_{\sqrt{k(\theta_n)} (\log \theta_n)^{-2}}(y)$ we have,
\bea\label{4.5}
\PP_y (\ell_{k(\theta_n)}(y)\gamma_{n}(y)> u)
 &\geq& \PP_y (\ell_{\eta(B_n^1)}(y)\gamma_{n}(y)> u, \eta(B_n^1)< k(\theta_n))\nonumber\\
 &\geq& \PP_y (\ell_{\eta(B_n^1)}(y)\gamma_{n}(y)> u ) -  \PP_y (\eta(B_n^1)\geq k(\theta_n)). \quad \quad
\eea
By \eqref{lemma322} of Lemma \ref{lemma3.2}, on $A_n$, the second term in \eqref{4.5} is smaller than $e^{-c_4 (\log \theta_n)^2}$.
To bound the first term in \eqref{4.5} we use the well-know fact that when $J$ starts in $y$, $\ell_{\eta(B_n^1)}(y)$ is exponentially distributed.  
Let
\bea\label{4.7}
\textstyle{g_{B_n^1}(y)\equiv E_y\left[\int_0^{\eta(B_n^1(y))} \1_{J(s)=y} ds\right]}
\eea
denote the mean value of $\ell_{\eta(B_n^1}(y)$.
Eq. \eqref{4.4} then becomes
\be\label{4.8}
\textstyle{\PP_y (\ell_{k(\theta_n)}(y)\gamma_{n}(y)> u)\geq e^{- u\left(\gamma_{n}(y)g_{B_n^1}(y)\right)^{-1}} - e^{-c_4 (\log \theta_n)^2}},
\ee
and we get
\be\label{4.9}
Q_n^u(0,y)\geq P(\sigma(y)\leq h(\theta_n)) \bigl(e^{-u(\gamma_n(y) g_{B_n^1}(y))^{-1}} - e^{-c_4 (\log \theta_n)^2}\bigr)\1_{y\in T_n}.
\ee
To get an upper bound we write (using the Markov property)
\bea\label{4.10}
\textstyle{\PP (\ell_{\theta_n}(y)\gamma_{n}(y)> u,\eta(B_{\theta_n})>\theta_n)}&\leq&\textstyle{\PP (\ell_{\theta_n}(y)\gamma_{n}(y)> u) }\nonumber\\&=&\textstyle{\int_0^{\theta_n} f_{\sigma(y)}(t)\PP_y (\ell_{\theta_n-t}(y)\gamma_{n}(y)> u)dt}\nonumber\\
&\leq&P\left(\sigma(y)\leq \theta_n\right) \PP_y (\ell_{\theta_n}(y)\gamma_{n}(y)> u).\quad \quad
\eea
Set $B_n^2\equiv B_{\sqrt{\theta_n} \log \theta_n}(y)$. By \eqref{lemma321} of Lemma \ref{lemma3.2} we know on $A_n$ that $J$ exits $B_n^2$ before time $\theta_n$ with a probability smaller than $e^{-c_4(\log \theta_n)^2}$. Thus, proceeding as in \eqref{4.5}
\bea
Q_n^u(0,y)\leq P\left(\sigma(y)\leq \theta_n\right) \bigl(e^{-u(\gamma_n(y) g_{B_n^2}(y))^{-1}} + e^{-c_4 (\log \theta_n)^2}\bigr)\1_{y\in T_n}\label{4.12}
\eea
The contribution to $\E\wt \nu_n^t(u,\infty)$ coming from the error terms $\exp(-c_4(\log \theta_n)^2)$ in \eqref{4.9} and \eqref{4.12} is negligible because
\be\label{4.13}
\textstyle{k_n(t) \sum_{|y|\leq \theta_n} \E\bigl[\1_{y\in T_n}e^{-c_4(\log \theta_n)^2}\bigr]\ll e^{-c_4/2(\log \theta_n)^2}.}
\ee
To calculate $\E( Q_n^u(0,y)|A_n)$, we distinguish whether $\theta>0$ or $\theta=0$. In the first case several objects depend on the random environment: the distribution of $\sigma(y)$, the mean local time $g_{B_n^i}(y)$, and $\gamma_n(y)$. Thus we first seek upper and lower bounds on the distribution of $\sigma(y)$ and on $g_{B_n^i}(y)$ that are independent of $\gamma_n(y)$. Moreover, we look for upper and lower bounds for $g_{B_n^i}(y)$ that are independent of $n$. 

Let us begin with bounds for $P\left(\sigma(y)\leq \theta_n\right)$. We show now that we may approximate the distribution of $\sigma(y)$ by that of $\min_{y'\sim y}\sigma(y')$, which is independent of $\gamma_n(y)$. Since $y\neq 0$ we know that $\min_{y'\sim y} \sigma(y')\leq \sigma(y)$, implying that
\be\label{4.14}
\textstyle{P(\sigma(y)\leq \theta_n)\leq P(\min_{y'\sim y} \sigma(y')\leq \theta_n).}
\ee
By the definition of $T_n$ we know that all the traps in the neighborhood $y\in T_n$ have size smaller than $\epsilon_n^{-2/\a}$. 
This implies that, as soon as $J$ visits a neighbor $y'$ of $y$, it jumps to $y$ with probability larger than $1-2d (\epsilon_n^{-2/\a}c_n^{-1})^{\theta}$. This term goes to $1$ when $\theta>0$ and we get, for all $\varepsilon>0$ and $y'\sim y$, that
\bea
P_{y'}( \varepsilon < \sigma(y)\leq h(\theta_n)) \leq \epsilon_n^{-2\theta/\a} c_n^{-\theta}\ll c_n^{-\theta/2}.\label{4.16}
\eea 
Thus,
\be\label{4.17}
\textstyle{P\left(\sigma(y)\leq h(\theta_n)\right) \1_{y\in T_n}\geq \bigl(P\left(\min_{y'\sim y} \sigma(y')\leq h(\theta_n)\right)-c_n^{-\theta/2}\bigr)\1_{y\in T_n}.}
\ee
As in \eqref{4.13}, we see that the contribution of the error $c_n^{-\theta/2}$ to $\E\wt \nu_n^t(u,\infty)$ is of order $o(1)$. 

Let us now approximate $g_{B_n^i}(y)$ by random variables, $\wt g_{\infty}(y)$, that are independent of $\gamma_n(y)$. This approximation follows closely the ideas of \cite{BaCe11}. For $i=1,2$ we use the classical variational representation (see e.g. Chapter 3 in \cite{Bov09}) to write
\be\label{4.18}
\textstyle{(g_{B_n^i}(y))^{-1}=\inf\left\{\tfrac12 \sum_{ x\sim z}\wt \lambda(x,z)(f(x)-f(z))^2: f|_y=1, f|_{B_n^i}=0\right\}}, 
\ee
and define, setting $N(y)\equiv\{y': y'\sim y\}\cup \{y\}$,
\be\label{4.19}
\textstyle{(\wt g_{B_n^i}(y))^{-1}\equiv\inf\left\{\tfrac12 \sum_{ x\sim z}\wt \lambda(x,z)(f(x)-f(z))^2: f|_{N(y)}=1, f|_{B_n^i}=0\right\}}.
\ee
As in the proof of Lemma 6.2 in \cite{BaCe11} one can show on $A_n$ that for all $\e>0$ there exists $N(\e)$ uniform in the realization of the random environment, such that for $n\geq N(\e)$,
\be\label{4.20}
\wt g_{B_n^i}(y)\leq g_{B_n^i}(y)\leq (1+\e) \wt g_{B_n^i}(y), \quad \forall y\in T_n \cap B_{\theta_n}.
\ee
Combining \eqref{4.17}, 
\eqref{4.20}, \eqref{4.9}, and \eqref{4.2i} we get that $\E\wt\nu_n^{t}(u,\infty)$ is bounded below by
\bea\label{4.21}
\textstyle{ k_n(t)\sum_{|y|\leq\theta_n}\E\bigl[P(\min_{y'\sim y}\sigma(y')\leq h(\theta_n))e^{-u(\gamma_n(y) \wt g_{B_n^1}(y))^{-1}}\1_{y\in T_n}|A_n\bigr]-o(1).}
\eea
Similarly, we obtain by \eqref{4.17}, \eqref{4.20}, and \eqref{4.12} that  $\E\wt\nu_n^{t}(u,\infty)$ is smaller than
\bea\label{4.22}
\textstyle{ k_n(t)\sum_{|y|\leq\theta_n}\E\bigl[P(\min_{y'\sim y}\sigma(y')\leq \theta_n)e^{-u(1-\e_n)(\gamma_n(y) \wt g_{B_n^2}(y))^{-1}}\1_{y\in T_n}|A_n\bigr]+o(1).}
\eea
Let $g_{\infty}(y)=\lim_{n\rightarrow \infty} g_{B_n^1}(y)= \lim_{n\rightarrow \infty} g_{B_n^2}(y)$. As in the proof of Lemma 3.5 in \cite{BaCe11}, one can show that, on the event $A_n$, for all $\e'>0$,  there exists $N(\e')$, uniform in the random environment, such that for $n\geq N(\e')$ we have 
\be\label{4.23}
(1-\e')g_{\infty}(y)\leq g_{B_n^i}(y)\leq g_{\infty}(y), \quad \quad \forall y\in B_{d_n(t)}.
\ee
This with \eqref{4.20} implies that for all $\e''>0$ there exists $N(\e'')$ such that for $n\geq N(\e'')$, for all $y\in T_n \cap B_{\theta_n}$, $(1-\e'')\wt g_{\infty}(y)\leq g_{B_n^i}(y)\leq (1+\e'') \wt g_{\infty}(y)$, 
where $\wt g_{\infty}(y)=\lim_{n\rightarrow \infty} \wt g_{B_n^1}(y) = \lim_{n\rightarrow \infty} \wt g_{B_n^2}(y)$. Equipped with \eqref{4.21} and \eqref{4.2i} 
we take expectation with respect to $\gamma_n(y)$ and obtain that $\E\wt \nu_n^t(u,\infty)$ is bounded below by 
\bea\label{4.25}
\textstyle{\frac{t(1-\e'')^{\a}\Gamma(1+\a, \epsilon_n)}{u^{\a}\theta_n} \sum_{|y|\leq\theta_n}\E\bigl[\wt g_{\infty}^\a(y) P(\min_{y'\sim y}\sigma(y')\leq h(\theta_n))(1-\1_{T_n(y)^c})\bigr]}-o(1),\quad
\eea
where $T_n(y)^c=\{\max_{y'\sim y}\tau(y')> \epsilon_n^{-2/\a}\}$. The contribution to \eqref{4.25} coming from $T_n(y)^c$ is of order $o(1)$:  using \eqref{4.20}, and \eqref{theorem311} of Theorem \ref{theorem3.1}, and proceeding as in \eqref{4.14}, it is by Lemma \ref{lemma3.3} smaller than $c_1c_5 4d^2\epsilon_n^{2}$. 
Also, as in the proof of Lemma \ref{lemma3.3} one sees that adding $|y|>\theta_n$ in \eqref{4.25} produces at most an error of the order of $e^{-c_4/2\theta_n}$, and so
\bea\label{4.26}
\E\wt \nu_n^t(u,\infty)\geq\frac{t(1-\e'')^{\a}\Gamma(1+\a, \epsilon_n)}{u^{\a}\theta_n} \sum_{y\in \Z^d}\E\bigl[\wt g_{\infty}^\a(y) P(\min_{y'\sim y}\sigma(y')\leq h(\theta_n))\bigr]-o(1).\quad
\eea
Similarly, 
\bea\label{4.27}
\E\wt \nu_n^t(u,\infty)\leq\frac{t\Gamma(1+\a)}{u^{\a}\theta_n} \sum_{y\in \Z^d}\E\bigl[\wt g_{\infty}^\a(y) P(\min_{y'\sim y}\sigma(y')\leq \theta_n)\ \bigr]+o(1).
\eea
Since $\epsilon_n\rightarrow 0$, $\Gamma(1+\a,\epsilon_n)\rightarrow \Gamma(1+\a)$. It remains to establish that
\be\label{4.28}
\textstyle{\lim_{n\rightarrow\infty}\theta_n^{-1}\sum_{y\in \Z^d}\E\bigl[ \wt g_{\infty}^\a(y)P(\min_{y'\sim y}\sigma(y')\leq h_i(\theta_n))\bigr] =K}, \quad \mbox{ for } i=1,2,
\ee
where $h_1(\theta_n)=\theta_n$ and $h_2(\theta_n)=h(\theta_n)$. Since $h_2=h_1-o(h_1)$, we only present the proof for $h_1$.  For $\b\in[0,1]$, set 
$f_{n}^\b(x)\equiv \sum_{y\in \Z^d}\E\bigl[ (\wt g_{\infty}(y)/c_6)^\b P(\min_{y'\sim y}\sigma(y')\leq n)\bigr]$, where $c_6\in(0,\infty)$ is such 
that $c_6\geq g_{\infty}(y)\geq \wt g_{\infty}(y)$ for all $y\in \Z^d$. 
Using a 'quasi' sub-additivity argument (see \eqref{4.28aa} below), we now establish that $\lim_{n\rightarrow\infty}f_n^\b/n=K'$ where 
$K'=\inf\{f_n^\b/n: n\in \N\}\in(0,\infty)$. First note that by Lemma \ref{lemma3.3} $ f_n^\b/n\leq f_n^0/n\leq 2d c_5 $, and so, $K'<\infty$. To see 
that $K'>0$, we use that $f_m^\b\geq f_m^1$ and bound $f_m^1$ from below:
\bea\label{4.30}
\textstyle{ \sum_{y\in \Z^d}g_{\infty}(y) P(\sigma(y)\leq m)\geq \sum_{y\in \Z^d}E\ell_{m}(y)\geq m}
\eea
As in the proof of \eqref{4.20} one can show that $\wt g_{\infty}(y)\geq (2d)^2 g_{\infty}(y)$, and hence $f_m^1\geq (2d)^{-2}m$, which proves that 
$K'>0$. Let us now assume that for all $\e>0$ there exists $N$ large enough such that for all $n,m\geq N$,
\be\label{4.28aa}
f_{n+m}^\b\leq(1+\e)  f_{m}^\b +f_{n}^\b.
\ee
Then convergence to $K'$ follows. Indeed, by construction of $K'$ there exists $M$ such that $f_{M}^\b/M< K'+\e/2$. Now, let $N^\star=N'M$, $N'\geq N$, be such that $f_{2M}^\b/N^\star<\e/2$. For $n\geq N^\star$ write $n=sM+r\geq N^\star$ where $s\geq N$, $r\leq M$. Then, 
by \eqref{4.28aa},
\be\label{4.28a}
f_{n}^\b/n\leq (1+\e)\tfrac{s-1}{n} f_{M}^\b + f_{M+r}^\b/n\leq (1+\e) \tfrac{s-1}{s+2}f_M^\b/M + f_{2M}^\b/N^\star\leq (1+2\e) K'-\e.
\ee
Thus, $f_n^\b/n$ converges to $K'$ because by construction, $f_n^\b/n\geq K'$. 
It remains to establish the claim of \eqref{4.28aa}. The difference $f_{n+m}^\b-f_{n}^\b$ is equal to
\bea
\hspace{-4mm}&&\hspace{-2mm}\textstyle{\sum_{y\in \Z^d} \E\bigl[(\tfrac{\wt g_{\infty}(y)}{c_6})^\b P(\min_{y'\sim y}\sigma(y')\in(n,n+m))\bigr]}\nonumber\\
\hspace{-4mm}&\leq&\hspace{-2mm}\textstyle{\sum_{z,y}\E\bigl[(q_n(z)-\E q_n(z))(\tfrac{\wt g_{\infty}(y)}{c_6})^\b P_z(\min_{y'\sim y}\sigma(y')\leq m)\bigr] + f_{m}^\b}.\quad\quad\label{4.28aaa}
\eea
The first summand in the right hand side of \eqref{4.28aaa} is smaller than $\e f_m^\b$ if
\be\label{4.30aa}
\textstyle{\e_m^\b\equiv\sum_{z}\sum_y \E\bigl[|q_n(z)-\E q_n(z)|P_z(\sigma(y)\leq m)\bigr]\leq \e m}.
\ee
We divide the sum into $z\in B_{n^{1/2}/\e'}$ and $z\notin B_{n^{1/2}/\e'}$. Let $z\in B_{n^{1/2}/\e'}$. From the proof of Lemma \ref{lemma3.3} we 
know that there exists $c''\in(0,\infty)$ such that
\bea\label{4.326}
\textstyle{\sum_{y\in \Z^d}P_z(\sigma(y)\leq m))}\hspace{-2mm}&\leq&\hspace{-2mm} \textstyle{c''\sum_{|z-y|\leq m^{1/2}\log m} |y|^{2-d} e^{-c_4/2|y|^2/m} (g_{\infty}(y))^{-1}}\nonumber\\
\hspace{-2mm}&+&\hspace{-2mm}\textstyle{m^{(d+1)/2}\1_{A_m^c(z)}+\sum_{|z-y|>m^{1/2}\log m}\1_{A_{|z-y|^2}^c(z)}},\quad\quad
\eea
where $A_m(z)$ is as in \eqref{3.7}. 
Let us first control the contribution to $\e_m^\b$ coming from the first sum in \eqref{4.326}. We bound $1/g_{\infty}(y)\leq 1/\e'+(g_{\infty}(y))^{-1}\1_{g_{\infty}(y)<\e'}$ and call $(I)$, respectively $(II)$ the contribution  to $\e_m^\b$ coming from $1/\e'$, respectively $(g_{\infty}(y))^{-1}\1_{g_{\infty}(y)<\e'}$. Now, 
\be\label{4.32aa}
\textstyle{(I) = (c''m)/\e'\sum_{|z|\leq n^{1/2}/\e'}\E\bigl[|q_n(z)-\E q_n(z)|/\E q_n(z)\bigr] \E q_n(z)}.
\ee
Since $q_n(z)\leq c_1 n^{-d/2}$, the contribution to $(I)$ from $z\in B_{\e'n^{1/2}}$ is smaller than $(\e')^{d-1} m$. 
By \eqref{theorem315} of Theorem \ref{theorem3.1} for $z\in B_{n^{1/2}/\e'}\setminus B_{\e'n^{1/2}}$, $\E\bigl[|q_n(z)-\E q_n(z)|/\E q_n(z)\bigr]$ tends uniformly to zero, and so $(I)$ is bounded above by $\e m$.
 Also,
\be\label{4.31aa}
\textstyle{
(II)\leq c_1 m(\e')^{-d}c\E\bigl[(g_{\infty}(0))^{-1}\1_{g_{\infty}(0)<1/\e'}\bigr]\leq c_1 m(\e')^{-d}\exp(-c/\e'^{1/3})<\e m },
\ee
where we used $g_{\infty}(0)\geq U_0^{1/(2-d)}$ and \eqref{theorem312}. We now bound the contribution to $\e_m^\b$ coming from the second line in \eqref{4.326}. By  \eqref{theorem311} of Theorem \ref{theorem3.1}, $|q_n(z)-\E q_n(z)|\leq n^{-d/2}$ for all $z\in  B_{n^{1/2}/\e'}$. By the identical distribution of the $U$'s and since $| B_{n^{1/2}/\e'}|n^{-d/2}\leq \e'^{-d/2}$, it suffices to prove that $m^{(d+1)/2}\P(A_m^c) +\sum_{|y|>m^{1/2}\log m}\P(A_{|y|^2}^c)$ vanishes as $m\rightarrow \infty$. This follows by definition of $A_n$ and \eqref{theorem312}. Finally, let $z\notin B_{n^{1/2}/\e'}$. 
By Cauchy-Schwarz inequality, the contribution to $\e_m^\b$ coming from ${z\notin B_{n^{1/2}/\e'}}$  is bounded above by
\be\label{4.30aaa}
\textstyle{\sum_{|z|>n^{1/2}/\e'}(\E|q_n(z)|^2)^{1/2}\sum_{y}( \E(P(\sigma(y)\leq m))^2)^{1/2}}.
\ee
By \eqref{theorem312} and \eqref{theorem313} of Theorem \eqref{theorem3.1}, $\E|q_n(z)|^2\leq n^{-d}e^{-2c_4 |z|^2/n^2}$. One can check that the sum over $z\notin B_{n^{1/2}/\e'}$ of $n^{-d/2}e^{-c_4 |z|^2/n^2}$ is bounded above by $e^{-c_4/\e'}$. It remains to establish that the sum over $y$ in \eqref{4.30aaa} is of smaller order than $m$. This follows from \eqref{4.326} and \eqref{theorem312}. This proves \eqref{4.28aa}. 
Thus \eqref{4.28} holds, and so,
\be\label{4.38}
u^{-\a} t K \Gamma(1+\a)(1+\e)^\a + o(1)\geq \E\wt\nu_n^t(u,\infty)\geq u^{-\a} t K \Gamma(1+\a, \epsilon_n) (1-\e)^\a- o(1).
\ee
Since $\e>0$ is arbitrary we see that $\lim_{n\rightarrow\infty}\E \wt \nu_n^t(u,\infty) =  u^{-\a} t {\cal K}$, where ${\cal K}\equiv K \Gamma(1+\a)$. This concludes the proof of Lemma \ref{lemma4.1} for $d\geq 3$ and $\theta>0$. 
When $\theta=0$, the proof simplifies because $J$ is independent of the random environment. More precisely, it suffices to use Lemma 3.5 in \cite{BaCe11} to replace $g_{B_n^i}(y)$ by $g_{\infty}(y)$ to get 
\bea\label{4.39}
\tfrac{t\Gamma(1+\a, \epsilon_n)}{u^{\a}\theta_n}\hspace{-3mm}&&\hspace{-3mm}\textstyle{\sum_{y\in \Z^d}g_{\infty}(y)^\a P\left(\sigma(y)\leq h(\theta_n)\right) - o(1) }\nonumber\\
\hspace{-3mm}&&\hspace{-3mm}\leq \E\wt \nu_n^t(u,\infty)\leq\textstyle{\frac{t\Gamma(1+\a)(1+\e')^{\a}}{u^{\a}\theta_n} \sum_{y\in \Z^d} g_{\infty}^\a(y)P(\sigma(y)\leq \theta_n) }+o(1).
\eea
By the same arguments as for $\theta>0$, we can show that both bounds converge to $u^{-\a} t {\cal K}$ as $n\rightarrow\infty$. This finishes the proof of Lemma \ref{lemma4.1} for $d\geq 3$.

\noindent \textbf{\emph{Case 2.}} Let $d=2$. The pattern of proof is similar to that of Case 1 and relies on \eqref{4.4} and \eqref{4.10}. The difference lies in the behavior $g_{B_n^i}(y)$. By definition in \eqref{4.7}, on $A_n$,
\be\label{4.39a}
\textstyle{g_{B_n^i}(y)=\int_0^\infty P_y(J(t)=y, \eta(B_n^1)>t)dt \geq \int_{\sqrt{\theta_n}}^{\theta_n}P_y(J(t)=y)dt -e^{-c_4(\log\theta_n)^2},}
\ee
where we used \eqref{lemma321} of Lemma \ref{lemma3.2}. By \eqref{theorem314} of Theorem 3.1, the integral in the right hand side of \eqref{4.39a} is  larger than $c_1/2 \log \theta_n$, showing that $g_{B_n^i}(y)$ diverges as $n\rightarrow\infty$. Thus, instead of substituting $\wt g_{\infty}(y)$ for $g_{B_n^i}(y)$ we use \eqref{4.20} to approximate $g_{B_n^i}(y)$ by $\wt g_{B_n^i}(y)$ for $n$ large enough. A number of results from \cite{Ce11} will allow us to deal with $\wt g_{B_n^i}(y)$.
 
Let us begin with the construction of a lower bound on $\E\wt \nu_n^t(u,\infty)$. We deduce from \eqref{4.17}, that bounding $P(\sigma(y)\leq h(\theta_n))\geq P(\min_{y\sim y'}\sigma(y')\leq h(\theta_n))$ for $y\in T_n$, produces in $\E\wt \nu_n^t(u,\infty)$ an error of the order $c_n^{-\e}$ for $\e>0$. We use \eqref{4.20} to substitute $\wt g_{B_n^1}(y)$ for $ g_{B_n^1}(y)$. Since $P(\min_{y'\sim y}\sigma(y')\leq h(\theta_n))$ and $\wt g_{B_n^1}(y)$ are independent of $\gamma_n(y)$, we can proceed as in Case 1 and take expectation with respect to $\gamma_n(y)$. Doing this yields
\bea\label{4.47}
\E\wt \nu_n^t(u,\infty)\hspace{-2mm}&\geq&\hspace{-2mm}\frac{t\log \theta_n \Gamma(1+\a,\epsilon_n)}{(u\log\theta_n)^{\a}\theta_n}\sum_{|y|\leq \theta_n}\E\bigl[\wt g_{B_n^1}^\a(y)P\bigl(\min_{y'\sim y}\sigma(y')\leq h(\theta_n)\bigr)|A_n\bigr]-o(1)\nonumber\\
\hspace{-2mm}&\geq&\hspace{-2mm}\frac{t\log \theta_n \Gamma(1+\a,\epsilon_n)}{(u\log\theta_n)^{\a}\theta_n}\sum_{|y|\leq \theta_n}\E\bigl[\wt g_{B_n^1}^\a(y)P(\sigma(y)\leq h(\theta_n))|A_n\bigr]-o(1),
\eea
since $\min_{y'\sim y}\sigma(y)\leq \sigma(y)$. We now construct an upper bound on $\E\wt \nu_n^t(u,\infty)$. Again, by \eqref{4.20} and since $\min_{y'\sim y}\sigma(y)\leq \sigma(y)$, $\E\wt \nu_n^t(u,\infty)$ is bounded above by
\bea\label{4.40}
\textstyle{\tfrac{t\log \theta_n \Gamma(1+\a)(1+\e)}{(u\log\theta_n)^{\a}\theta_n}\sum_{|y|\leq \theta_n}\E\bigl[g_{B_n^2}^\a(y)}P(\min_{y'\sim y}\sigma(y')\leq \theta_n)|A_n \bigr]+o(1). \quad
\eea
We show now that, up to a negligible error term, we may substitute $P(\sigma(y)\leq \theta_n)$ for $P(\min_{y'\sim y}\sigma(y')\leq \theta_n)$ for all $y\in B_{\theta_n}$. To see this note for $y\in B_{\theta_n}$
\be\label{4.41}
\textstyle{P(\min_{y'\sim y}\sigma(y')\leq \theta_n)\leq P(\sigma(y)\leq \theta_n) +\sum_{y'\sim y}P(\sigma(y')\leq \theta_n<\sigma(y))}.
\ee
By the Markov property, for all $y'\sim y$, $P(\sigma(y')\leq \theta_n<\sigma(y))$ is bounded above by
\bea
&&\textstyle{P(\sigma(y')\leq h(\theta_n))P_{y'}(\sigma(y)>k(\theta_n))+P(\sigma(y')\in (h(\theta_n),\theta_n))}\nonumber\\
&=& \d_n^1(y')+\d_n^2(y').\label{4.42}
\eea
By \eqref{4.2i} it thus suffices to establish that
\be\label{4.43}
\textstyle{\theta_n^{-1}(\log\theta_n)^{1-\a}\sum_{|y|\leq \theta_n}\sum_{y'\sim y}\E\bigl[ \wt g_{B_n^1}^\a(y)\bigl(\d_n^1(y')+\d_n^2(y')\bigr)| A_n\bigr]=o(1)}.
\ee
As in Lemma 3.3 in \cite{Ce11} one can show on $A_n$ that there exists $c_{9}\in(0,\infty)$ such that, for all $y\in B_{d_n(t)}$, 
$\wt g_{B_n^i}(y)\leq c_{9} \log \theta_n$ for $i=1,2$. By \eqref{lemma332} of Lemma \ref{lemma3.3}, 
$\sum_{y,y'\sim y} \d_n^2(y')\leq c_5\theta_n/(\log\theta_n)^2$. Hence the contribution  to the left hand side of \eqref{4.43} coming from $\d_n^2$ is of order $o(1)$. 
To see that the same is true for $\d_n^1$, we use \eqref{lemma332} of Lemma 
 \ref{lemma3.3} to bound 
 \be\label{4.44}
\E[P(\sigma(y')\leq h(\theta_n)) P_{y'}(\sigma(y)>k(\theta_n))|A_n]\leq f_{h(\theta_n)}(|y'|)\E[P(\sigma(x)>k(\theta_n))|A_n],\quad
\ee
where $|x|=1$. By recurrence and irreducibility, $\E[ P(\sigma(x)>k(\theta_n))|A_n]\leq \e$ for $n$ large enough 
 and \eqref{4.44} implies that $\sum_{y} \d_n^1(y)\leq c_5\e$. This concludes the proof of \eqref{4.43}. Finally, combining \eqref{4.40}-\eqref{4.44},
\be\label{4.46}
\E\wt \nu_n^t(u,\infty)\leq \textstyle{\frac{t\log \theta_n \Gamma(1+\a)(1+\e)}{(u\log\theta_n)^{\a}\theta_n}\sum_{|y|\leq \theta_n}\E\bigl[g_{B_n^2}^\a(y)P(\sigma(y)\leq \theta_n) |A_n\bigr]+o(1).}
\ee
We now show that \eqref{4.47} and \eqref{4.46} tend to the same limit ${\cal K}t u^{-\a}$. By Proposition 3.1 in \cite{Ce11} we know that there exists $\bar K$ such that, as $r\rightarrow\infty$, $(\bar K\log r)^{-1}\wt g_{B_r^{1/2}(0)}(0)$ converges $\P$-a.s.~to one for $i=1,2$. Thus, $\P$-a.s.,
\be\label{4.49a}
\textstyle{\lim_{n\rightarrow\infty}(\bar K\log\theta_n)^{-1}\wt g_{B_n^1}(0)
=\lim_{n\rightarrow\infty}(\bar K\log\theta_n)^{-1}\wt g_{B_n^2}(0)=1}.
\ee
For $\varepsilon>0$ define ${\cal B}_n(y)\equiv\{|(\bar K\log \theta_n)^{-1}g_{B_n^2}(y)-1|\leq \varepsilon\}$. Then,
\be\label{4.49}
\eqref{4.46}\leq\frac{t\log \theta_n(1+\e)^{\a}{\cal K'}}{u^{\a}\theta_n}\sum_{y\in B_{\theta_n}}\E\bigl[P(\sigma(y)\leq \theta_n)((1+\e)+ c_{9}^\a/{\cal K'} \1_{{\cal B}_n^c(y)})|A_n\bigr],
\ee
where ${\cal K'}\equiv\Gamma(1+\a) \bar K^\a$ and where we used that $\wt g_{B_n^1}(y)\leq c_9\log\theta_n$. As in Lemma 3.3 in \cite{Ce11}, on $A_n$, we have that $\wt g_{B_n^1}(y)\geq c_8\log\theta_n$, for all $y\in B_{\theta_n}$. Hence, we can bound \eqref{4.47} from below in a similar way. Thus, convergence of \eqref{4.47} and \eqref{4.46} follows if we can establish that
\bea\label{4.50}
&&\textstyle{\lim_{n\rightarrow\infty}\theta_n (\log\theta_n)^{-1}\E[ ER_{\theta_n}|A_n]=\bar K^{-1}},\\
&&\textstyle{\lim_{n\rightarrow\infty}\theta_n (\log\theta_n)^{-1}\sum_{y\in B_{\theta_n}}\E\bigl[P(\sigma(y)\leq \theta_n)\1_{{\cal B}_n^c(y)}|A_n\bigr]=0},\label{4.50a}
\eea
where $R_{\theta_n}$ is defined in \eqref{3.10}. Let us first prove \eqref{4.50a}. By \eqref{lemma332} of Lemma \ref{lemma3.3} and \eqref{4.2i}
\bea
\textstyle{\sum_{y}\E\bigl[P(\sigma(y)\leq \theta_n)\1_{{\cal B}_n^c(y)}|A_n\bigr]\leq \sum_{y}f_{\theta_n}(|y|)\P\bigl({\cal B}_n^c(y)\bigr)+\e\leq \frac{\theta_n}{\log\theta_n} \P\bigl({\cal B}_n^c\bigr)},
\eea
where we used that the $\tau$'s are i.i.d. By Proposition 3.1 in \cite{Ce11},  $\P({\cal B}_n^c)\rightarrow 0$, and so \eqref{4.50a} holds.
Let us now construct upper and lower bounds for $\theta_n (\log\theta_n)^{-1}\E[ ER_{\theta_n}|A_n]$ that coincide in the limit. We begin with the lower bound. By the Markov property,
\bea\label{4.51}
\textstyle{\theta_n=\sum_{y\in B_{\theta_n}} E \ell_{\theta_n}(y)\leq \sum_{y\in B_{\theta_n}}  P(\sigma(y)\leq\theta_n) E_y \ell_{\theta_n}(y).}
\eea
To bound $\E[ ER_{\theta_n}|A_n]$ from below it suffices to construct an upper bound on $E_y\ell_{\theta_n}(y)$. By Theorem \ref{theorem311} one can show on $A_n$ that for all $y\in B_{d_n(t)}$, $E_y\ell_{\theta_n}(y)\in (c_8\log\theta_n, c_9\log \theta_n)$, yielding
\bea
\theta_n\leq  \log\theta_n\textstyle{\sum_{|y|\leq\theta_n}\E\bigl[  P(\sigma(y)\leq\theta_n>0)\bigl(\bar K (1+\varepsilon)+ c_8 \1_{{\cal B}_n^c(y)}\bigr)|A_n\bigr]+o(1).}\label{4.52}
\eea
Together with \eqref{4.50a},
\be\label{4.54}
1\leq \theta_n^{-1}\log \theta_n\bar K(1+\varepsilon) \E[ E R_{\theta_n}|A_n] +c_8\P({\cal B}_n^c),
\ee 
i.e.~$\lim_{n\rightarrow\infty}\theta_n^{-1}\log \theta_n\E[ E R_{\theta_n}|A_n]$ is bounded below by ${\bar K}^{-1}$.
For the upper bound we again use the Markov property and get that
 \bea \label{4.55}
\textstyle{\theta_n+k(\theta_n)=\sum_{y\in B_{\theta_n}} E \ell_{\theta_n+k(\theta_n)}(y)\geq\sum_{y\in B_{\theta_n}}P(\ell_{\theta_n}(y)>0) E_y \ell_{k(\theta_n)}(y).}
\eea
Since $k(\theta_n)\log\theta_n/\theta_n\rightarrow 0$, we can show that the upper bound coincides with the lower bound. The claim of \eqref{4.50} is proved. Finally, using \eqref{4.50} and \eqref{4.50a} in \eqref{4.47} and \eqref{4.46}, 
\be\label{4.57}
\textstyle{{\cal K}u^{-\a} t(1-\e)^{1+\a}\leq\lim_{n\rightarrow\infty}\E\wt\nu_n^t(u,\infty)\leq {\cal K}u^{-\a} t(1+\e)^{1+\a},}
\ee
where ${\cal K}={\cal K'}\bar K^{-1}$. Since $\e>0$ is arbitrary this proves the convergence of $\E\wt \nu_n^t(u,\infty)$. This finishes the proof of Lemma \ref{lemma4.1} for $d=2$.
\end{proof}
\subsection{Convergence of $\wt \nu_n^t(u,\infty)$}\label{S4.2}
Let $u>0$, $t>0$.  In this section we prove that $\P$-a.s., $\lim_{n\rightarrow\infty} \wt \nu_n^t(u,\infty)=t\nu(u,\infty)$. Once more, let us consider subsequences $c_{N,r}=\exp(N+r)$, $r\in [0,1)$ (see the paragraph below \eqref{3.19a}). For $r,s\in[0,1)$, we define
\be\label{qnrs}
\textstyle{Q_{N,r,s}(x)\equiv \sum_y Q_{N,r,s}(x,y)},
\ee
where we write, for  $x,y\in \Z^d$, 
\be\label{4.59}
\textstyle{ Q_{N,r,s}(x,y)\equiv\PP_x(\ell_{\theta_{N,s}}(y)\gamma_{N,r}(y)>u, \eta(B_{\theta_{N,r}})\leq \theta_{N,s})\1_{y\in T_{N,r}}, \quad}
\ee
where $T_{N,r}$ is defined above \eqref{3.17}. For $i=0,\ldots, N$, we set $r_i\equiv i/ N$ and define
\bea
\nu_{N,i}^{1}\hspace{-2mm}&\equiv&\hspace{-2mm}\textstyle{ \sum_{x\in B_{d_{N,r_{i}}}}  Q_{N,r_{i+1},r_i,} (x)\inf_{r\in {\cal I}_i}\overline \pi_{N,r_i}^t(x)},\\
\nu_{N,i}^2\hspace{-2mm}&\equiv&\hspace{-2mm}\textstyle{\sum_{x\in B_{d_{N,r_{i+1}}}}  Q_{N,r_i,r_{i+1}} (x)\sup_{r\in {\cal I}_i}\overline \pi_{N,r_i}^t(x).}
\eea
Then we have that, for $r\in {\cal I}_i\equiv [r_i,r_{i+1})$,
\bea\label{431}
\textstyle{\nu_{N,i}^1\leq\wt\nu_{N,r}^t(u,\infty)\leq\nu_{N,i}^2.}
\eea
To prove that $\lim_{n\rightarrow\infty} \wt \nu_n^t(u,\infty)=t\nu(u,\infty)$~$\P$-a.s., we use \eqref{431} 
 as follows. First, we will derive  from Lemma \ref{lemma4.1} that,  for $j=1,2$, $\E\nu_{N,i}^j$ converges to $t\nu(u,\infty)$. Then, we will prove in Lemma \ref{lemma4.2} that, for $d\geq 3 $, respectively, $d=2$, the second and fourth moments of $\nu_{N,i}^j-\E \nu_{N,i}^j$ are bounded above by $N^{-3}$. Together with \eqref{431}, we thus get by Chebyshev's inequality that
\bea
\hspace{-2mm}&&\hspace{-2mm}\textstyle{\P(\sup_{r\in[0,1)} |\wt \nu_{N,r}^t(u,\infty)-t\nu(u,\infty)|>\e)}\nonumber\\
\hspace{-2mm}&\leq&\hspace{-2mm} \textstyle{\sum_{i=0}^{ N-1} \bigl\{\P(|\nu_{N,i}^1 -\E\nu_{N,i}^1|>\e/2)+\P(|\nu_{N,i}^2 -\E\nu_{N,i}^2|>\e/2)\bigr\}\leq 32 \e^{-4} N^{-2},}\quad
\eea
which is summable in $N$ and therefore proves that $\P$-a.s., $\lim_{n\rightarrow\infty} \wt \nu_n^t(u,\infty)=t\nu(u,\infty)$.

Fix $j\in\{1,2\}$. Let us derive from Lemma  \ref{lemma4.1} that $\lim_{N\rightarrow\infty}\E\nu_{N,i}^j=t\nu(u,\infty)$. By definition of $r_i$ this lemma implies that $k_{N,r_{i}}(t)\E( Q_{N,r_{i+1},r_i,} (0))$ and $ k_{N,r_{i+1}}(t)\E( Q_{N,r_i,r_{i+1}} (0))$ both converge to $t\nu(u,\infty)$. Convergence of $\E \nu_{N,i}^j$  to the same quantity will follow if we can establish that the following vanishes as $N\rightarrow\infty$
\be\label{ulb}
\textstyle{\sup_{k=1, \ldots, k_{N,r_{i}}(t)}\sum_{x\in B_{d_{N,r_{i+1}}}}\sup_{r,s\in{\cal I}_i}\E |P(J(k\theta_{N,r})=x)-P(J(k\theta_{N,s})=x)|}.
\ee
Fix $k\in1, \ldots, k_{N,r_{i}}(t)$. We divide $B_{d_{N,r_{i+1}}}$ into $B_{(\e k\theta_{N,1})^{1/2}}$,  $B_{d_{N,r_{i+1}}}\setminus B_{(k\theta_{N,0}/\e)^{1/2}}$, and $B_{(k\theta_{N,0}/\e)^{1/2}}\setminus B_{(\e k\theta_{N,1})^{1/2}}$ and construct bounds that are uniform in $k$. By \eqref{theorem311} of Theorem \ref{theorem3.1}, the contribution to \eqref{ulb}  coming from $B_{(\e k\theta_{N,1})^{1/2}}$ is bounded above by $\e^d$. By \eqref{theorem312} and \eqref{theorem313} of Theorem \ref{theorem3.1}, the contribution  to \eqref{ulb} coming from $B_{d_{N,r_{i+1}}}\setminus B_{(k\theta_{N,0}/\e)^{1/2}}$ is smaller than $e^{-c_4/\e}$. By definition of $r_i$, dominated convergence, and the uniform local central limit theorem (\eqref{theorem315} of Theorem \ref{theorem3.1}), the contribution to \eqref{ulb} coming from $B_{(k\theta_{N,0}/\e)^{1/2}}\setminus B_{(\e k\theta_{N,1})^{1/2}}$ tends to zero. This proves that \eqref{ulb} vanishes as $N\rightarrow\infty$.

The following lemma bounds the second and fourth moment of  $\nu_{N,i}^j-\E \nu_{N,i}^j$ for $d\geq 3$, respectively, $d=2$.
\begin{lemma}\label{lemma4.2}
Let $i\in \{0,\ldots,N-1\}$, $j\in \{1,2\}$. For $d\geq 3$, for all $u>0$, $t>0$, for $N$ large enough,  
  \be\label{lemma421}
\E(\nu_{N,i}^j - \E\nu_{N,i}^j)^2\leq K(u,t)(\log\theta_{N,1})^2\theta_{N,1}^{-1/2},
 \ee
and for $d=2$,  for all $u>0$, $t>0$, for $N$ large enough, 
 \be\label{lemma4211}
\E(\nu_{N,i}^j - \E\nu_{N,i}^j)^4\leq K(u,t) u^{-\a}(\log \log a_{N,1})^6 (\log \theta_{N,0})^{-4},
 \ee
where $K(u,t)\in(0,\infty)$.
\end{lemma} 
The proof of Lemma \ref{lemma4.2} relies on Lemma \ref{lemma4.3} below. We first prove Lemma \ref{lemma4.3} and Lemma \ref{lemma4.2} next.
\begin{lemma}\label{lemma4.3}
For all $u>0$, for $k\in\{2,4\}$, for $r,s\in[0,1)$, for $N$ large enough, 
\be\label{lemma431}
\E(Q_{N,r,s} (0))^k\leq\textstyle{ \sum_{x}\E (Q_{N,0,1} (0,x))^k +c_{N,0}^{-3\alpha/2} \leq \rho_N (d,k)\equiv K u^{-\a} \rho_N(d,k)},
\ee
where $K\in(0,\infty)$ and where we set
\be\label{lemma432}
\rho_N(d,k)\equiv \theta_{N,1} c_{N,0}^{-\alpha} (\log \theta_{N,0})^{-k+\alpha}\1_{d=2} +  
\theta_{N,1}^{1/2} c_{N,0}^{-\alpha}\1_{d\geq 3}.
\ee
\end{lemma}
\begin{proof}[Proof of Lemma \ref{lemma4.3}]
Let $k\in \{2,4\}$ and $r,s\in [0,1)$. Since $Q_{N,r,s} (x)\leq Q_{N,0,1} (x)$,  we have by \eqref{4.59} and \eqref{qnrs} that
\be\label{4.62}
\textstyle{\E(Q_{N,r,s} (0))^k\leq \sum_{y}\E(Q_{N,0,1} (0,y))^k+\sum_{{y'\neq y}\atop{ |y'-y|\leq \theta_{N,1}}}\E(\1_{y',y\in T_{N,0}})}.\hspace{-5mm}\ \ 
\ee
The double sum in \eqref{4.62} is bounded above by
\bea
\textstyle{\sum_{{y'\neq y}\atop{  |y'-y|\leq \theta_{N,1}}} \E \1_{y,y'\in T_{N,0}}}
\leq\textstyle{\theta_{N,1}^{d}(c_{N,0}\epsilon_{N,1})^{-2\a}\leq c_{N,0}^{-3\alpha/2},}\label{4.63}\ 
\eea
where we used \eqref{1.44} and \eqref{1.45}. It remains to bound the first term in the right hand side of \eqref{4.62}. Since $\P(A_{N,1}^c) \theta_{N,1}^d\leq \rho_N(d,k)$, it suffices to bound $\E(Q_{N,0,1} (0,y)\1_{A_{N,1}})^k$. For $y\in B_{\theta_{N,1}}$ we know by \eqref{4.10} and \eqref{4.12} that $\E(Q_{N,0,1} (0,y)\1_{A_{N,1}})^k$ is smaller than
\bea
\hspace{-2mm}&&\hspace{-2mm} \textstyle{\E\bigl(\1_{y\in T_{N,0}}P_y\left(\ell_{\theta_{N,1}}(y)\gamma_{N,0}(y)> u\right)P(\sigma(y)\leq \theta_{N,1})\1_{A_{N,1}} \bigr)^k}\nonumber\\
\hspace{-2mm}&\leq&\hspace{-2mm}\textstyle{\E\bigl(\1_{y\in T_{N,0}}P_y\bigl(\ell_{B_{N,1}^2}(y)\gamma_{N,0}(y)> u\bigr)P(\sigma(y)\leq \theta_{N,1})\1_{A_{N,1}}\bigr)^k}
\nonumber\\
\hspace{-2mm}&+&\hspace{-2mm}e^{-c_4(\log \theta_{N,0})^2} \P(y\in T_{N,0}). \label{4.64}\quad \quad
\eea
The contribution to $\E(Q_{N,0,1} (0,y)\1_{A_{N,1}})^k$ of the second term in the right hand side of \eqref{4.64} is negligible because $\theta_{N,1}^d c_{N,0}^{-\a}\epsilon_{N,0}^{-\a}\ll \rho_N (d)$. 
It remains to bound the first summand in \eqref{4.64}. Recall from the proof of Lemma \ref{lemma4.1} that $\ell_{B_{N,1}^2}(y)$ has exponential distribution with mean value $g_{B_{N,1}^2}(y)$, which on $A_{N,1}$ is bounded above by $c( \log \theta_{N,1}\1_{d=2}+\1_{d\geq 3})$, for all $y\in B_{\theta_{N,1}}$. Also, recall that $P(\sigma(y)\leq\theta_{N,1})\leq P(\min_{y'\sim y}\sigma(y)\leq\theta_{N,1})$. Thus, for all $y\in B_{\theta_{N,1}}$
\bea\label{4.66}
\hspace{-2mm}&&\hspace{-2mm}\E \bigl(\1_{A_{N,1}}\PP(\ell_{B_{N,1}^2}(y)\gamma_{N,0}(y)>u)\1_{y\in T_{N,1}}\bigr)^k\nonumber\\
\hspace{-2mm}&\leq&\hspace{-2mm}\textstyle{ \E\Bigl(\1_{A_{N,1}}P(\min_{y'\sim y}\sigma(y')\leq \theta_{N,1})\exp(-u (\gamma_{N,0}(y)c(\1_{d\geq 3}+ \log \theta_{N,1}\1_{d=2}))^{-1}})\Bigr)^k\nonumber\\
\hspace{-2mm}&\leq&\hspace{-2mm} \textstyle{\E \exp(-ku (\gamma_{N,0}(0)c(\1_{d\geq 3}+ \log \theta_{N,1}\1_{d=2}))^{-1}) \E\bigl(P(\min_{y'\sim y}\sigma(y')\leq \theta_{N,1})\bigr)^k}.\quad\quad\ \ 
\eea
We bound the terms in \eqref{4.66} separately. The expectation with respect to $\gamma_{N,0}(0)$ is bounded above by $Cu^{-\a} c_{N,0}^{-\a}(\1_{d\geq 3}+\log \theta_{N,1}\1_{d=2})$, for some $C\in(0,\infty)$. Moreover,
\bea\label{4.68}
\textstyle{\sum_{y\in B_{\theta_{N,1}}}\E\bigl(P(\min_{y'\sim y}\sigma(y')\leq \theta_n)\bigr)^k}
\leq\textstyle{\sum_{y} c\E(\1_{A_{N,1}}P(\sigma(y)\leq \theta_{N,1}))^k+e^{-c'N}},\ \ 
\eea
where $c,c'\in(0,\infty)$ and where we used the definition of $A_{N,1}$. By Lemma \ref{lemma3.3}, \eqref{4.68} is bounded above by $c_5  \theta_{N,1}(\log\theta_{N,0})^{-k}$, for $d=2$. For $d\geq 3$ and $k=2$,  the same lemma implies that \eqref{4.68} is smaller than $c_5\theta_{N,1}^{1/2}$. Since the right hand side of  \eqref{4.68} is decreasing in $k$ the same is true for $k=4$. Collecting \eqref{4.66}-\eqref{4.68} yields
\be\label{4.69}
\textstyle{\sum_{y\in B_{\theta_{N,1}}}\E \bigl(\PP(\ell_{B_{N,1}^2}(y)\gamma_{N,0}(y)>u)\1_{y\in T_{N,0}}\bigr)^k\leq K u^{-\a}\rho_N(d,k)},
\ee
for some $K\in (0,\infty)$. This finishes the proof of Lemma \ref{lemma4.3}.
\end{proof}
We are now ready to present the proof of Lemma \ref{lemma4.2}.
\begin{proof}[Proof of Lemma \ref{lemma4.2}]
Without loss of generality, we prove Lemma \ref{lemma4.2} for $j=1$ only. We distinguish whether $d=2$ or $d\geq 3$. We begin with $d\geq 3$. 
By \eqref{defpi}, the variance of $\nu_{N,i}^1$ is bounded above by
\bea
\hspace{-2mm}&&\hspace{-2mm}\textstyle{k_{N,1}^2(t)\sum_{x\in B_{d_{N,1}(t)}}(\overline\pi_N^t(x))^2\E\bigl(Q_{N,0,1} (x)\bigr)^2 }\label{4.70}\\
\hspace{-2mm}&+&\hspace{-2mm}\textstyle{2 k_{N,1}^2(t)\sum_{ x\neq x', |x-x'|\leq 2\theta_{N,1}} \overline\pi_N^t(x)\overline\pi_n^t(x')\E\bigl(Q_{N,0,1} (x)Q_{N,0,1} (x')\bigr)\label{4.71}},\quad\quad
\eea
where we used that $Q_{N,0,1} (x)$ only depends on $\tau(x)$ for $x\in B_{\theta_{N,1}}(x)$. 
 By Lemma \ref{lemma4.3}, \eqref{4.70} is bounded above by\be\label{4.78}
\leq\textstyle{\sum_{x\in B_{d_{N,1}(t)}}(k_{N,1}(t)\overline\pi_N^t(x))^2\rho_N (d)\ll K(u,t)(\log\theta_{N,1})^{4}\theta_{N,1}^{-1/2}},
\ee
for $K(u,t)\in (0,\infty)$. Since this satisfies \eqref{lemma421}, it suffices to control \eqref{4.71}. 

 Fix $x,x'\in B_{d_{N,1}(t)}$. To bound  $ \E Q_{N,0,1} (x)Q_{N,0,1} (x')$ we distinguish two cases, depending on the size of $|x-x'|$.  We define the sets $D_1(x)\equiv B_{\theta_{N,1}}\setminus B_{\sqrt{\theta_{N,1}}(x)\log \theta_{N,1}}(x)$ and $D_2(x)\equiv B_{\sqrt{\theta_{N,1}}(x)\log \theta_{N,1}}(x)$. Let $x'\in D_1(x)$ first. We use the same arguments as in the proof of Lemma \ref{lemma4.3} to bound $ \E\bigl(Q_{N,0,1} (x)Q_{N,0,1} (x')\bigr)$. 
By analogy with \eqref{4.63}-\eqref{4.69}, 
\be\label{4.81}
\textstyle{ \E\bigl(Q_{N,0,1} (x)Q_{N,0,1} (x')\1_{A_{N,1}}\bigr)}\leq \textstyle{ K' 2d u^{-\a}c_{N,0}^{-\a}\E \1_{A_{N,1}} I_{\theta_{N,1}}(x,x')+c_{N,0}^{-3\a/2},}
\ee
for some $K'\in(0,\infty)$ and where $I_{\theta_{N,1}}(x,x')$ is the expected intersection range of $J$ starting in $x$ and an independent copy $J'$ starting in $x'$, 
\be\label{4.82}
I_{\theta_{N,1}}(x,x')\equiv\textstyle{\sum_{y: |x-y|\wedge|x'-y|\leq\theta_{N,1}}E_x E'_{x'}\1_{\sigma(y)\leq \theta_{N,1}}\1_{\sigma'(y)\leq \theta_{N,1}}}.
\ee
We bound $I_{\theta_{N,1}}(x,x')$ by
\be\label{4.83}
\textstyle{I_{\theta_{N,1}}(x,x')\leq P_xP'_{x'}\bigl(\max_{y\in\Z^d}(\sigma(y)\vee\sigma'(y))\leq \theta_{N,1}\bigr)\bigl(E_{x}R_{\theta_{N,1}}\vee E_{x'}R_{\theta_{N,1}}\bigr).}
\ee
Since $x'\in D_1(x)$, the probability in \eqref{4.83} is smaller than the probability that, during the time interval $[0,\theta_n]$, $J$ (or $J'$) visits a point that is at distance at least  $\tfrac 12\sqrt{\theta_n}\log \theta_n$ from its starting point. By \eqref{lemma321} of Lemma \ref{lemma3.2}, on $A_{N,1}$, this is 
bounded above by $e^{-c_4/4(\log\theta_{N,0})^2}$. Thus, by Lemma \ref{lemma3.3},
\be\label{4.84}
\E I_{\theta_{N,1}}(x,x')\1_{A_{N,1}}\leq c_5 \theta_{N,1}\exp(-c' (\log\theta_{N,0})^2) ,
\ee 
where $c'=c_4/4$. 
We use \eqref{4.84} and get that, for $x\in B_{d_{N,1}(t)}, x'\in D_1(x)$,
\be\label{4.85}
\E\bigl(Q_{N,0,1} (x)Q_{N,0,1} (x')\1_{A_{N,1}}\bigr)\leq c_5(\rho_N (d)\theta_{N,1} e^{-c' (\log\theta_{N,0})^2}+c_{N,0}^{-3\a/2}).
\ee
By \eqref{4.76} of Lemma \ref{lemma3.7} we have for any ball $B_r(y)$ with $r\leq d_{N,1}(t)$ that
\bea
\textstyle{k_{N,1}(t) \sum_{x\in B_r(y)}\overline\pi_N^t(x)}\leq \tfrac{\log\log a_{N,1}}{\theta_{N,1}}\min(r^2, a_{N,1}).\quad\quad \label{4.86}
\eea
By \eqref{4.86}, 
\bea\label{4.87}
\textstyle{(k_{N,1}(t))^2\sum_{x\in B_{d_{N,1}(t)}, x'\in D_1(x)}\overline\pi_N^t(x)\overline\pi_N^t(x')\leq c_3 k_{N,1}(t)(\log\log a_{N,1})^2 }.
\eea
Combining \eqref{4.85} and \eqref{4.87}, 
\be
\textstyle{\sum_{x\in B_{d_{N,1}(t)}, x'\in D_1(x)} \overline\pi_{N}^t(x)\overline\pi_{N}^t(x')\E\bigl(Q_{N,0,1} (x)Q_{N,0,1} (x')\bigr)\leq K(u,t) e^{-c' (\log\theta_{N,0})^2} },
\ee
which is smaller than the right hand side of \eqref{lemma421}. Let $x'\in D_2(x)$. By Cauchy-Schwarz inequality, $\E\bigl(Q_{N,0,1} (x)Q_{N,0,1} (x')\bigr)\leq \rho_N (d)$ and so, by \eqref{4.86}, 
\bea\label{4.88}
\textstyle{(k_{N,1}(t))^2\sum_{x\in B_{d_{N,1}(t)}, x'\in D_2(x)}\overline\pi_N^t(x)\overline\pi_N^t(x')\rho_N (d)\leq c_3k_{N,1}(t)\rho_N (d)(\log \theta_{N,1})^2,}
\eea
as desired in \eqref{lemma421}. This finishes the proof of Lemma \ref{lemma4.2} for $d\geq 3$. 

Let $d=2$. Writing $\overline Q_{N,r_{i+1},r_{i}} (x)\equiv Q_{N,r_{i+1},r_{i}} (x)-\E Q_{N,r_{i+1},r_{i}} (x)$, the left hand side of \eqref{lemma4211} is given by
\bea\label{maind2}
\textstyle{k_{N,1}^4(t)\sum_{(x_1,x_2,x_3,x_4):\min_{i,j}|x_i-x_j|\leq2\theta_{N,1} }\E\bigl(\prod_{j=1}^4\inf_{r\in {\cal I}_i}\overline\pi_{N,r}^t(x_j)\overline Q_{N,r_{i+1},r_{i}} (x_j)\bigr)}\label{4.711},\quad\quad
\eea
where we used that the $\overline Q_{N,r_{i+1},r_{i}} $'s are independent whenever $\min_{i,j}|x_i-x_j|>2\theta_{N,1}$. Let us distinguish three cases, depending on the size of $\max_{i,j}|x_i-x_j|$. Suppose first that $\max_{i,j}|x_i-x_j|\leq \theta_{N,1}^{1/2} \log\log\theta_{N,1}$. Then, by Cauchy-Schwarz inequality and Lemma \ref{lemma4.3},
\be
\textstyle{\E\bigl(\prod_{j=1}^4\overline Q_{N,r_{i+1},r_{i}} (x_j)\bigr)\leq \rho_N (2,4) + 24(\rho_N (2,2))^2 }.
\ee
By \eqref{defpi} and since $\rho_N (2,4)\geq(\rho_N (2,2))^2$, it suffices to control the following quantity:
\bea
\hspace{-2mm}&&\hspace{-2mm}\textstyle{k_{N,1}^4(t)\sum_{(x_1,\ldots,x_4):\max_{i,j}|x_i-x_j|\leq \theta_{N,1}^{1/2}\log\log\theta_{N,1}}\prod_{j=1}^4\overline\pi_N^t(x_j)\rho_N (2,4)}.
\eea
This is smaller than
\bea\textstyle{k_{N,1}^4(t)\sum_{x_1\in B_{d_{N,1}(t)}}\overline\pi_N^t(x_1)\prod_{j=2}^4\sum_{x_l\in B_{4\theta_{N,1}^{1/2}\log\log\theta_{N,1}}(x_1)}\overline\pi_N^t(x_j)\rho_N (2,4)}.
\label{4.7112}
\eea
By \eqref{4.74} of Lemma \ref{lemma3.7} we have that
\be
\textstyle{k_{N,1}(t) \sum_{x\in B_r(y)}\overline\pi_N^t(x)\leq \theta_{N,0}^{-1}(r^2\wedge a_{N,1})\bigl(\log (a_{N,1}/|y|^2)\vee \log\log a_{N,1}\bigr).}\label{4.861}
\ee
This proves that the contribution to \eqref{maind2} coming from $x_i$'s such that $\max_{i,j}|x_i-x_j|\leq \theta_{N,1}^{1/2} \log\log\theta_{N,1}$   is bounded above by $K(u,t)(\log\log a_{N,1})^6 (\log\theta_{N,0})^{-4}$. Now suppose that $\max_{i,j}|x_i-x_j|>\theta_n^{1/2}\log\theta_{N,1}$. Assume that $|x_1-x_2|=\max_{i,j}|x_i-x_j|$. For $j=3,4$ we have the bound $|\overline Q_{N,r_{i+1},r_{i}} (x_j)|\leq\theta_{N,1}^{d}$. 
Then, the only term that remains in the expectation in \eqref{4.711} is $\E\overline Q_{N,r_{i+1},r_{i}} (x_1)\overline Q_{N,r_{i+1},r_{i}} (x_2)$, which is smaller than
\be\label{4.722a}
\textstyle{\E(Q_{N,0,1} (x_1)Q_{N,0,1} (x_2)\1_{A_{N,1}})+ e^{-cN}\leq 2\rho_N (2,2)e^{-c_4(\log\theta_{N,0})^2},}
\ee
where $c\in (0,\infty)$, and where the second inequality is proved as in \eqref{4.85}. 
Hence, the contribution to \eqref{maind2}  coming from $x_i$'s such that $\max_{i,j}|x_i-x_j|>\theta_{N,1}^{1/2}\log\theta_{N,1}$ is bounded above by $e^{-c(\log\theta_{N,0})^2}$. Finally, let $\theta_{N,1}^{1/2}\log\log\theta_{N,1}\leq\max_{i,j}|x_i-x_j|\leq \theta_{N,1}^{1/2}\log\theta_{N,1}$. Suppose again that $|x_1-x_2|=\max_{i,j}|x_i-x_j|$. By Cauchy-Schwarz  and Lemma \ref{lemma4.3},
\bea\label{4.722}
\textstyle{\E\bigl(\prod_{j=1}^4\overline Q_{N,r_{i+1},r_{i}} (x_j)\bigr)}\leq\textstyle{ \bigl\{\E(Q_{N,0,1} (x_1)Q_{N,0,1} (x_2))^2\rho_N (2,4)\bigr\}^{1/2}+c(\rho_N (2,2))^2}.\ 
\eea
As in \eqref{4.722a} and \eqref{4.85} we know that the first summand in \eqref{4.722} is bounded above by
\be\label{4.722aa}
\textstyle{\E(Q_{N,0,1} (x_1)Q_{N,0,1} (x_2))^2\leq \rho_N (2,2)e^{-c_4 (\log\log\theta_{N,0})^2}},
\ee
which is larger than the second summand. Hence, by \eqref{4.861}, the contribution to \eqref{maind2} coming from $x_i$'s such that $\theta_{N,1}^{1/2}\log\log\theta_{N,1}\leq\max_{i,j}|x_i-x_j|\leq \theta_{N,1}^{1/2}\log\theta_{N,1}$ is bounded above by $K'(u,t) e^{-c_4/4(\log\log\theta_{N,0})^2}$. The proof of Lemma \ref{lemma4.2} is finished.
\end{proof}

\subsection{Convergence of $\wt\sigma_n^t(u,\infty)$}\label{S4.3}
We establish that, $\P$-a.s., $\lim_{n\rightarrow\infty}\E\wt\sigma_n^t(u,\infty)=0$. 

As in the proof of Lemma \ref{lemma3.5} we consider subsequences $c_{N,r}=\exp(N+r)$ first (see the paragraph below \eqref{3.19a}). With the notation of \eqref{qnrs} and \eqref{defpi},
\be\label{4.93a}
\textstyle{\E\sup_{r\in[0,1)}\wt\sigma_{N,r}^t(u,\infty)\leq k_{N,1}(t)\sum_{x\in B_{d_{N,1}(t)}}\overline\pi_{N}^t(x)\E ( Q_{N,0,1}^u(x))^2\equiv \E\sigma_N^t(u,\infty) }.\ 
\ee
As in Lemma \ref{lemma4.3} one can show that
 $ \E (Q_{N,0,1}^u(0))^2 \leq \rho_N^u(d)$, and so, by Lemma \ref{lemma3.7} there exists $K'\in (0,\infty)$ such that 
\bea\label{4.94}
\textstyle{\E\sigma_N^t(u,\infty)}\leq K'tu^{-\a}((\log\theta_{N,0})^{-1}\1_{d=2}+\theta_{N,1}^{-1/2}\1_{d\geq3}).
\eea
For $d\geq 3$, this is summable in $N$ and we get that, $\P$-a.s., $\lim_{n\rightarrow\infty}\E\wt\sigma_{n}^t(u,\infty)=0$. When $d=2$,  the right hand side of \eqref{4.94} is not summable in $N$. Thus, to prove that $\wt\sigma_{n}^t(u,\infty)$ vanishes $\P$-a.s., we bound the variance of $\sigma_N^t(u,\infty)$. 
Using the same calculations as in the proof of \eqref{lemma4211} of Lemma \ref{lemma4.2} one can show that the variance of $\sigma_N^t(u,\infty)$ is bounded above by $N^{-4+\e}$. Since this is summable in $N$, we obtain that, $\P$-a.s., $\lim_{n\rightarrow\infty}\E\wt\sigma_{n}^t(u,\infty)=0$ for $d=2$ as well. 
 
\subsection{Verification of Condition (C-4)}\label{S4.4}
We follow the same strategy as in the verification of (C-2). We first prove that $\lim_{n\rightarrow\infty}\E m_n^t(\e)\leq C(t) \varepsilon^{1-\a}$. Then, one can show as in Section \ref{S4.2} that, $\P$-a.s., $\lim_{n\rightarrow\infty} m_n^t(\e)\leq C(t) \e^{1-\a}$. Since this is very similar to Section \ref{S4.2}, we leave the details to the interested reader. Let us bound $\E m_n^t(\e)$.
Since the $\tau$'s are i.i.d. and since $a_n\theta_n^{d}\P(A_n^c)\leq n^{-3}$, it suffices to find $c\in(0,\infty)$ such that
\bea\label{4.95}
\textstyle{ \sum_{y\in B_{\theta_n}} \E(M_n^\e(0,y)\1_{A_n})}\leq c \theta_n a_n^{-1}\e^{1-\a}.
\eea
Fix $y\in B_{\theta_n}$ and set $h(\theta_n)=\theta_n-k(\theta_n)$ for $k(\theta_n)=\theta_n^{3/4}$. As in \eqref{4.4} and \eqref{4.10}, by the Markov property,
\bea
M_n^\e(0,y)\hspace{-2mm}&\leq&\hspace{-2mm}\e  P(\sigma(y)\in (h(\theta_n), \theta_n))\1_{y\in T_n}\label{4.97}\\
\hspace{-2mm}&+&\hspace{-2mm}P(\sigma(y)\leq h(\theta_n))\gamma_n(y)\EE_y\bigl(\ell_{\theta_n}(y)\1_{\ell_{k(\theta_{n})}(y) \gamma_{n}(y)\leq \varepsilon}\bigr)\1_{y\in T_n}\equiv M_{n,1}(y)+M_{n,2}(y) \nonumber
\eea
Let us first establish, that the sum over $M_{n,1}(y)$ is as in \eqref{4.95}. Following the same argumentation as between \eqref{4.14} and \eqref{4.17}, we can show that
\be\label{4.98}
\textstyle{M_{n,1}(y)\leq \e (P(\min_{y'\sim y}\sigma(y)\in (h(\theta_n), \theta_n))+c_n^{-\theta/2})\1_{y\in T_n}}.
\ee
Since $\min_{y'\sim y}\sigma(y)$ is independent of $\gamma_n(y)$, we get by Lemma \ref{lemma3.3} that
\be
\textstyle{\sum_y \E(M_{n,1}(y)\1_{A_n})\leq\sum_{y}\E\bigl[\1_{A_n}P(\sigma(y)\in (h(\theta_n), \theta_n))\1_{y\in T_n}\bigr]\leq c_5 c_n^{-\a}\epsilon_n^{-\a} \theta_n^{3/4}},\label{4.99}
\ee 
as desired.
Let us now bound the expectation of $M_{n,2}(y)$. First we calculate the expected value with respect to $\EE_y$ in $M_{n,2}(y)$. As in \eqref{4.12} and \eqref{4.13}, we get that, up to an error of the order of $e^{-c_4(\log \theta_n)^2}$,  on $A_n$,  we can bound $\ell_{\eta(B_n^1)}(y)\leq\ell_{k(\theta_n)}(y)$ and $ \ell_{\theta_{n}}(y)\leq \ell_{\eta(B_n^2)}(y)$, where $B_n^1=B_{k(\theta_n)^{1/2}(\log \theta_n)^{-2}}(y)$ and $ B_n^2=B_{\theta_n^{1/2} \log \theta_n}(y)$, for all $y\in B_{\theta_n}\cap T_n$. Setting $\e(y)\equiv\{\ell_{\eta(B_n^1)}(y) \gamma_{n}(y) \leq \varepsilon\}$, we get
\bea\label{4.100}
\EE_y\ell_{\theta_{n}}(y)  \1_{ \ell_{k(\theta_{n})}(y)\gamma_n(y)\leq \varepsilon}= \bigl(\EE_y \1_{\e(y)}\ell_{\eta(B_n^1)}(y)+\EE_y\1_{\e(y)}\bigl(\ell_{\eta(B_n^2)}(y)-\ell_{\eta(B_n^1)}(y)\bigr)\bigr).\quad\quad
\eea
By the strong Markov property, the second term in \eqref{4.100} is given by
\bea\label{4.101}
\textstyle{\sum_{z\in \partial B_n^1}\EE_y\bigl(\1_{\e(y)} \1_{J(\eta(B_n^1))=z}\bigr)E_z\int_0^{\eta(B_n^2)}\1_{J(s)=y} ds\leq g_{B_n^2}(y)\PP_y(\e(y)),}\ 
\eea
where we used $E_z\int_0^{\eta(B_n^2)}\1_{J(s)=y} ds\leq g_{B_n^2}(y)$.
The first term in \eqref{4.100} equals
\be\label{4.102}
\textstyle{g_{B_n^1}(y)\bigl[1- \exp\bigl(-\varepsilon /(\gamma_{n}(y)g_{B_n^1}(y))\bigr)\bigr]=g_{B_n^1}(y)\PP_y(\e(y))}.
\ee
Using \eqref{4.101} and \eqref{4.102} and the fact that $g_{B_n^1}(y)\leq g_{B_n^2}(y)$, \eqref{4.100} is bounded by
\bea\label{4.103}
 \hspace{-2mm}&&\hspace{-2mm} \textstyle{2 g_{B_n^2}(y)\bigl[1- \exp\bigl(-\varepsilon/ (\gamma_{n}(y)g_{B_n^1}(y))\bigr)\bigr]}\nonumber\\
\hspace{-2mm}&\leq&\hspace{-2mm}\textstyle{2 c_8/c_7{\bar g}_n^d(y)\bigl[1- \exp\bigl(-\varepsilon /(\gamma_{n}(y) {\bar g}_n^d(y))\bigr)\1_{A_n} + c_1(\log \theta_n\1_{d=2}+\1_{d\geq3})\1_{A_n^c}\bigr]},\quad\quad
\eea
where by \eqref{4.23} (for $d\geq 3$) and Lemma 3.3 in \cite{Ce11} (for $d=2$), ${\bar g}_n^d(y)\equiv c_7(\log \theta_n\1_{d=2}+ g_{\infty}(y)\1_{d\geq 3})$.
Together with \eqref{4.14}, $\E M_{n,2}(y)\1_{A_n}$ is bounded above by
\bea\label{4.104}
\textstyle{\e\ \E\bigl[P(\min_{y'\sim y}\sigma(y')\leq \theta_n){\bar g}_n^d(y)\gamma_{n}(y)\bigl(1-e^{-\varepsilon (\gamma_{n}(y){\bar g}_n^d(y))^{-1}}\bigr) \1_{y\in T_n}\1_{A_n}\bigr].}\quad
\eea
An asymptotic analysis and the fact that ${\bar g}_n^d(y)\leq c_7(\log \theta_n\1_{d=2}+ c_6\1_{d\geq 3})$ yield
\bea\label{4.105}
\eqref{4.104} \hspace{-2mm}&\leq& \hspace{-2mm}\textstyle{c'\ \e \ \E\Bigl[P(\min_{y'\sim y}\sigma(y')\leq \theta_n)  e^{-2\varepsilon(c' \gamma_{n}(y) {\bar g}_n^d(y))^{-1}}\1_{y\in T_n}\1_{A_n}\Bigr]}\nonumber\\
 \hspace{-2mm}&\leq& \hspace{-2mm}\textstyle{c''\ (\log \theta_n\1_{d=2}+ c_6\1_{d\geq 3})^\a c_n^{-\a} \e^{1-\a}  \E( P(\min_{y'\sim y}\sigma(y')\leq \theta_n)\1_{A_n}).}\quad
\eea
for some $c',c''\in (0,\infty)$. Thus, evoking Lemma \ref{lemma3.3}, we get that
\be
\textstyle{\sum_{y}\E M_{n,2}(y)\1_{A_n}\leq c'' c_6 c_5 t \e^{1-\a},}
\ee
i.e. \eqref{4.95} is satisfied. Thus, 
$\lim_{n\rightarrow\infty} \E m_n^t(\e)\leq c\varepsilon^{1-\a}$. The verification of (C-4) now follows as in Section \ref{S4.2}.
\subsection{Verification of Condition (C-5)}\label{S4.5} We proceed as in the verification of (C-2) and (C-4) to establish that (C-5) is satisfied. Namely, we first take the expected value in the left hand side of \eqref{c51} and \eqref{c52} and prove that both are bounded above by $C(u,t)\e$ for some $C(u,t)\in(0,\infty)$. Then, one can proceed as in Section \ref{S4.2} to obtain $\P$-a.s.~upper bounds. Since the proofs are similar, we only prove the claim for \eqref{c51}. The expectation of the left hand side of \eqref{c51} is given by
\bea\label{4.106}
&&\textstyle{\sum_{(x,k)\in{\cal A}_n} (k\theta_n)^{-d/2}e^{-c_2|x|^2/k\theta_n}\sum_{y}\E Q_n^u(x,y)}\nonumber\\
&\leq&\textstyle{ 2\nu(u,\infty) a_n^{-1}\theta_n\sum_{k=1}^{k_n(t)}(k\theta_n)^{-d/2}\sum_{|x|^2< \e k\theta_n\ \vee\ |x|^2>k\theta_n/\e}e^{-c_2|x|^2/k\theta_n}},
\eea
where we used that by \eqref{4.2}, for $n$ large enough, the second sum in \eqref{4.106} is smaller than $ 2\nu(u,\infty)\theta_n/a_n$. Let us first control the contribution to the right hand side of \eqref{4.106} coming from $x\in B_{\sqrt{\e k\theta_n}}$. Bounding the exponential term by one and using the fact that $|\{x: |x|^2<\e k\theta_n\}|\leq C\e^{d/2} (k\theta_n)^{d/2}$ for some $C\in(0,\infty)$,  this contribution is bounded above by $K(u,t)\e^{d/2}$ for  $K(u,t)\in (0,\infty)$. Also, the sum over $x\in B_{d_n(t)}\setminus B_{(k\theta_n/\e)^{1/2}}$ of $(k\theta_n)^{-d/2}e^{-c_2|x|^2/k\theta_n}$ is bounded above by $K'(t) e^{-c_2 /\e}$.
Thus,  \eqref{4.106} is bounded above by $K(u,t)\e^{d/2}$. 
The verification of (C-5) can now be finished as in Section \ref{S4.2}.
\subsection{Conclusion of the proof}\label{S4.6}
We are now ready to conclude the proof of Theorem \ref{theorem1.5}. By Lemma \ref{lemma4.1}, for all $u>0$, $t>0$, $\E\wt\nu_n^t(u,\infty)\rightarrow t\nu(u,\infty)$ as $n\rightarrow \infty$. Together with the results of Section \ref{S4.2} this shows that (C-2) is satisfied. By the results of Section \ref{S4.3}, $\wt\sigma_n^t(u,\infty) $ tends $\P$-a.s.~to zero, proving   (C-3). By the same arguments as those used to establish (C-2),  (C-4) follows from the results of Section \ref{S4.4}. Finally, it follows from the results of Section \ref{S4.5} that (C-5) is satisfied. Hence, we may apply Proposition \ref{proposition3.6} and get that $\overline S_n^{J,b}\stackrel{J_1}{\Longrightarrow}V_{\a}$. By Lemma \ref{lemma3.5} this proves that $S_n^{J,b}\stackrel{J_1}{\Longrightarrow}V_{\a}$, as claimed in Theorem \ref{theorem1.5}.


\section{Aging in BATM}\label{S5}
In this section we present the proofs of Theorem \ref{theorem1.4} and Theorem \ref{theorem1.6}. Section \ref{s51}, respectively Section \ref{s52} and Section \ref{s53}, contains the proof of Theorem \ref{theorem1.4} for $i=1$, respectively 
$i=2$ and $i=3$.
We then prove Theorem \ref{1.6} in Section \ref{s54}. The proofs in Sections \ref{s51}-\ref{s53} follow a common scheme. We show that for each $i\in \{1,2,3\}$,  as $s\rightarrow\infty$, 
${\cal C}^i_s(1,\rho)$, coincides  with the probability of ${\cal M}_{s,\rho}\equiv\{{\cal R}_{s}\cap (1,1+\rho) =\emptyset\}$, where ${\cal R}_{s}=\{S_{s}^{J,b}(t), t\geq 0\}$ is the range of $S_{s}^{J,b}$. We then use that $\P$-a.s., 
\be\label{5.000}
\textstyle{\lim_{s\rightarrow\infty}\PP({\cal M}_{s,\rho})={\asl_\a}(1/(1+\rho)).}
\ee 
The proof of \eqref{5.000} closely follows that of Theorem 1.6 in \cite{G10a}. We thus only sketch it here. Namely, it relies on the continuity of the overshoot function that maps $Y\in D[0,\infty)$ to 
\be\label{os}
\chi_u(Y)=Y({\cal L}_u(Y)) - u,\quad u>0,
\ee
where ${\cal L}_u\equiv \inf\{ t\geq 0: Y(t)>u\}$ is the time of the first passage of $Y$ beyond the level $u>0$. For L\'evy motions having $\PP$-a.s.~diverging paths, this mapping is $\PP$-a.s.~continuous on $D[0,\infty)$ equipped with Skorohod's $J_1$ topology. Now, ${\cal M}_{s,\rho}=\{\chi_1(S_{s}^{J,b}) \geq 1+\rho\}$ and by Theorem \ref{theorem1.5}, $\P$-a.s., $S_{s}^{J,b}\stackrel{J_1}{\Longrightarrow} V_\a$. Since $V_\a$ has $\PP$-a.s.~diverging paths we deduce that, $\P$-a.s.,
\be\label{5.8a}
\textstyle{\lim_{s\rightarrow\infty}\PP({\cal M}_{s,\rho})=\PP(\xi_1(V_\a)\geq 1+\rho)=\asl_\a(1/(1+\rho)),\quad \rho>0,}
\ee
where the last equality follows from the arcsine law for stable subordinators (see Section III in \cite{Ber}). Given \eqref{5.8a}, it remains to establish that, $\P$-a.s.,
\be\label{5.3a}
\textstyle{\lim_{s\rightarrow\infty}|\PP({\cal M}_{s,\rho})-\PP({\cal M}^i_{s,\rho})|=0, \quad \forall \rho> 0,}
\ee
where ${\cal M}^i_{s,\rho}$ stands for the events appearing in the right hand sides of \eqref{C1}-\eqref{C3}, 
namely
${\cal C}^i_s(1,\rho)=\PP({\cal M}^i_{s,\rho})$. We will verify  \eqref{5.3a} in Sections \ref{s51}-\ref{s53}.

For this, let  $A'\subseteq A$ (where $A$ as in \eqref{3.7a}) such that $\P(A')=1$ and such that for all $\omega\in A'$, $S_s^{J,b}\stackrel{J_1}{\Longrightarrow} V_{\a}$. Fix $\omega\in A'$. We write $a_s\equiv a_{\lfloor s\rfloor}$ and similarly for $\theta_s, k_s, \epsilon_s$, and $\d_s$ (see Theorem \ref{theorem1.4} and Section \ref{s3.2} for their definitions).

\subsection{Convergence of ${\cal C}^1_s(1,\rho)$}\label{s51}
In this section we prove that \eqref{5.3a} holds for $i=1$.

\noindent \textbf{\emph{Step 1.}} Let $\d>0$ and $s>0$. We prove that for all $s$ large enough, $\PP({\cal M}_{s,\rho} , ({\cal M}^1_{s,\rho})^c)\leq \d$. For $k\in \N$ we define
\bea
\textstyle{B_k\equiv\{\sum_{i=1}^{k}Z_{s,i}^{J}< 1}, \quad\mbox{and}\quad \textstyle{\sum_{i=1}^{k+1}Z_{s,i}^{J}> (1+\rho)\}}.
\eea   
Then, ${\cal M}_{s,\rho}=\bigcup_{k\geq 1} B_k$. For $s$ large enough, there exists $T>0$ large enough such that 
\be\label{56789}
\textstyle{\PP({\cal M}_{s,\rho},({\cal M}^1_{s,\rho})^c)\leq \PP(\bigcup_{k\leq k_s(T)}B_k, ({\cal M}^1_{s,\rho})^c)+\d},\quad \forall s>s',
\ee
where $k_s(T)=\lfloor a_s T\rfloor/\theta_s$. To see the claim of \eqref{56789} note that, since $S_s^{J,b}\stackrel{J_1}{\Longrightarrow} V_{\a}$,
\bea
\textstyle{\PP(\bigcup_{k\geq k_s(T)} B_k, ({\cal M}^1_{s,\rho})^c)\leq\PP(S^{J,b}_s(T)<1)\leq\PP(V_{\a}(T)<1+\d)+\d},\quad\label{5.4aaa}
\eea
which vanishes as $T\rightarrow\infty$ and proves \eqref{56789}. 
Let us now study $B_k$ for fixed $k$. By Lemma \ref{lemma3.5}, we know that for $s$ large enough, the contribution to $Z_{s,k+1}$ coming from $y\notin T_s$ is bounded above by $\d_s$ and therefore, on $B_k$, there exists $x\in T_s$ such that $\ell_{\theta_s (k+1)}(x)-\ell_{\theta_s k}(x)>0$. Moreover, as in the proof (C-2) $\Rightarrow$ (B-2) (see proof of Proposition \ref{proposition3.6}), one can show that  this $x$ is unique. Therefore, we have that $\ell_{a_s T}(x)\gamma_s(x)>\rho-\d_s$. We prove now that, for all $s$ large enough, with $\PP$-probability larger than $1-\d$, we know on $B_k$ that $\gamma_s(x)>\d_s^{1/2}$. 
For this, note that by definition of $A'$ (last paragraph of Section \ref{S5}) and Lemma \ref{lemma3.2}, for all but finitely many values of $\lfloor s\rfloor$, $\PP(\eta(B_{d_s(T)})> a_s T)\geq 1- \d$. Moreover, by the same lemma and \eqref{theorem311}, we have for all $s$ large enough, $g_{B_{d_s(T)}(x)}(x)\leq c_9\log a_s\1_{d=2}+c_6\1_{d\geq3}$ for all $x\in B_{d_s(T)}$. Therefore, for $s$ large enough,  
\bea
\hspace{-2mm}&&\hspace{-2mm}\textstyle{\PP(\exists x: \gamma_s(x)\ell_{a_s T}(x)>\rho-\d_s, \g_s(x)\leq \d_s^{1/2})}\nonumber\\
\hspace{-2mm}&\leq&\hspace{-2mm}\textstyle{\d+ \sum_{x\in B_{d_s(T)}} P_x(\ell_{\eta(B_{d_s(T)}(x))}(x)>\rho\d_s^{-1/2})}
\leq\d+|B_{d_s(T)}|\exp(-c'(\log a_s)^2),\quad
\eea
where we used that $\d_s^{-1/2}\geq c_0( \log a_s)^3$. Therefore, for $s$ large enough, we know with high $\PP$- probability that, on $B_k$, there exists a unique $x\in T_s'\equiv\{x\in T_s:\gamma_s(x)>\d_s^{1/2}\}$ that contributes to $Z_{s,k+1}$.  
Since this holds for all $1\leq k\leq M$, we get  for $s$ large enough that 
\bea\label{517}
\textstyle{\PP({\cal M}_{s,\rho}\cap ({\cal M}^1_{s,\rho})^c)\leq\PP(\bigcup_{s'\in\{s,s(1+\rho)\}} M_{s'})+3\d},
\eea
where 
\bea
\textstyle{M_{s'}\equiv\{X(s')\notin T_s'\}\cap\{ \exists s'(1-\d_s)<v<s': v: X(v)\in T_s'\}}.
\eea
Let us now bound $\PP(M_{s'})$ for $s'=s$ and fixed $x\in T_s'$; the proof for $s'=s(1+\rho)$ is the same. We distinguish two cases with respect to $\theta$. We begin with $\theta>0$. Fix $v$ such that $s(1-\d_s)<v<s$ and $X(v)=x$. Let us first bound the probability that there exists $v'$ such that $v\leq v'\leq s$ and $X(v')\notin B_1(x)\equiv\{x\}\cup\{y\sim x\}$. Writing $N_x$ for the number of returns to $x$ before $J$ escapes $ B_1(x)$, we have
\bea\label{542}
\hspace{-2mm}&&\hspace{-2mm}\PP(X(v)=x, \exists v': v\leq v'\leq s: X(v')\notin B_1(x))\nonumber\\
\hspace{-2mm}&\leq&\hspace{-2mm}\textstyle{ P_x(N_x\leq \d(s\d_s^{1/2}\epsilon_s^{-2/\a})^{\theta})+\PP_x(\sum_{i=1}^{\d(s\d_s^{1/2}\epsilon_s^{-2/\a})^{\theta}}(\lambda(x))^{-1}e_i\leq \d_s s)}.\quad
\eea
Since $\max_{y\sim x}(1-p(y,x))\leq(s\d_s^{1/2}\epsilon_s^{-2/\a})^{-\theta}$,  the first probability in \eqref{542} is, smaller than $\d$ . Since $\lambda(x)\leq 2d s^{-1+\theta} \epsilon_s^{-2\theta/\a}\d_s^{-(1-\theta)/2}$ the law of large numbers implies that also the second probability in \eqref{542} is bounded above by $\d$, for $s$ large enough. It remains to bound 
\be\label{5673}
\PP( X(v)=x, X(s)\neq x,  \forall v\leq v'\leq s: X(v')\in B_1(x)).
\ee
By definition of  $T_s'$, $\max_{y\sim x}(\lambda(y))^{-1}\leq \epsilon_s^{-2\theta/\a}\d_s^{\theta/2}s^{-\theta}$, and so, with probability larger than $1-\exp(-\d_s^{-\theta/2})$, there exists $v'$ such that $s-v'\leq s^{-\theta}\epsilon_s^{-2\theta/\a}$ and $X(v')=x$. By the Markov property we have for all such $v'$,
\be
\PP_x( X(s-v')\neq x)\leq \PP_x( e_1\lambda^{-1}(x)<s-v')\leq 1-e^{-s^{-1}\d_s^{(1-\theta)/2}},\label{5118}
\ee
which tends to zero. Thus, \eqref{5673} tends to zero as $s\rightarrow\infty$. This finishes the proof for $\theta>0$. When $\theta=0$, one can bound \eqref{517} directly as in \eqref{5118}. This shows that for all $s$ large enough, $\PP({\cal M}_{s,\rho},({\cal M}^1_{s,\rho})^c)\leq \d$.

\noindent \textbf{\emph{Step 2.}} Let us now show that $\PP(({\cal M}_{s,\rho})^c, {\cal M}^1_{s,\rho})\rightarrow 0$. Let $m_{s,\rho}\equiv (S^{J})^{\leftarrow}(s(1+\rho)) -(S^{J})^{\leftarrow}(s)$, where $(S^{J})^{\leftarrow}(t)=\inf\{v\geq 0: S^{J}(v)> t\}$. Notice that $({\cal M}_{s,\rho})^c\subseteq \{m_{s,\rho}\geq \theta_s\}\cup\{Z_{s,1}>1\}$. By (A-0), $\PP(Z_{s,1}>\rho)$ tends to zero and so, for all $\d>0$ there exists $s$ large enough such that $\PP(({\cal M}_{s,\rho})^c, {\cal M}^1_{s,\rho})\leq \PP(m_{s,\rho}\geq \theta_s, {\cal M}^1_{s,\rho})+\d$. 
 Let us distinguish whether $d\geq 3$ or $d=2$. In the first case we use the identity $X(t)= J((S^{J})^{\leftarrow}(t))$ and get by \eqref{theorem311} of Theorem \ref{theorem3.1}, uniformly in $x\in \Z^d$,
\be\label{5321}
\textstyle{\PP_x(X(s(1+\rho))=x, m_{s,\rho}\geq \theta_s)=\PP_x(J(m_{s,\rho})=x, m_{s,\rho}\geq\theta_s)\leq \int_{\theta_s}^{\infty} v^{-d/2}dv,}
\ee
which is smaller than $\theta_s^{-d/2+1}$ and shows that $\PP(({\cal M}_{s,\rho})^c, {\cal M}^1_{s,\rho})\rightarrow 0$ for $d\geq 3$. 

Let $d=2$. We construct a more precise bound for $\PP(({\cal M}_{s,\rho})^c\cap{\cal M}^1_{s,\rho})$ than $\PP(\{m_{s,\rho}\geq\theta_s\}\cap{\cal M}^1_{s,\rho})+\d$. Assume first that $\dist(\cal R_{s},1+\rho)>\d$ and that there are $t,t'>0$ such that $S_s^{J,b}(t), S_{s}^{J,b}(t+t')\in (1,1+\rho-\d/2)$. Then, $s<S^{J}(k_s(t)\theta_s)< S^{J}(k_s(t+t')\theta_s)<s(1+\rho)$ and so $m_{s,\rho}\geq \theta_s( k_s(t+t')- k_s(t'))$. Moreover, by \eqref{5.4aaa} there exists $T>0$ such that $m_{s,\rho}\leq \theta_s k_s(T)$. Since $\dist(\cal R_{s},1+\rho)>\d$ one can show as in Step 1 that $X(s(1+\rho))=x\in T_s$. But then, on $({\cal M}_{s,\rho})^c\cap {\cal M}^1_{s,\rho}$, we have with probability larger than $1-(\log\theta_s)^{-2}$ that $\ell_{m_{s,\rho}}(x)-\ell_{m_{s,\rho}-\theta_s}(x) >c\log \theta_s/\log\log\theta_s$ for some $c\in(0,\infty)$. 
${\cal R}_s\cap(1,1+\rho)=\emptyset$ or $X(s)\neq x$.) 
By \eqref{theorem311} of Theorem \ref{theorem3.1} we get, for all $x\in \Z^d$, 
\bea
\hspace{-2mm}&&\hspace{-2mm}\PP_x\bigl(J(m_{s,\rho})=x, m_{s,\rho}\in(\theta_s k_s(t), \theta_sk_s(T)), \ell_{m_{s,\rho}}(x) - \ell_{m_{s,\rho}-\theta_s}(x) >\tfrac{c\log\theta_s}{\log\log\theta_s}\bigr)\nonumber\\
\hspace{-2mm}&\leq&\hspace{-2mm}\textstyle{ P_x(\ell_{\theta_sk_s(T)}(x)-\ell_{\theta_s k_s(t)-\theta_s}(x)>\tfrac{c\log\theta_s}{\log\log\theta_s})}\leq c\tfrac{c\log\log\theta_s}{\log\theta_s}\log (T/t),
\eea
which tends, as $s\rightarrow\infty $, to zero. It remains to establish that for all $\d'>0$ there exist $\d>0$, $t>0$ such that $\PP(\dist(\cal R_{s},1+\rho)> \d)\leq 1-\d'$ and $\PP(S_{s}^{J,b}(t+t') \in(1,1+\rho-\d/2)|S_{s}^{J,b}(t) \in(1,1+\rho-\d))\geq 1-\d'$. This can be derived from the convergence of $S_s^{J,b}$ to $V_\a$ and properties of $V_\a$ (see Section III in \cite{Ber}). Thus, $\PP(({\cal M}_{s,\rho})^c, {\cal M}^1_{s,\rho})\leq\d+2\d'$ for $d\geq2$. Together with Step 1 this finishes the proof of \eqref{5.3a} for $i=1$.
\subsection{Convergence of ${\cal C}^2_s(1,\rho)$}\label{s52}
In this section we prove the claim of \eqref{5.3a} for $i=2$.

\noindent \textbf{\emph{Step 1.}} We show that $\PP\bigl({\cal M}_{s,\rho},({\cal M}^2_{s,\rho})^c\bigr)\rightarrow 0$. Note that ${\cal M}_{s,\rho} \subseteq \{m_{s,\rho}\leq\theta_s\}$. Let $x\in B_{a_s}$. Using $X(s)=J((S^{J})^{\leftarrow}(s))$ and the Markov property
\bea\label{5.567}
\PP(X(s)=x,({\cal M}^2_{s,\rho})^{c}, m_{s,\rho}\leq \theta_s)
\leq\textstyle{\PP(X(s)=x) P_x(\eta(B_{(\theta_s\log\theta_s)^{1/2}}(x))\leq\theta_s),}
\eea
where $\eta(B)$ is the exit time of $B$ for $J$ as defined in Section \ref{S3.1}.
By definition of $A'$ (last paragraph of Section \ref{S5}) and Lemma \ref{lemma3.2} we have for all but finitely many values of $\lfloor s\rfloor$, for all $x\in B_{a_s}$,
\bea\label{5.3aa}
\textstyle{P_x(\eta(B_{(\theta_s\log\theta_s)^{1/2}}(x))\leq\theta_s)\leq \exp(-c_4\log\theta_{s}).}
\eea
If we can show that $\PP(X(s)\notin B_{a_s})\leq \d$, then \eqref{5.567}-\eqref{5.3aa} imply that $\PP({\cal M}_{s,\rho}({\cal M}^2_{s,\rho})^c)\leq\d$. To bound $\PP(X(s)\notin B_{a_s})$ we recall that by \eqref{5.4aaa}, with probability larger than $1-\d$, $(S^{J})^{\leftarrow}(s)\leq k_s(T)\theta_s$ for $T>0$ and so, for $s$ large enough,
\be\label{5.4aa}
\textstyle{\PP(X(s)\notin B_{a_s}, (S^{J})^{\leftarrow}(s)\leq a_s T)\leq P(\eta(B_{a_s})\leq a_sT)\leq e^{-c_4 a_s^{1/2}T^{-2}}},
\ee
by Lemma \ref{lemma3.2}. This tends to zero and we conclude that $\PP({\cal M}_{s,\rho},({\cal M}^2_{s,\rho})^c)\leq \d$.

\noindent \textbf{\emph{Step 2.}} Now we prove that $\PP(({\cal M}_{s,\rho})^c,{\cal M}^2_{s,\rho})$ vanishes. On $({\cal M}_{s,\rho})^c$ one can show as in Section \ref{s51} (Step 2) that there exist $t,t'$ such that $m_{s,\rho}\geq\theta_s(k_s(t+t')-k_s(t))\gg \theta_s^2$. Moreover, by \eqref{5.4aa} we know that, with probability larger than $1-\d$, $X(s)\in B_{a_s}$. By definition of $A'$ and Lemma \ref{lemma3.2}, for all but finitely many values of $\lfloor s\rfloor$, for all $x\in B_{a_s}$
\bea
\textstyle{P_x(\eta(B_{(\theta_{s}\log\theta_s)^{1/2}}(x))>\theta_s^2)\leq \exp(-c_4\log\theta_{s})}.\label{5.aa}
\eea
As in \eqref{5.567} we thus get $\PP(({\cal M}_{s,\rho})^c,{\cal M}^2_{s,\rho})\rightarrow 0$. Together with Step 1, the proof of \eqref{5.3a} is finished for $i=2$.
\subsection{Convergence of ${\cal C}^3_s(1,\rho)$}\label{s53}
We now show that \eqref{5.3a} holds for ${\cal C}^3_s(1,\rho)$. This follows readily from Sections \ref{s51} and \ref{s52}. 
Indeed on the one hand,
\be\label{5.601}
\PP(({\cal M}_{s,\rho})^c,{\cal M}^3_{s,\rho})\leq \PP(({\cal M}_{s,\rho})^c,{\cal M}^1_{s,\rho})+\PP(({\cal M}_{s,\rho})^c,{\cal M}^2_{s,\rho}),
\ee
and on the other hand, $\PP({\cal M}_{s,\rho},({\cal M}^3_{s,\rho})^c)\leq \PP({\cal M}_{s,\rho},({\cal M}^1_{s,\rho})^c)$.
Both upper bounds tend by Sections \ref{s51} and \ref{s52} $\P$-a.s.~to zero, which proves \eqref{5.3a} for ${\cal C}^3_s(1,\rho)$.
\subsection{Convergence of ${\cal C}_s^\e(1,\rho)$ and ${\cal C}^\e(1,\rho)$} \label{s54}
In  this section we prove Theorem \ref{theorem1.6}. The convergence of ${\cal C}_s^\e(1,\rho)$ can be proved as that of ${\cal C}^2_s(1,\rho)$ and we
only establish the claim of Theorem \ref{theorem1.6} for ${\cal C}^\e(1,\rho)$. 
Let us write in short ${\cal M}^\e(\rho)$ for the event in the right hand side of \eqref{1.56a}, i.e.~${\cal C}^\e(1,\rho)=\PP({\cal M}^\e(\rho))$. One can show that
\bea\label{5.21ab}
\textstyle{{\cal C}_s^{\e/2}(1-\e^2,\frac{\rho+2\e^2}{1-\e^2})-\bigl[1- {\cal C}_s^{\e/2}\bigl(1-\e^2,\frac{2\e^2}{1-\e^2}\bigr)\bigr]-\d_s }\leq {\cal C}^\e(1,\rho)\leq {\cal C}_s^{\e}(1,\rho) + \d_s,\ \ 
\eea
where $\d_s\equiv\d_s(\rho,\e)$ is given by
\be
\d_s=\textstyle{\PP(\max_{v\in(1-\e^2,1+\rho+\e^2)}\max_{v'\in(1-\e^2,1+\e^2)}\bigl|X_s(v')-X_s(v)\bigr|\leq\e/2,({\cal M}^\e(\rho))^c)}.
\ee
Now, by Theorem 1.3 in \cite{BaCe11} (see the erratum \cite{BaCeErratum} to this theorem) and Theorem 1.1 in \cite{Ce11}, $\P$-a.s., $X_s\stackrel{J_1}{\Longrightarrow}Z_{d,\a}$. By definition of Skorohod's metric, there exists $\d>0$ such that, for $s$ large enough and $\l:[0,1+\rho]\rightarrow [0,1+\rho]$ strictly increasing and continuous,
\be
\textstyle{\PP\bigl(\max\bigl\{\max_{v\in[0,1+\rho]}\bigl|X_s(\l(v))-Z_{d,\a}(v)\bigr|,\max_{v\in[0,1+\rho]}|\lambda(v)-v|\bigr\} >\e^2 \bigr)\leq \d},
\ee
and so $\d_s$ vanishes as first $s\rightarrow\infty$ and then $\e\rightarrow 0$. 
By the statement of Theorem \ref{1.6} for ${\cal C}_s^\e(1,\rho)$, ${\cal C}_s^\e(1,\rho)$ and ${\cal C}_s^{\e/2}(1-\e^2,\frac{\rho+2\e^2}{1-\e^2})$ tend to $\asl_\a(1/(1+\rho))$ as first $s\rightarrow\infty$ and then $\e\rightarrow 0$. It remains to show that $1- {\cal C}_s^{\e/2}(1-\e^2,\frac{2\e^2}{1-\e^2})$ vanishes. But this can be done as in Section \ref{s52} (Step 2). The proof of Theorem \ref{theorem1.6} is finished.

\section{Appendix}\label{S6}
\begin{proof}[Proof of Lemma \ref{lemma3.2}]
Fix $x\in B_{a_n}$ and take $\omega\in A_n$. We use Proposition 2.18 in \cite{BaDe10} to prove \eqref{lemma321}. This proposition states that there exists $c_4\in(0,\infty)$ such that for all $m_n\gg r_n$ for which $U_z\leq m_n/r_n$ for all $z\in B_{r_n}(x)$, $P_x(\eta( B_{r_n}(x)) \leq m_n) \leq e^{-c_4 r_n^2 m_n^{-1}}$, as desired in \eqref{lemma321}. Since we assume $m_n\gg r_n$, it remains to verify whether $U_z\leq m_n/r_n$ for all $z\in  B_{r_n}(x)$. But $B_{r_n}(x)\subseteq B_{2an}$ and so \eqref{3.7} implies that, $U_z\leq c_0(\log a_n)^3\leq m_n/r_n$ for all $z\in   B_{r_n}(x)$. This finishes the proof of \eqref{lemma321}.
The proof of \eqref{lemma322} is as the proof of Lemma 3.2 in \cite{Ce11}, where the claim is proved for $d=2$. 
\end{proof}

\begin{proof}[Proof of Lemma \ref{lemma3.3}] 
Let us first establish \eqref{lemma331}. We begin with the contribution to $\E E R_{m_n}^k$ coming from $y\notin B_{\sqrt{m_n} \log m_n}$.  By \eqref{lemma321} of Lemma \ref{lemma3.2} and \eqref{theorem312} of Theorem \ref{theorem3.1}, 
\bea\label{6.6}
&&\textstyle{\sum_{|y|\geq\sqrt{m_n} \log m_n}\E(P\bigl(\sigma(y)\leq m_n\bigr))}\nonumber\\
&\leq&\textstyle{  \sum_{|y|\geq\sqrt{m_n} \log m_n}(|y|^{d-1} e^{-c_4(\log |y|)^2}+|y|^{2d}\exp(-c_2|y|^{1/3})),}\quad \quad\label{largedist}
\eea
where we used that $P\bigl(\sigma(y)\leq m_n\bigr)\leq P\bigl(\sigma(y)\leq |y|^2\bigr)$ for $|y|\geq \sqrt{m_n} \log m_n$. Hence, the contribution to $\E E R_{m_n}^k$ coming from such $y$'s  tends to zero. Now, let $y\in B_{\sqrt{m_n} \log m_n}$. Since $|B_{\sqrt{m_n} \log m_n}|\P(A_n^c)\ll n^{-2} $, it suffices to bound $\E[ P(\sigma(y)\leq m_n)\1_{A_n}]$. We have that,
\bea\label{6.7}
 P(\sigma(y)\leq m_n) 
\hspace{-2mm}&\leq&\hspace{-2mm} P(\eta(B_{\sqrt{m_n} \log m_n})\leq m_n)+ P(\sigma(y)\leq \eta(B_{\sqrt{m_n} \log m_n})).
\eea
By \eqref{lemma321} of Lemma \ref{lemma3.2}, on $A_n$, the first probability in \eqref{6.7} is smaller than $e^{-c_4(\log m_n)^2}$. By the strong Markov property, 
\be\label{6.8}
P(\sigma(y)\leq \eta(B_{\sqrt{m_n} \log m_n})) = g_{B_{\sqrt{m_n} \log m_n}}(0,y) (g_{B_{\sqrt{m_n} \log m_n}}(y,y))^{-1},
\ee
where $g_{B}(x,z)=E_x\bigl(\int_0^{\eta(B)}\1_{J(t)=z} dt\bigr)$. Write  $D_1=B_{\sqrt{m_n}/2}$, and $D_2=B_{\sqrt{m_n} \log m_n}\setminus D_1$. We distinguish  whether $d\geq 3$ or $d=2$. Let $d\geq 3$ first. Take $y\in D_1$. By \eqref{theorem311} and \eqref{theorem313} of Theorem \ref{theorem3.1},  $g_{B_{\sqrt{m_n} \log m_n}}(0,y)\leq c_3(m_n^{-d/2+1}\wedge |y|^{2-d})$, and so
\be\label{6.9}
\textstyle{\sum_{y\in D_1}\E( P(\sigma(y)\leq m_n))^k\leq\sum_{y\in D_1}c_3(m_n^{-d/2+1}\wedge|y|^{2-d})^k\E \bigl(g_{B_{\sqrt{m_n} \log m_n}}(y,y)\bigr)^{-k}.}
\ee
Since $ B_{\sqrt{m_n}}(y)\subseteq  B_{\sqrt{m_n} \log m_n}$ and since the $\tau$'s are identically distributed, \eqref{6.9} is bounded by
\be\label{6.10}
\textstyle{\E g_{B_{\sqrt{m_n}}}^{-k}(0,0)\sum_{y\in D_1}c_{3}(m_n^{-d/2+1}\wedge|y|^{2-d})^k\leq c_6m_n^{1/k}\E g_{B_{\sqrt{m_n}}}^{-k}(0,0)\leq c_5 m_n^{1/k},}
\ee
where we  used \eqref{theorem312} and  \eqref{theorem314} of Theorem \ref{theorem3.1} to bound $\E g_{B_{\sqrt{m_n}}}^{-k}(0,0)\leq c$ for $c\in (0,\infty)$. Thus, the contribution to $\E E R_{m_n}^k$ coming from $y\in D_1$ satisfies \eqref{lemma331}. Now let $y\in D_2$. We bound $P(\sigma(y)\leq m_n)$ by
\bea\label{6.14}
\hspace{-2mm}&&\hspace{-2mm}\textstyle{\sum_{z:\ |z|=|y|/2}P( J(\eta( B_{|z|/2}))=z, \eta(B_{|y|/2})\leq m_n, \sigma(y)\leq m_n)}\nonumber\\
\hspace{-2mm}&\leq&\hspace{-2mm} \textstyle{\sum_{z:\ |z|=|y|/2}P_z(\sigma(y)\leq m_n)P( J(\eta(B_{|y|/2}))=z, \eta(B_{|y|/2})\leq m_n).}
\eea
As in \eqref{6.7} and \eqref{6.8}, $P_z(\sigma(y)\leq m_n)\leq 2 c_3|z-y|^{2-d}  (g_{B_{\sqrt{m_n}}(y)}(y,y))^{-1}$. For $c\in (0,\infty)$ large enough, $2c_3|z-y|^{2-d}\leq c|y|^{2-d}$, and so we get that
\bea\label{6.16}
P(\sigma(y)\leq m_n)\leq
\textstyle{\frac{c|y|^{2-d} }{g_{B_{\sqrt{m_n}}(y)}(y,y)}P(\eta(B_{|y|/2})\leq m_n)}\leq\textstyle{\frac{c|y|^{2-d} }{g_{B_{\sqrt{m_n}}(y)}(y,y)}\exp(-\tfrac{c_4}{4} |y|^2 m_n^{-1})},\ 
\eea
where we used \eqref{lemma321} of Lemma \ref{lemma3.2} in the last step. Using \eqref{6.16} and proceeding as in \eqref{6.10}, one sees that the contribution to $\E E R_{m_n}^k$  coming from $y\in D_2$ is as claimed in \eqref{lemma331}. The proof of \eqref{lemma331} is finished for $d\geq 3$. 

Let $d=2$. Recall that for $y\in D_1\cup D_2$ it is sufficient to bound the expected value on $A_n$. In fact, we will bound $P(\sigma(y)\leq \theta_n)$ for every $\omega\in A_n$ by a function $f_{m_n}(|y|)$ that is as in \eqref{lemma332} and this way also prove that the contribution coming from $y\in D_1\cup D_2$ satisfies \eqref{lemma331}. Let $y\in D_1$.  By \eqref{theorem313} of Theorem \ref{theorem3.1}, one can show that $g_{\sqrt{m_n} \log m_n}(0,y)\leq c_3(\log \sqrt{m_n} /|y|)$, and so
\be\label{6.12}
\textstyle{\sum_{y\in D_1}P(\sigma(y)\leq \theta_n)\leq \sum_{y\in D_1} c_3 g_{B_{\sqrt{m_n} \log m_n}}^{-1}(y,y)(\log \sqrt{m_n} /|y|).}
\ee
As in Lemma 3.3 in \cite{Ce11}, one sees that $g_{\sqrt{m_n} \log m_n}(y,y)\geq g_{B_{\sqrt{m_n}}(y)}(y,y) \geq c_7 \log m_n$. We set $f_{m_n}(|y|)\equiv 2c_3/c_7 (1-\log(|y|/\sqrt{m_n}))$. A simple calculations shows that hence the contribution coming from $y\in D_1$ is as claimed in \eqref{lemma331} and \eqref{lemma332}. Let $y\in D_2$. As in \eqref{6.14} - \eqref{6.16} we bound 
\be\label{6.17}
P_z(\sigma(y)\leq m_n)
\leq 
2 g_{B_{\sqrt{m_n} \log m_n}}(z,y)/g_{B_{\sqrt{m_n}}(y)}(y,y).
\ee
Since $|z-y|\geq \sqrt{m_n}/2$, one can check that $g_{B_{\sqrt{m_n} \log m_n}}(z,y)\leq c$ to get that
\be\label{6.18}
P(\sigma(y)\leq m_n)\leq  c\exp(-\tfrac{c_4}{4} |y|^2 m_n^{-1}) (c_7\log m_n)^{-1}\equiv f_{m_n}(|y|),
\ee
where we used that, $g_{B_{\sqrt{m_n}}(y)}(y,y)\geq c_7\log m_n$. The sum over $y\in D_2$ of $(f_{m_n}(|y|))^k$ satisfies \eqref{lemma331} and \eqref{lemma332} for $k=1,2,4$. The proof of \eqref{lemma331} is complete.

To finish the proof of Lemma \ref{lemma3.3}, it remains to prove \eqref{lemma332} for $y\in B_{m_n}\setminus (D_1\cup D_2)$. By \eqref{lemma322} of Lemma \ref{lemma3.2}, on $A_n$, we know that we may set $f_{m_n}(|y|)=e^{-c_4(\log |y|)^2}$ for $y\in B_{m_n}\setminus (D_1\cup D_2)$. Eq. \eqref{largedist} shows that the contribution to the sum in \eqref{lemma332}  coming from these $y$'s vanishes. This finishes the proof of Lemma \ref{lemma3.3}. 
\end{proof}
\bibliographystyle{abbrv}      
 
\bibliography{cup_ref}

 \end{document}